\documentclass[a4paper,12pt]{book}

\usepackage{amsthm}
\usepackage{amssymb,amsmath}

\newtheorem{thm}{Theorem}%[section]
\newtheorem{lem}{Lemma}%[section]
\newtheorem{prop}{Proposition}%[section]
\newtheorem{corol}{Corollary}%[section]
\newtheorem{thmabc}{Theorem}

\newtheorem{corolb}{Corollary}

\def\sym{\mathbb}
\def\ex#1{ \exp \{  L_{#1} (1) \} }
\def\er{e}
\def\dr{d}

\begin{document}
\begin{titlepage}
\begin{center}
\vspace{-3cm} \textbf{\large VILNIUS UNIVERSITY\\}
\vspace{4cm}
\textbf{\large Vytas Zacharovas\\} \vspace{3cm} \textbf{\LARGE
DISTRIBUTION OF RANDOM VARIABLES ON THE SYMMETRIC GROUP\\}
\vspace{4cm}
\textbf{\large Doctorial dissertation\\
Physical sciences, mathematics (01P)\\} \vspace{2cm} \textbf{\large
Vilnius 2004}
\end{center}
\end{titlepage}

{The thesis has been written in 2000 -- 2004 at Vilnius University.
\bigskip \\
Research adviser: \\
prof. habil. dr. Eugenijus Manstavi\v cius (Vilnius University, physical sciences, mathematics 01P).}
\bigskip

% Doctoral studies committee:
%
% chairman:
%
% \noindent prof. habil. dr. Bronius Grigelionis (Institute of Mathematics and Informatics, physical sciences, mathematics 01P);
%
% members:
% \begin{enumerate}
% \item
% prof. habil. dr. Art\=uras Dubickas (Vilnius University, physical sciences, mathematics 01P);
% \item
% prof. habil. dr. Antanas Laurin\v cikas (Vilnius University, physical sciences, mathematics 01P);
% \item
% prof. habil. dr. Eugenijus Manstavi\v cius (Vilnius University, physical sciences, mathematics 01P);
% \item
% prof. habil. dr. Leonas Saulis (Vilnius Gediminas Technical University, physical sciences, mathematics 01P);
% \end{enumerate}
%
% opponents:
% \begin{enumerate}
% \item
% prof. habil. dr. Mindaugas Bloznelis (Vilnius University, physical sciences, mathematics 01P);
% \item
% prof. habil. dr. Jonas  Sunklodas (Vilnius University, physical sciences, mathematics 01P).
% \end{enumerate}

%\begin{document}

\chapter*{Acknowledgments}
\addcontentsline{toc}{chapter}{Acknowledgments} I would like to
express my sincere gratitude to Professor E. Manstavi\v{c}ius for
his attention to this work as well as for his help and advices.

I would also like to thank the other members  of the doctoral
committee for their agreeing to participate in the defence of this
work.

%\mainmatter

\chapter*{Introduction}
\addcontentsline{toc}{chapter}{Introduction}
\section*{Actuality}
Distribution of many characteristics of different combinatorial
structures can be very well approximated by the distribution of
some variables defined on the symmetric group $S_n$. One example
of such approximation is the distribution of the degree of the
splitting field of a random polynomial with integer coefficients.
The degree of the splitting field of a polynomial is an important
characteristic which allows us to estimate how many steps it would
take to decompose the polynomial into the product of prime
polynomials. Thus the information about the distribution of the
degree of the splitting field of a random polynomial would allow
us to estimate how much time  would it take to decompose a
randomly chosen polynomial into the product of prime polynomials.
\section*{Aims and problems}
The aim of this work is to  obtain the estimates for the remainder
term in the Erd\H os Tur\'an law and also to prove analogous result
for distribution of the logarithm of the order of a random
permutation on some subsets $S_n^{(k)}$ of the symmetric group. To
obtain these aims we also prove some estimates for the mean values
of multiplicative functions on $S_n$ and $S_n^{(k)}$, which are of
independent interest, and which allow us to estimate the convergence
rate to normal law of additive functions on $S_n$.

We also study the distribution of the degree  of the splitting
field of a random polynomial and obtain sharp estimates for the
convergence rate of it to normal law.
\section*{Methods}
In research we apply both probabilistic and analytic methods. Some
analytic methods used here have their origins in the probabilistic
number theory, and some have their roots in the theory of
summation of divergent series. We also prove some tauberian
theorem for Voronoi summability of divergent series.
\section*{Novelty}
All obtained results are new. They generalize and improve the
earlier results of Erd\H os and Tur\'an, Manstavi\v cius, Nicolas,
Pavlov, Barbour and Tavar\'e.
\section*{Results}
Let $S_n$ be the symmetric group. Each $\sigma \in S_n$ can be
represented as a product of independent cycles
\begin{equation}
\label{sigma}
 \sigma=\kappa_1\kappa_2\ldots \kappa_{\omega(\sigma)}.
\end{equation}
This representation is unique up to the
order of cycles.

On $S_n$ we can define a uniform probability measure $P_n$ by
assigning to each subset $A \subset S_n$ probability
$$
P_n(A)=\frac{|A|}{|S_n|}=\frac{|A|}{n!},
$$
here $|A|$ is the number of elements of the set $A$.
Let us denote by  $\alpha_j(\sigma)$  the number of cycles in the decomposition of $\sigma$
into product of independent cycles (\ref{sigma}) whose length is equal to $j$. Then, obviously
\begin{equation}
\label{eq_alpha}
\alpha_1(\sigma)+2\alpha_2(\sigma)+\cdots+n\alpha_n(\sigma)=n.
\end{equation}
And $\omega(\sigma)$ -- the number of independent cycles in the decomposition of $\sigma$ in (\ref{sigma})
can be expressed as
\begin{equation*}
\omega(\sigma)=\alpha_1(\sigma)+\alpha_2(\sigma)+\cdots+\alpha_n(\sigma).
\end{equation*}
In 1942 V. L. Goncharov \cite{goncharov} proved that
\begin{equation*}
P_n\left( \frac{\omega(\sigma)-\log n}{\sqrt{\log n}}<x \right)\to \Phi(x),\quad \mbox{as}\quad n\to \infty
\end{equation*}
here $\Phi(x)$ is  the standard normal distribution. He also
proved that the distribution of $\alpha_j(\sigma)$ when $j$ is
fixed is asymptotically distributed as poissonian random variable
with parameter $1/j$.

The uniform probability measure is not the only probability
measure with respect to which one can study the distribution of
random variables on $S_n$. Let $\theta>0$ be  fixed parameter. By
putting
\begin{equation}
P_{n,\theta}(\sigma)=\frac{\theta^{\omega(\sigma)}}{\theta(\theta+1)\ldots(\theta+n-1)}
\end{equation}
we define the Ewense probability measure $P_{n,\theta}$.

As in \cite{manstberry} we will call a function $f:S_n \to \sym C$  multiplicative
if  for $\sigma \in S_n$
having representation (\ref{sigma}),
$$
f(\sigma)=f(\kappa_1)f(\kappa_2)\ldots f(\kappa_{\omega})
$$
and $f(\kappa)$ depends on cycle length $|\kappa|$ only.
For any multiplicative function $f$ on $S_n$ we will
denote by $\hat f(j)$ -- the value of $f$ on the cycles whose
length is equal to $j$, \ $1\leqslant j \leqslant n$.
 That means that $f(\kappa)=\hat f(|\kappa |)$,
 where $|\kappa |$ is the length of cycle $\kappa $.

Suppose $d(\sigma )$ is a non-negative multiplicative
function, $d(\sigma )\not \equiv 0$. Then we can define
 a probabilistic measure $\nu_{n,d}$ on $S_n$ by the formula

\begin{equation}
\label{meas}
\nu_{n,d}(\{\sigma\})={{d(\sigma )}\over {\sum_{w \in S_n} d(w )}}.
\end{equation}

If $\hat d(j)=1$, we obtain the
uniform probability measure, and for $\hat d(j)\equiv \theta >0$
we obtain   Ewens probability measure.

Let us denote by $M_n^d(f)$ the weighted mean of a
multiplicative function $f:S_n \to \sym C$ with respect to the
measure $\nu_{n,d}(\sigma)$:
$$
M_n^d(f)=\sum_{\sigma \in S_n}f(\sigma )\nu_{n,d}(\sigma )=
{{\sum_{\sigma \in S_n}f(\sigma )d(\sigma )}\over {\sum_{\sigma
\in S_n}d(\sigma )}}.
$$
In 2002  E. Manstavi\v cius proved the following result.

%------------------------------------------------------
%---------------Theorem 1------------------------------
%--------------------------------------------------------

\begin{thmabc}[\cite{manstdecom}]
\label{t_manstdecomp}
Let $f:S_n \to \sym C$ be a
multiplicative function, such that $|f(\sigma )|\leqslant 1$,
satisfying the conditions:
\begin{equation}
\label{bound}
\sum_{j\leqslant n}{{1 -\Re \hat f(j)}\over j}\leqslant D
\end{equation}

and
$$
{1\over n}\sum_{j=1}^n |\hat f (j) -1| \leqslant \mu_n =o(1),
$$
for some positive constant $D$ and some sequence $\mu_n$.

Suppose that the measure defining multiplicative
function $d(\sigma)$ satisfies the condition
 $0<d^-\leqslant d(\kappa )\leqslant d^+$ on cycles $\kappa \in S_n$  for some fixed positive
constants $d^-$ and $d^+$, then there exist  positive constants
$c_1=c_1(d^-,d^+)$ and $c_2=c_2(d^-,d^+)$  such that
$$
M_n^d(f)=\exp \left\{ \sum_{j\leqslant n}{d_j{{\hat f(j)-1}\over
j}} \right\} +O\left(  \mu_n^{c_1}+{1\over {n^{c_2}}} \right),
$$
where $d_j=\hat d(j)$ is the value of  $d$  on the cycles of
length $j$.
\end{thmabc}

We prove the following result.

%-------------------------------------------------------
%-----------Theorem 2 ----------------------------------
%-------------------------------------------------------

\begin{thm}
\label{meanf}
Let $f:S_n\to \sym C$ be a
multiplicative function satisfying the condition
$|f(\sigma)|\leqslant 1$ for all $\sigma \in S_n$. Suppose that
the measure defining multiplicative function
$d(\sigma)$ is such that $0<d^-\leqslant d_j \leqslant d^+$, where
$d_j$ is the value of $d(\sigma)$ on cycles of length $j$. Then we
have
\begin{multline*}
\Delta_n:=\left| M_n^d(f)- \exp \left\{ \sum_{k=1}^{n}d_k{{\hat
f(k)-1}\over {k}} \right\} \right|
\\
\leqslant c \left( {\biggl(\sum_{j=0}^np_j
\biggr)}^{-1}\sum_{k=1}^{n}|\hat f(k) -1|p_{n-k}+ {1\over {n^{d^-}
}} \sum_{k=1}^{n}|\hat f(k) -1|k^{d^- - 1}\right.
\\
\left. + {1\over
n}\sum_{k=1}^{n}|\hat f(k) -1| \right)
\end{multline*}
for $d^-<1$ and
%$$
%\left| M_n^d(f)- \exp \left\{ \sum_{k=1}^{n}{{d_k(\hat
%f(k)-1)}\over {k}} \right\} \right|
%$$
$$ \Delta_n\leqslant c \left( {\biggl(\sum_{j=0}^np_j
\biggr)}^{-1}\sum_{k=1}^{n}|\hat f(k) -1|p_{n-k}+ {1\over {n }}
\sum_{k=1}^{n}|\hat f(k) -1| \left( 1+\log {n\over k} \right)
\right)
$$
for $d^-\geqslant 1$,  where $c=c(d^-,d^+)$ is a positive constant
which depends on $d^-$ and $d^+$ only, and
$$
p_n={1\over {n!}}\sum_{\sigma \in S_n}d(\sigma)
$$
\end{thm}
Thus Theorem \ref{meanf} shows that  condition (\ref{bound}) in Theorem \ref{t_manstdecomp} is
superfluous. Our result yields more accurate estimates of the
remainder term than Theorem \ref{t_manstdecomp}. The proof of theorem \ref{meanf} is different from that of
Theorem \ref{t_manstdecomp} of Manstavi\v{c}ius. It is based on some properties of Voronoi means of divergent
series. In chapter 1 we establish a tauberian type theorem for  Voronoi summability which is of independent interest.

As in \cite{manstberry}, we will call a function $h:S_n\to \sym R$
additive if for $\sigma \in S_n$ having decomposition (\ref{sigma}) we have
$$
h(\sigma)=h(\kappa_1)+h(\kappa_2)+\cdots +h(\kappa_\omega)
$$
and we assume that the value of $h$ on cycles $\kappa$ depend on
the length of cycles only. This means that there exist $n$ real
numbers $\hat h(1), \hat h(2),\ldots , \hat h(n)$ such that
$h(\kappa)=\hat h(|\kappa |)$ for all cycles $\kappa \in S_n$.
%It is clear that the characteristic function of an additive function is a multiplicative function.

In 1998 E. Manstavi\v cius proved the following result.

%-------------------------------------------------------
%-------------Theorem 3 ------------------------------
%--------------------------------------------------------

%{\bf Theorem 3 (\cite{manstberry}).} {\it
\begin{thmabc}[\cite{manstberry}]
Let $d_j\equiv 1$. Suppose $\tilde
h_{n}(k)$ satisfy the condition
\begin{equation}
\label{normalization}
\sum_{k=1}^n{{\tilde h_{n}^2(k)} \over k}=1.
\end{equation}
Then
   $$
         \sup_{x\in{\sym R}}{\left|\nu_{n}(x)-\Phi(x)-
             {{D_nx}\over {2\sqrt{2\pi}}}e^{-x^2/2}\right|}\ll L_n,
   $$
where $\nu_n(x):=\nu_{n,1}( h(\sigma )-a(n)<x)$,
   $$
            D_n=\sum_{\scriptstyle 1\leqslant k,l \leq n \atop
                    \scriptstyle k+l>n} {{\tilde h_n(k) \tilde h_n(k)} \over kl},\quad
            L_n=\sum_{k=1}^{n}{{|\tilde h_n(k)|^3}\over k},\quad a(n)=\sum_{k=1}^n{{\tilde h_n(k)} \over k}.
   $$
\end{thmabc}

%{\bf Corollary 1(\cite{manstberry}).} {\it
\begin{corolb}[\cite{manstberry}]
We have
      $$
             R_n:=\sup_{x\in R}|\nu_{n,1}(x)-\Phi(x)|\ll L_n^{2/3}.
      $$
There exists a sequence of constants $\hat h_{n}(k)$, satisfying
the normalizing condition (\ref{normalization}), such that
 $L_n=o(1)$ and
      $$
                  R_n\gg L_n^{2/3}.
      $$
\end{corolb}
This result has been generalized for $\hat d(j)\equiv \theta >0$ in our paper \cite{zakclt}.

 We now generalize this result for the case when the probabilistic measure on the symmetric group is $\nu_{n,d}$.
 % endowed
 %with a positive multiplicative function $d(\sigma)$ by means of
 %formula (2).

 Let us denote

   $$
           A(n)=\sum_{k=1}^nd_k{{\tilde h_{n}(k)}\over k},\quad
                 C_n=\sum_{j=1}^nd_j{{\tilde h_{n}(j)}\over j}\left( {{p_{n-j}}\over {p_n}}-1
       \right),
   $$
and
   $$
       L_{n,p}=\sum_{k=1}^n{{|\tilde h_{n}(k)|^p}\over k},
       \quad L_{n,2}'=\sum_{j=1}^n{{\tilde h_{n}^2(j)}\over j}\left| {{p_{n-j}}\over
{p_{n}}} -1 \right|.
   $$
Henceforth assume that $\tilde h_{n}(k)$ satisfies the
normalizing condition

\begin{equation}
\label{dnormalize}
\sum_{k=1}^nd_k {{\tilde h_{n}^2(k)}\over k}=1.
\end{equation}

%-------------------------------------------------------------------------
%---------------THEOREM 4----------------------------------------------
%---------------------------------------------------------------

%{\bf Theorem 4.} {\it
\begin{thm}
\label{clt1}
Suppose $0<d^-\leqslant d_j \leqslant d^+$,
and $p$ is a fixed number  such that
 $\infty ~\geqslant~ p>\max \left\{ 2, 1/d^- \right\}$.
Suppose
   $$
   F_n(x)=\nu_{n,d} \left(
           h(\sigma )-A(n)<x  \right),
   $$
where $h(\sigma )$ is a additive  function satisfying
 condition (\ref{dnormalize}). Then  we have
  $$
      \sup_{x\in \sym R}\left| F_n(x)-\Phi (x)+{1\over {\sqrt{2\pi}}}e^{-x^2/2}
          C_n\right|\ll L_{n,3}+L_{n,p}^{2/p}+L_{n,2}',
  $$
here we assume that
$$
 L_{n,\infty }^{1/\infty}=\lim_{p\to \infty}
  L_{n,p}^{1/p}=\max_{1\leqslant j \leqslant n}|\hat h(j)|
 $$
 for $p=\infty$.
 \end{thm}

The condition $L_{n,3}=o(1)$ is not necessary to ensure the
convergence of an additive function to the normal law. In
\cite{manstbabubrown} it was presented the example of  an additive
function for which $L_{n,3}\gg 1$ and whose distribution
nevertheless converges to the normal law. In section
\ref{sec_example} we investigate the convergence rate for this
special additive function. Surprisingly  the rate is essentially
fasten than could be obtain from Theorem \ref{clt1}.

Let $a=(a_1,a_2,...,a_n)$ be a vector with nonnegative integer
components. Following \cite{barbour_tavare}, we denote
$$
O_n(a)=\hbox{ l. c. m. }\{i\colon\ 1\leqslant i\leqslant n,\ a_i>0\}
\quad \hbox{\rm and }\quad
P_n(a)=
\prod_{i=1}^ni^{a_i}.
$$
One can easily see that $O_n(a)\leqslant P_n(a)$.

%   Let
%  $\alpha_k(\sigma)$ be the number of cycles in $\sigma$ whose length is
%  equal to $k$. Then $\alpha_1(\sigma)+
% 2\alpha_2(\sigma)+\cdots +n\alpha_n(\sigma)=n$.
Let us denote $\alpha=\alpha(\sigma)=(\alpha_1(\sigma),\alpha_2(\sigma),\ldots,
\alpha_n(\sigma))$. Then  $O_n(\alpha(\sigma))$ is equal to the order of the permutation  $\sigma$, i. e.
$$
O_n(\alpha(\sigma))=\min \{ m\geqslant 0|\sigma^m=e \},
$$
where $e$ -- unit of the group $S_n$.

Erd\H os and Tur\' an~\cite{erdos_turan_III} proved
that if $P_n$ is the uniform probability measure  on $S_n$, then
$$
P_n\left(\frac{\log O(\alpha(\sigma))-\frac{1}{2}\log^2 n}{\frac{1}{\sqrt{3}}\log^{3/2} n} \right)\to \Phi(x)\quad \mbox{as}\quad n\to \infty
$$
the distribution of $\log O_n(\alpha)$  converges to the standard
normal law. Their proof consisted of two steps. First they proved
that the distributions of $\log O(\alpha(\sigma))$ and $\log
P(\alpha(\sigma))$ are close enough to allow to replace
investigation of $\log O(\alpha(\sigma))$ by that of $\log
P(\alpha(\sigma))$. And then they proved that $\log
P(\alpha(\sigma))$ in asymptotically normally distributed. $\log
P_n(\alpha(\sigma))=\sum_{j=1}^n\alpha_j(\sigma)\log j$ being an
additive function, its distribution is much easier to investigate.

 Nicolas \cite{nicolas_sym} obtained the estimate
${\rm O}\bigl(\log^{-1/2}n\log{\log{\log n}}\bigr)$ for the
convergence rate in their theorem. He has also conjectured that
the iterated logarithm in his estimate is superfluous. Barbour and
Tavar\'e \cite{barbour_tavare} proved that

\begin{multline}
\label{barbth}
\sup_{x\in {\sym R}} \left|    \nu_{n,\theta}
\left\{ {{\log  O_n(\alpha(\sigma))-{({\theta}/ 2)}\log^2 n+\theta \log n \log{\log n}}\over
{   ({\theta /{\sqrt 3})}\log^{3/2}n   } } <x   \right\}-\Phi(x) \right|
\\
\ll
{1 \over {\sqrt{\log n}}},
%\eqno(2)
\end{multline}
where $\nu_{n,\theta}$ is the Ewens probability measure.
% defined  on $S_n$ by the formula
% $$\nu_{n,\theta}(\sigma)={{ {\theta}^{k(\sigma)} }\over
% {\theta (\theta +1)\ldots (\theta +n-1) }},
% $$
% where $k(\sigma)$ denotes the number of cycles
% in the decomposition of $\sigma$ into the product of independent cycles.

In chapter 3 we prove the following theorems.
\begin{thm}
\label{cltaltha}
\begin{multline}
\label{cltalph}
\sup_{x\in {\sym R}} \left|    \nu_{n,1} \left\{ {{\log  O_n(\alpha)-
{\bf M}\log  O_n(\alpha )}\over
{   ({1/ {\sqrt 3})}\log^{3/2}n   } } <x
\right\}-\Phi(x)\right.
\left. -{{3^{3/2}}\over {24\sqrt{2\pi}}}
{{(1-x^2)\er^{-{x^2/2}}}\over \sqrt{\log n}} \right|
\\ \ll
{{\left( {{\log {\log n}}\over {\log n}} \right) }^{2/3}}.
\end{multline}

\end{thm}

\begin{thm}
\label{meanoalpha}
\begin{multline*}
{\bf M}\log  O_n(\alpha)
={1\over 2}\log^2n-\log n (\log \log n -1)
\\+\sum_{\rho}(\log n)^{\rho}\Gamma(-\rho)
+{\rm O}\big((\log \log n)^2\big),
\end{multline*}
where $\sum_{\rho}$ is the sum over all nontrivial zeros of the Riemann
zeta-function.
\end{thm}

Let $F_q$ be a finite field with $q$ elements.
We  denote by $E_n$ the set of normed polynomials of degree $n$ with
coefficients in $F_q$. Then each element
$P\in E_n$ can be decomposed into a product
of prime (over $F_q$) and normed  polynomials.
This decomposition is unique up to the order of multiplicands.
We denote by
  $\xi_k=\xi_k (P)$ the number of prime polynomials of degree
 $k$ in the canonical decomposition of a polynomial $P$.
On $E_n$, we define the uniform probability measure
$$
{\nu_n}(\ldots )={1\over{q^n}} |\{P\in E_n \colon\ \ldots \}|.
$$

Then $O_n(\xi)$, where
$\xi:=\xi(P)=(\xi_1(P),$ $\xi_2(P),\ldots ,\xi_n(P))$, is equal to the
degree of the splitting field of $P$. Nicolas~\cite{nicolas_poly} has proved that

\begin{equation}
\label{nicilth}
\sup_{x\in {\sym R}}
\left|{\nu_n}\left\{ {{\log  O_n(\xi)-{1\over 2}\log^2 n}\over
{{1\over {\sqrt 3}}\log^{3/2}n   } } <x   \right\}-\Phi(x) \right| \ll
{{(\log {\log n})^4}\over {\sqrt{\log n}}}.%\eqno(1)
\end{equation}

The investigation of  $\xi(P)$ on $E_n$ has much in common with that of the random
vector $\alpha(\sigma)=(\alpha_1(\sigma),\alpha_2(\sigma),\ldots
,\alpha_n(\sigma))$ defined on the symmetric group $S_n$  with the uniform probability measure. In fact theorems \ref{cltaltha} and \ref{meanoalpha} will be just simple consequences of the following theorems.

\begin{thm}
\label{cltxith}

\begin{multline}
\label{cltxi}
\sup_{x\in {\sym R}} \left|    {\nu_n} \left\{ {{\log O_n(\xi)-
{\bf M}\log O_n(\xi)}\over
{   ({1/{\sqrt 3})}\log^{3/2}n   } } <x   \right\}
-\Phi(x)
-{{3^{3/2}}\over {24\sqrt{2\pi}}}
{{(1-x^2)\er^{-{x^2/2}}}\over \sqrt{\log n}} \right| \\
\quad \ll{{\left( {{\log {\log n}}\over {\log n}} \right) }^{2/3}}.
 %\eqno(3)
\end{multline}

\end{thm}

\begin{thm}
\label{meanoxi}
\begin{multline*}
{\bf M}\log O_n(\xi)
={1\over 2}\log^2n-\log n (\log \log n -1)
\\
+\sum_{\rho}(\log n)^{\rho}\Gamma(-\rho)
+{\rm O}
\big((\log \log n)^2\big),
\end{multline*}
where $\sum_{\rho}$ denotes the sum over all nontrivial zeros of the
Riemann zeta-function.

\end{thm}

Assuming that the Riemann hypothesis is true, the sum over
the nontrivial zeros of the Riemann zeta-function in Theorems \ref{meanoxi} and \ref{meanoalpha}
can be estimated as ${\rm O}(\sqrt{\log n})$. Using the well-known fact that
$\zeta(\sigma +it)\not= 0$ for
 $\sigma \geqslant 1-{c\over {\log (2+|t|)}}$, one can estimate those
 sums over the nontrivial zeros of  $\zeta(s)$ as
${\rm O}(\log n\exp\{-c{{\log\log n}\over{\log\log\log n}}\})$.

   From theorems \ref{cltxith} and \ref{cltaltha} it follows that, for $\theta =1$,
the convergence rate in (\ref{barbth}) cannot be improved in the sense of order.

The proof of (\ref{barbth}) of  Barbour and Tavare is based on approximating
 the distribution of the random vector
$\alpha=(\alpha_1(\sigma),\alpha_2(\sigma),\ldots ,\alpha_n(\sigma))$
by the distribution of a random vector
 $Z=(Z_1,Z_2,\ldots ,Z_n)$, where $Z_j$ are
independent Poisson random variables with parameters
$1/j$. The distributions of $O_n(\alpha)$ and $O_n(Z)$
are approximated by the distributions of
 $O_n(Z)$ and  $P_n(Z)$, respectively.
In the proof of Theorem~ \ref{cltaltha} below, we directly
approximate the variable  $O_n(\xi (P))$ by $P_n(\xi (P))$,
without using any auxiliary independent random variables.

 Let us denote
by $S_n^{(k)}=\{ \sigma=x^k  | x\in S_n  \}$ the subset of $S_n$, which consists of all permutations $\sigma \in S_n$,  from which one can extract root of degree $k$.

On the subset $S_n^{(k)}$
we can define the uniform probability measure $\nu_n^{(k)}$ by means of formula
$$
\nu_n^{(k)}(A)=\frac{|A|}{|S_n^{(k)}|}
$$
for any $A\subset S_n^{(k)}$.

In Chapter \ref{ch_sub} we investigate the distribution additive and multiplicative functions
with respect to measure $\nu_n^{(k)}(\sigma)$, establishing the analog of Theorem \ref{meanf} for the mean values of multiplicative functions on $S_n^{(k)}$, thus generalizing earlier results of Pavlov \cite{pavlov}.

Pavlov \cite{pavlov} proved that $ {\frac{\log  O_n(\alpha(\sigma))-
{\bf M_n}\log  O_n(\alpha(\sigma ))}{  \sqrt{\frac{\phi(k)}{3k}}\log^{3/2}n   } }$
is asymptotically normally distributed on $S_n^{(k)}$ when $n\to \infty$. Here $\phi(k)$  -- Euler's function.
In chapter \ref{ch_sub} we estimate the convergence rate.

%{\bf Theorem 2} {\it
\begin{thm}
\label{clt}
For fixed  $k$ we have
\begin{multline*}
 \sup_{x\in {\sym R}} \biggl|    \nu_{n}^{(k)} \biggl\{ {\frac{\log  O_n(\alpha(\sigma))-
{\bf M_n}\log  O_n(\alpha(\sigma) )}{   \sqrt{\frac{\phi(k)}{3k}}\log^{3/2}n   } } <x   \biggr\}-\Phi(x)
-
r_k(x){\frac{e^{-{x^2/2}}}{ \sqrt{\log n}}} \biggr|
\\
\ll{{\left( {{\log {\log n}}\over {\log n}} \right) }^{2/3}},
\end{multline*}
where $r_k(x)=\sqrt{\frac{3}{\gamma_0}}
\frac{(1-8C_0-x^2)}{8\sqrt{2\pi }}$ and
$$
C_0=\gamma_0\int_0^1\frac{(1-y)^{\gamma_0-1}-1}{y}\,dy,
$$
here
$\gamma_0=\frac{\phi(k)}{k}$.
% }
\end{thm}
\begin{thm}
\label{riem}
\begin{multline*}
{\bf M_n }\log O(\alpha)=\frac{\gamma_0}{2}\log^2 n - \gamma_0\log n(\log \log n +C(k))
\\
+\sum_{\rho}\Gamma(-\rho)(\gamma_0\log n)^{\rho} +O((\log \log n)^2),
\end{multline*}
where
$$
C(k)=\log \gamma_0 -1-\int_0^1\frac{(1-y)^{\gamma_0-1}-1}{y}dy -\sum_{p|k}\frac{\log p}{p-1},
$$
and     $\sum_{\rho}$  is the sum over the non-trivial zeroes of the Riemann Zeta function.
\end{thm}

%\mainmatter

\chapter{Additive and multiplicative functions on $S_n$}

\section{Voronoi sums}

  Let
 $\alpha_k(\sigma)$ be the number of cycles in $\sigma$ whose length is
 equal to $k$. Then $\alpha_1(\sigma)+
2\alpha_2(\sigma)+\cdots +n\alpha_n(\sigma)=n$ and we have the
following representation for multiplicative function
$$
f(\sigma)=\hat f(1)^{\alpha_1 (\sigma)} \hat f (2)^{\alpha_2
(\sigma)} \ldots \hat f(n)^{\alpha_n(\sigma)},
$$
where as before $\hat f(j)$ are the values of
multiplicative function $f(\sigma )$ on cycles of length $j$. We
assume that $0^0=1$ in the above relationship.

 Since the quantity of $\sigma \in S_n$ such that
$\alpha_k(\sigma)=s_k$ for $1\leqslant k \leqslant n$ is equal to
$$
n!\prod_{j=1}^{n} {1 \over {s_j!j^{s_j}}},
$$ when
$s_1+2s_2+\cdots +ns_n=n$, we have for any
multiplicative function $f(\sigma )$
$$
\sum_{\sigma \in S_n}f(\sigma )= n!
 \sum_{k_1+2k_2+\cdots +nk_n=n} {\prod_{j=1}^{n}}
          {     {     \left(     {{ \hat f(j)}\over j}     \right)   }^{k_j}    }
           {1 \over {k_j!}}.
$$
One can easily see that if $ N_m={1\over {m!}}\sum_{\sigma \in
S_m}f(\sigma ) $ then

\begin{equation}
\label{genfunc}
\sum_{j=0}^{\infty}N_jz^j=\exp \left\{ \sum_{j=1}^{\infty} {{\hat
f(j) \over j}z^j}
         \right\}
\end{equation}

Here we assume that $N_0=1$. Therefore
$$
M_n^d(f)= {{\sum_{\sigma \in S_n}f(\sigma )d(\sigma )}\over
{\sum_{\sigma \in S_n}d(\sigma )}}={{M_n}\over {p_n}},
$$
where $M_j$ and $p_j$ are defined by relationships
$$
\sum_{j=0}^{\infty}M_jz^j=\exp \left\{
 \sum_{j=1}^{\infty}d_j {{{\hat
f(j)} \over j}z^j}
         \right\}\quad \hbox{and} \quad
        p(z) = \sum_{j=0}^{\infty}p_jz^j=\exp \left\{
 \sum_{j=1}^{\infty} {{d_j} \over j}z^j \right\}.
$$
We have
\begin{equation}
\label{genfunc2}
F(z)=\exp \left\{\sum_{j=1}^{\infty} d_j{{\hat f(j) \over j}z^j}
\right\} =\sum_{j=0}^{\infty}M_jz^j= {{\exp\{ L(z)
            \}}p(z)},     %  \eqno(7)
\end{equation}
here $L(z)=\sum_{j=1}^{\infty}d_j{{\hat f(j)-1}\over j}z^j$. Let us
denote $\exp\{ L(z) \} =\sum_{j=0}^{\infty}m_jz^j$, then
$$
M_n^d(f)={{M_n}\over {p_n}}={1\over {p_{n}}}\sum_{j=0}^n
m_jp_{n-j}.
$$

The proof of the theorem \ref{meanf} will be based on the following theorem.

%------------------------------------------------------
%------------------Theorem 5 --------------------------
%-------------------------------------------------------
\begin{thm}
\label{fundthm1}%{\bf Theorem 5.} {\it
Suppose $f(z)=\sum_{n=0}^{\infty}a_nz^n$,
when $|z|<1$, and $0<d^-<d_j<d^+$. Then for $n\geq 1$ we have
\begin{multline*}
\left|{1\over
{p_n}}\sum_{k=0}^na_kp_{n-k}-f(e^{-1/n})-{{S(f;n)}\over
{np_n}}\right|
\\
\leqslant c\left( {1\over {n^\theta}}\sum_{j=1}^n {{|S(f;j)|}\over
{p(e^{-1/j})}}j^{\theta -1}+ {1\over {p(e^{-1/n
})}}\sum_{j>n}{{|S(f;j)|}\over j}e^{-j/n}\right),
\end{multline*}
where $S(f;m)=\sum_{k=1}^ma_kkp_{m-k}$, \ $\theta =\min \{
d^-,1\}$ and $c=c(d^+,d^-)$ is a constant which depends on $d^+$
and $d^-$ only. %}
\end{thm}
Hence we obtain

%------------------------------------------------------
%-----------Theorem 6 ------------------------------------
%----------------------------------------------------

\begin{thm}
\label{voronmean}%{\bf Theorem 6.} {\it
Let $f(z)$ and $p(z)$ be the same as in theorem \ref{fundthm1}.
% for
%$|z|<1$, $a_n\in \sym C$ and
%$$
%p(z)=\exp \left\{ \sum_{k=1}^{\infty}{{d_k}\over k}z^k \right\}
%=\sum_{n=0}^{\infty}p_n z^n,
%$$
%where $0<\lambda^{-}\leqslant d_n \leqslant \lambda^+ <\infty$,
%for some fixed $\lambda^-$ and $\lambda^+$.
Then the  relation
$$
\lim_{n \to \infty}{1\over {p_n}}\sum_{k=0}^{n}a_kp_{n-k}=A\in
\sym C,
$$
holds if and only if the following two conditions are
satisfied:
$$
\leqno1)\quad \lim_{x\uparrow 1}f(x)= A
$$
$$
\leqno {2)}\quad S(f;n)=\sum_{k=1}^{n}a_kkp_{n-k}=o(np_n) \quad
\hbox{as}\quad n\to \infty .
$$
\end{thm}
Theorem \ref{voronmean} can be reformulated in terms of Voronoi summation theory
(see \cite{hardy}). Suppose $\sum_{j=0}^{\infty}a_k$ is a formal series and
$r_j$ is a sequence of non-negative real numbers. If
$$
\lim_{n\to \infty }{{r_0s_n+r_1s_{n-1}+\cdots+r_ns_0}\over
{r_0+r_1+\cdots +r_n}}=s\in \sym R,
$$
where $s_k=a_0+a_1+\cdots +a_k$, then we say that series
$\sum_{j=0}^{\infty}a_k$ can be summed in the sense of Voronoi and
its Voronoi sum is equal to $s$. In such case we write
$$
(W,r_n)\sum_{j=0}^{\infty}a_k=s.
$$
Theorem \ref{voronmean} yields the necessary and sufficient conditions for a
series to have a Voronoi sum when $r_n$ are defined by formula
$$
\sum_{m=0}^{\infty}r_mz^m=\exp
\left\{\sum_{j=1}^{\infty}{{\lambda_j}\over j}z^j \right\},
$$
where $-1<\lambda^-\leqslant\lambda_j\leqslant\lambda^+<\infty$.
For such $r_j$ we have that
$$
(W,r_n)\sum_{j=0}^{\infty}a_k=A
$$
if and only if
$$
\lim_{x\uparrow 1}\sum_{j=0}^{\infty}a_jx^j=A\in \sym C
$$
and
%The condition 2 of theorem 2 in such case might be represented in
% the following form
$$
{{r_0D_n+r_1D_{n-1}+\cdots +r_nD_0}\over {r_0+r_1+\cdots +r_n}}=o(n)\quad
\hbox{as}\quad n\to \infty,
$$
where $D_n=1a_1+2a_2+\cdots+na_n$. When $\lambda_j \equiv 0$, this condition takes form $D_n=o(n)$ and
we obtain the classical theorem of Tauber.

Note that in such case $r_n$ can be negative, though condition
$-1<\lambda^-\leqslant \lambda_j$ ensures that $r_0+r_1+\cdots
+r_m \geqslant {{m+\lambda^-} \choose {m}} >0$  for
$m\geqslant 1$.

 Let us now find the generating function of the characteristic
function of the distribution of $h_n(\sigma)$.
 Since for additive function we have
 $h(\sigma)=\hat h(1)\alpha_1(\sigma) +\hat h (2) \alpha_2 (\sigma)+\cdots +\hat h(n)\alpha_n (\sigma)$
 and
$$
\nu_{n,d}(\alpha_1(\sigma)=s_1,\ldots,\alpha_n(\sigma)=s_n)=
    {{1}\over  {p_n}}
       \prod_{j=1}^{n} {\left({d_j \over
       j}\right)}^{s_j}{1\over{s_j!}},
$$
therefore
   $$
       g_n(t)=\sum_{\sigma \in S_n}
           { \exp \left\{ ith_n(\sigma) \right\} }
           \nu_{n,d}(\sigma)
           ={1\over {p_n}}
         \sum_{m_1+2m_2+\cdots +nm_n=n} {\prod_{j=1}^{n}}
          {     {     \left(     d_j{{ \hat f(j)}\over j}     \right)   }^{m_j}    }
           {1 \over {m_j!}},
       $$
where  $\hat f(k)=\exp\{it\tilde h_{n}(k)\}$. Thus the
characteristic function of the random variable $\tilde h_{n}(k)$
is the weighted mean of a multiplicative function.

%-------------------------------------------------------
%-------------------------------------------------------
%-------------PROOFS-------------------------
%-------------------------------------------------------
%-------------------------------------------------------

%\input{ex.tex}

%\section{Proofs}

%-------------------------------------------------------------

Let $p(z)$ be denoted as before.
%$$
%p(z)=\exp \left\{ \sum_{k=1}^{\infty}{{d_k}\over {k}}z^k  \right\}
%=\sum_{n=0}^{\infty}p_nz^n.
%$$
In what follows we assume that $0<d^-\leqslant d_k\leqslant d^+$,
and $\theta =\min \{ 1,d^- \}$, where $d^-, d^+$ are fixed
positive numbers. We also denote $\tilde d_k=d_k-\theta$ and
$$
\tilde p(z)=\exp \left\{ \sum_{k=1}^{\infty}{{\tilde d_k}\over {k}}z^k  \right\}
=\sum_{n=0}^{\infty}\tilde p_nz^n.
$$
One can easily see that
$$
\tilde p(z)={{p(z)}{(1-z)^{\theta}}}.
$$

%-------------------------------------------------------
%-------------Lemma 1-----------------------------------
%-------------------------------------------------------
%{\bf Lemma 1.} {\it
\begin{lem}
\label{pvbound}
If $m\geqslant n\geqslant 1$, then
$$
{\left( {m\over n}  \right) }^{d^-} e^{-{{d^-}/ {n}}}\leqslant
{{p( e^{-{1/ {m}}})}\over {p( e^{-{1/ n} })}}\leqslant {\left(
{m\over n} \right) }^{d^+} e^{{{d^+}/ {m}}},
$$
and
$$
{\left( {m\over n}  \right) }^{\tilde d^-} e^{-{{\tilde d^-}/
{n}}} \leqslant {{\tilde p( e^{-{1/ {m}}})}\over {\tilde p(
e^{-{1/ n} })}}\leqslant {\left( {m\over n}  \right) }^{\tilde
d^+} e^{{{\tilde d^+}/ {m}}},
$$
where $\tilde d_k^+=d_k^+ -\theta$, \ $\tilde d_k^-=d_k^- -\theta$
and $\tilde d^+=d^+-\theta$.
\end{lem}

\begin{proof}
%{\bf Proof.}
We have
\begin{equation*}
\begin{split}
{{p( e^{-{1/ {m}}})}\over
{p( e^{-{1/ n} })}}&=\exp \left\{ \sum_{k=1}^{\infty}{{d_k}\over k}
( e^{-{k/ {m}}} -e^{-{k/ n} }) \right\}
\leqslant \exp \left\{ d^+\sum_{k=1}^{\infty}
{{ e^{-{k/ {m}}} -e^{-{k/ n} }}\over k} \right\}
\\
&=\exp \left\{ d^+ \log {{1- e^{-{1/ n} }}\over{1- e^{-{1/ {m}}}}}
\right\}={\left( {{1- e^{-{1/ n} }}\over{1- e^{-{1/ {m}}}}}
\right)}^{d^+} \leqslant {\left( {m\over n}  \right) }^{d^+}
e^{{{d^+}/ {m}}}.
\end{split}
\end{equation*}
here we have used the inequalities $e^{-x}x\leqslant
1-e^{-x}\leqslant x$ for $x\geqslant 0$.

In the same way we obtain the lower bound estimate.

The proof of the second inequality is analogous.

The lemma is proved.
\end{proof}

%--------------------------END----------------------------------
Further we will often use the inequality
$$
a_0+a_1+\cdots +a_n \leqslant eg(e^{-1/n}),
$$
where $g(x)=\sum_{j=0}^{\infty}a_kx^k$ and $a_k\geqslant 0$,
$k\geqslant 0$.
Differentiating $p(z)$ and $\tilde p(z)$ we obtain
$$
zp'(z)=p(z)\sum_{k=1}^\infty d_kz^k \quad \hbox{and} \quad z\tilde
p'(z)=\tilde p(z)\sum_{k=1}^\infty \tilde d_kz^k,
$$
hence we have for $n\geqslant 1$
\begin{equation}
\label{p_n}
p_n={1\over n}\sum_{k=1}^nd_kp_{n-k}\quad \hbox{and} \quad \tilde
p_n={1\over n}\sum_{k=1}^n\tilde d_k{\tilde p_{n-k}}.%\eqno (8)
\end{equation}
Hence we have
\begin{equation}
\label{p_n2}
p_n\leqslant {{d^+e}\over n} p(e^{-1/n}) \quad \hbox{and} \quad
\tilde p_n\leqslant {{d^+e}\over n} \tilde p(e^{-1/n}).%\eqno(9)
\end{equation}

It has been proved in \cite{manstdecom} that there is a positive constant
$c(d^+)$ such that
\begin{equation}
\label{plwbound}
p_n\geqslant {d^-c(d^+)}{{p(e^{-1/n})}\over n}.%\eqno(10)
\end{equation}
For the sake of
completeness we will give here another proof of this estimate
based on the following theorem which is of interest in itself.

%-----------------------------------------------------------------
%---------------Theorem 7----------------------------------------
%----------------------------------------------------------------
%{\bf Theorem 7.} {\it
\begin{thm}
\label{lowcoefth}
Suppose $f(x)=\sum_{k=0}^{\infty}a_kx^k$,
where $a_k\geqslant 0$ and
$$
{{f'(x)}\over{f(x)}}\leqslant {c\over {1-x}}
$$
when $0\leqslant x <1$. Then there exists such a positive constant
$K=K(c)$ that
$$
\sum_{j=0}^Na_k\geqslant K(c) f(e^{-1/N})
$$
when $N\geqslant 2c$.
\end{thm}

%{\it Proof.}
\begin{proof}
 If $0\leqslant x < 1$, we have
\begin{eqnarray*}
f(x)&\leqslant& \sum_{k=0}^Na_kx^k+{1\over N}\sum_{k=0}^{\infty}ka_kx^k\leqslant
\sum_{k=0}^Na_k +{{xf'(x)}\over {N}}
\\
&\leqslant& \sum_{k=0}^Na_k +f(x){x \over N}{{f'(x)}\over {f(x)}}
\leqslant \sum_{k=0}^Na_k +f(x){{cx}\over {N(1-x)}}.
\end{eqnarray*}
Inserting here $x=e^{-1/n}$ with $n=\left[{N\over{2c}}\right]$, we obtain
$$
f(e^{-1/n})\leqslant  \sum_{k=0}^Na_k+{{cn}\over N}f(e^{-1/n})\leqslant
 \sum_{k=0}^Na_k+{1\over 2}f(e^{-1/n}),
$$
therefore
$$
{1\over 2}f(e^{-1/n})\leqslant  \sum_{k=0}^Na_k.
$$
 If $c\leqslant 1/2$, then $N\leqslant [\frac{N}{2c}]=n$ and
$$
{1\over 2}f(e^{-1/N})\leqslant {1\over 2}f(e^{-1/n})\leqslant  \sum_{k=0}^Na_k,
$$
therefore in such a case the theorem will be true with $K(c)={1\over 2}$.

Suppose now that $c>{1\over 2}$, then  $N\geqslant [\frac{N}{2c}]=n$ and we have

\begin{equation*}
\begin{split}
{{f(e^{-1/n})}\over{f(e^{-1/N})}}&= \exp \left\{ \log f(e^{-1/n})
-\log f(e^{-1/N})  \right\}= \exp \left\{
-\int_{e^{-1/n}}^{e^{-1/N}}{{f'(x)}\over{f(x)}}\,dx  \right\}
\\
&\geqslant \exp \left\{ -c\int_{e^{-1/n}}^{e^{-1/N}}{{dx}\over{(1-x)}}  \right\}
={\left( {{1-e^{-1/N}}\over{1-e^{-1/n}}} \right)}^c
\geqslant e^{-c/N}{\left( {n\over N}\right)}^c\\
&=
e^{-c/N}{\left( {1\over N} {\left[ {N\over {2c}} \right] }\right)}^c
\geqslant
K(c)>0,
\end{split}
\end{equation*}
where $K(c)=\inf_{m\geqslant 2c}
e^{-c/m}{\left( {1\over m} {\left[ {m\over {2c}} \right] }\right)}^c$.

The theorem is proved.
\end{proof}
%--------------------END-----------------------------------------
The application of this theorem for $f(z)=p(z)$ together with (\ref{p_n})
yields the proof of estimate (\ref{plwbound}).
%--------------------------------------------------------------
%-----------------------Lemma 2----------------------------------
%----------------------------------------------------------------

\begin{lem}
\label{diffpnm}%{\bf Lemma 2.} {\it
If  $0 \leqslant s \leqslant n/2$, then
$$
|p_{n+s}-p_{n}|\ll {{p(e^{-1/n})}\over n} \left( {s\over
n}\right)^{\theta} \ll  p_n \left( {s\over n}\right)^{\theta},
$$
where $\theta =\min\{ d^-,1 \}$.
\end{lem}
\begin{proof}
 Since $p(z)={{\tilde p(z)}\over {(1-z)^{\theta}}}$,
then we have
$$
p_n=\sum_{k=0}^n\tilde p_k {{n-k+\theta -1}\choose{n-k}} ,
$$
therefore
\begin{eqnarray*}
p_{n+s}-p_n&=&\sum_{k=0}^n\tilde p_k
\left( {{n+s-k+\theta -1}\choose{n+s-k}}
-{{n-k+\theta -1}\choose{n-k}} \right)
\\
&&\mbox{}
+\sum_{n+s\geqslant k >n}\tilde
p_k{{n+s-k+\theta -1}\choose{n+s-k}}=:S_1+S_2.
\end{eqnarray*}
If $s=0$, then the estimate of the theorem is trivial, therefore
we assume that $s>0$. Applying here the estimate (\ref{p_n2}) together with
Lemma \ref{pvbound} we have
\begin{equation*}
\begin{split}
S_2&\leqslant \sum_{l=0}^s {{l+\theta -1}\choose{l}}
\max_{n+s\geqslant k>n} \tilde
p_k \leqslant {{s+\theta }\choose{s}} \tilde
p(e^{-1/n})ed^+ \max_{n+s\geqslant k>n} {{\tilde p(e^{-1/k})}\over
{\tilde p(e^{-1/n})}} {1\over k}
\\
&\ll s^\theta {{\tilde p(e^{-1/n})}\over n}=s^\theta {{
p(e^{-1/n})(1-e^{-1/n})^\theta }\over n}\leqslant \left( {s\over
n}\right)^\theta  {{p(e^{-1/n})}\over n}.
\end{split}
\end{equation*}
If $\theta =1$, then $S_1=0$, therefore estimating $S_1$ we may
assume that $\theta <1$.

It is well-known that ${{n-k+\theta -1}\choose{n-k}}={{n^{\theta -1}}\over {\Gamma
(\theta)}}\left( 1+O\left( {1\over n}\right) \right)$ (see e. g.
\cite{flajolet}).
 Once again applying Lemma \ref{pvbound} and the estimate (\ref{p_n2}) we have

\begin{eqnarray*}
S_1&\ll& \sum_{k=0}^{n-s}\tilde p_k{{n-k+\theta -1}\choose{n-k}}
\left|{{{n+s-k+\theta -1}\choose{n+s-k}}\over{{n-k+\theta -1}\choose{n-k}}}-1 \right|
\\
&&\mbox{}+
\sum_{n-s<k\leqslant n}\tilde p_k(s^{\theta-1}+(n-k+1)^{\theta -1}
)
\\
&\ll& \sum_{k=0}^{n-s}\tilde p_k {{n-k+\theta -1}\choose{n-k}} {s\over{n-k}} +
{{\tilde p(e^{-1/n})}\over n}s^\theta
\\
&\ll & \sum_{k\leqslant n/2}\tilde p_k{s\over n}n^{\theta -1}
+\sum_{n/2<k\leqslant n-s}\tilde p_k s(n-k)^{\theta-2}
+ {{p(e^{-1/n})}\over n}\left( {s\over n} \right)^\theta
\\
&\ll &  \tilde p(e^{-1/n}){s\over n}n^{\theta-1}+
 {{\tilde p(e^{-1/n})}\over n}s\sum_{l\geqslant s}l^{\theta -2}
+{{p(e^{-1/n})}\over n}\left( {s\over n} \right)^\theta
\\
&\ll & {{p(e^{-1/n})}\over n}{s\over n}+ {{p(e^{-1/n})}\over n}\left(
{s\over n} \right)^\theta \ll {{p(e^{-1/n})}\over n}\left( {s\over
n} \right)^\theta.
\end{eqnarray*}

The lemma is proved.
\end{proof}
%--------------------END---------------------------------------

For $0\leqslant x \leqslant 1$ we denote
$$
G_x(z)={{p(z)}\over {p(zx)}}=\sum_{k=0}^\infty g_{k,x}z^k \quad
\hbox{and} \quad \tilde G_x(z)={{\tilde p(z)}\over {\tilde
p(zx)}}=\sum_{k=0}^\infty \tilde g_{k,x}z^k,
$$
and

$$
C_x(z)= \left( {{1-zx}\over {1-z}}
\right)^\theta=\sum_{k=0}^\infty c_{k,x}z^k.
$$

Since $\tilde p(z)={{p(z)}{(1-z)^{\theta}}}$, we have
$$
G_x(z)=\tilde G_x(z) \left( {{1-zx}\over {1-z}} \right)^\theta .
$$
Differentiating $C_x(z)$ and $G_x(z)$ with respect to $z$ we have
$$
zC_x'(z)=C_x(z)\theta \sum_{k=1}^\infty z^k(1-x^k)\quad \hbox{and}
\quad zG_x'(z)=G_x(z)\sum_{k=1}^\infty d_k z^k(1-x^k).
$$
Whence we obtain that
$$
c_{n,x}={{\theta } \over n} \sum_{k=1}^n c_{n-k,x}(1-x^k)\quad
\hbox{and} \quad g_{n,x}={1\over n}\sum_{k=1}^n
g_{n-k,x}d_k(1-x^k),
$$
for $n\geqslant 1$ and $c_{0,x}=g_{0,x}=1$. Hence we deduce that
$c_{n,x},g_{n,x}\geqslant 0$ and therefore
$$
\sum_{m=0}^nc_{m,x}\leqslant e C_x(e^{-1/n}) \quad \hbox{and}\quad
\sum_{m=0}^ng_{m,x}\leqslant e G_x(e^{-1/n}).
$$
It follows hence
$$
c_{n,x}\leqslant {{e\theta C_x(e^{-1/n}) }\over n} \quad
\hbox{and}\quad g_{n,x}\leqslant {{ed^+ G_x(e^{-1/n}) }\over n}.
$$

%--------------------------------------------------------------
%------------------Lemma 3-------------------------------------
%--------------------------------------------------------------
%{\bf Lemma 3.} {\it
\begin{lem}
\label{diffcmn}
Suppose $0<x<1$ and $s\leqslant
m/2$, then we have
$$
%\left[
%\left(
%{{1-zx}\over {1-z}}
%\right)^\theta
%\right]_{(m)}
%-
%\left[
%\left(
%{{1-zx}\over {1-z}}
%\right)^\theta
%\right]_{(m-s)}
|c_{m,x}-c_{m-s,x}|\ll sm^{\theta-2}(1-x)^\theta +{s \over {m^2}},
$$
for $m\geqslant 1$.
\end{lem}

\begin{proof}
Suppose  $L_{\epsilon}$ is a contour
$
L_\epsilon=L_1\cup L_2\cup L_3 \cup L_4
$, where
$$
L_1=\{ z| z={2e^{it}}, \pi \geqslant |t|\geqslant \epsilon \},\quad
L_2=\left\{ z| z=1+{e^{it}\over m}, \pi \geqslant |t|\geqslant
\epsilon \right\},
$$
$$
L_3=\left\{ z| z=t\left(1+{e^{i\epsilon }\over m}\right)
+(1-t)2e^{i\epsilon } ,\quad 0\leqslant t\leqslant 1\right\},
$$
$$
L_4=\left\{ z| z=t\left(1+{e^{-i\epsilon }\over m}\right)
+(1-t)2e^{-i\epsilon } ,\quad 0\leqslant t\leqslant 1\right\}.
$$
 Applying Cauchy formula we have
\begin{equation*}
\begin{split}
|c_{m,x}-c_{m-1,x}|&=
\left| {1\over {2\pi i}}\int_{L_\epsilon}
C_x(z)
{{(1-z)}\over{z^{m+1}}}\,dz \right|
\leqslant
 {1\over {2\pi}}\int_{L_\epsilon}
{{|1-xz|^\theta|1-z|^{1-\theta}}\over{|z|^{m+1}}}\,|dz|
\\
&\leqslant
 {1\over {2\pi}}\int_{L_\epsilon}
{{(|1-z|+|z||1-x|)^\theta|1-z|^{1-\theta}}\over{|z|^{m+1}}}\,|dz|
\\
&\ll
 \int_{L_\epsilon}
{{|1-z|+|1-z|^{1-\theta}|1-x|^\theta}\over{|z|^{m+1}}}\,|dz|
\end{split}
\end{equation*}

Allowing now $\epsilon \to 0$, we have

\begin{equation*}
\begin{split}
|c_{m,x}-c_{m-1,x}|&\ll \frac{1}{2^m}
+\int_{1+{1\over m}}^2 {{(y-1)+(y-1)^{1-\theta}|1-x|^\theta}\over{y^{m+1}}}\,dy
\\
&\quad+\int_{|z-1|={1\over m}} {{|1-z|+|1-z|^{1-\theta}|1-x|^\theta}\over{|z|^{m+1}}}\,|dz|
\\
&\ll
%\frac{1}{m^2}+
\frac{(1-x)^\theta}{m}\int_{m\log \left( 1+{1\over m}\right) }^{2m}
\frac{(e^{u/m}-1)^{1-\theta}}{e^u}
\,du +\frac{1}{m^2}+m^{\theta -2}(1-x)^\theta
\\
&\ll(1-x)^\theta m^{\theta -2}\int_{m\log \left( 1+{1\over m}\right) }^{2m}
{u^{\theta}}{e^{-u}}
\,du +\frac{1}{m^2}+m^{\theta -2}(1-x)^\theta
\\
&\ll
 m^{\theta-2}(1-x)^\theta +{1 \over {m^2}}.
 \end{split}
\end{equation*}
Now we have for $s\leqslant m/2$
\begin{equation*}
\begin{split}
|c_{m,x}-c_{m-s,x}|&\leqslant |c_{m,x}-c_{m-1,x}|+|c_{m-1,x}-c_{m-2,x}|+\cdots
+|c_{m-s+1,x}-c_{m-s,x}|
\\
&\ll sm^{\theta-2}(1-x)^\theta +{s \over {m^2}}.
\end{split}
\end{equation*}
The lemma is proved.
\end{proof}

%----------------------END------------------------------------------

%-----------------------------------------------------------------
%--------------------Lemma 4-------------------------------------
%------------------------------------------------------------------

%{\bf Lemma 4.} {\it
\begin{lem}
\label{diffgnm}
For $0\leqslant x \leqslant e^{-1/n}$ and
$k\leqslant n/8$, we have
$$
g_{n,x} -g_{n-k,x} \ll {{p(e^{-1/n})}\over {np(x)}}\left( \left(
{k\over n} \right)^\theta +{1\over {(n(1-x))^\theta}} \right).
$$
\end{lem}
\begin{proof}
Since $G_x(z)=\tilde G_x(z) C_x(z)$, we obtain
\begin{equation*}
\begin{split}
g_{n,x}-g_{n-k,x}&=
\sum_{s=0}^n \tilde g_{s,x}c_{n-s,x}
-\sum_{s=0}^{n-k} \tilde g_{s,x}c_{n-k-s,x}
=\sum_{s=0}^{n-2k}\tilde g_{s,x}(c_{n-s,x}-c_{n-k-s,x})
\\
&\quad+\sum_{n-2k<s \leqslant n} \tilde g_{s,x}c_{n-s,x}
-\sum_{n-2k<s\leqslant n-k} \tilde g_{s,x}
c_{n-k-s,x}=:S_1+S_2+S_3.
\end{split}
\end{equation*}

If $\theta=1$ then $c_j=1-x$, for $j\geqslant 1$, it follows hence that in this case $S_1=0$.
Therefore, while estimating $S_1$, we may assume that $\theta<1$.
 Applying Lemma \ref{diffcmn} we have
\begin{eqnarray*}
S_1&\ll&
\sum_{0\leqslant s \leqslant n-2k}
\tilde g_{s,x}\left(
 k(n-s)^{\theta-2}(1-x)^\theta +{k \over {(n-s)^2}}
     \right)
\\
&\ll&  kn^{\theta-2}(1-x)^\theta
\sum_{0\leqslant s \leqslant n/2} \tilde g_{s,x}
\\
&&\mbox{}+\sum_{n/2 \leqslant s \leqslant n-2k} \tilde g_{s}\left(
 k(n-s)^{\theta-2}(1-x)^\theta +{k \over {(n-s)^2}}
     \right)
\\
&\ll&  kn^{\theta-2}(1-x)^\theta \tilde G_x(e^{-1/n})
\\
&&\mbox{}+k{{\tilde
G_x(e^{-1/n})}\over n} \sum_{n/2 \leqslant s \leqslant n-2k}
\left(
 (n-s)^{\theta-2}(1-x)^\theta +{1 \over {(n-s)^2}}
     \right)
\\
&\ll&  kn^{\theta-2}(1-x)^\theta  G_x(e^{-1/n}) \left(
{{1-e^{-1/n}}\over {1-xe^{-1/n}}} \right)^\theta
\\
&&\mbox{}+k
{{G_x(e^{-1/n})}\over n} \left( {{1-e^{-1/n}}\over {1-xe^{-1/n}}}
\right)^\theta \left(
 k^{\theta-1}(1-x)^\theta +{1 \over k}
     \right).
\end{eqnarray*}
Since   $1-xe^{-1/n}\geqslant 1-x$, we have
$$
S_1\ll {k\over {n^2}}{{p(e^{-1/n})}\over {p(xe^{-1/n})}} +{1\over
n}{{p(e^{-1/n})}\over {p(xe^{-1/n})}} \left( \left( {k\over n}
\right)^\theta +{1\over {(n(1-x))^\theta}} \right).
$$
In a similar way we obtain
\begin{equation*}
\begin{split}
S_2+S_3&\ll {1\over n} \tilde G_x(e^{-1/n})\sum_{0\leqslant l \leqslant 2k}c_{x,l}
\\
&\ll {1\over n}{{p(e^{-1/n})}\over {p(xe^{-1/n})}}
 \left(
{{1-e^{-1/n}}\over {1-xe^{-1/n}}}
\right)^\theta    \left(
{{1-xe^{-1/2k}}\over {1-e^{-1/2k}}}
\right)^\theta
\\
&\ll {1\over n}{{p(e^{-1/n})}\over {p(xe^{-1/n})}}
 \left(
{{1-x +x(1-e^{-1/2k})}\over {1-x}}
\right)^\theta   \left( {k\over n} \right)^\theta
\\
&\ll {1\over n} {{p(e^{-1/n})}\over {p(xe^{-1/n})}}
 \left(
1+{1\over {k(1-x)}}
\right)^\theta  \left( {k\over n} \right)^\theta
\\
&\ll {1\over n} {{p(e^{-1/n})}\over {p(xe^{-1/n})}}
\left( \left( {k\over n}
\right)^\theta +{1\over {(n(1-x))^\theta}} \right).
\end{split}
\end{equation*}
Since $p(xe^{-1/n})\gg p(x)$ if $0\leqslant x \leqslant e^{-1/n}$,
 the proof of the lemma follows.
\end{proof}
%---------------------END---------------------------------------

%----------------------------------------------------------------
%-------------Lemma 5--------------------------------------------
%----------------------------------------------------------------
%{\bf Lemma 5.} {\it
\begin{lem}
\label{intbounds}
Suppose $u(x)=\exp \left\{
\sum_{k=1}^{\infty}{u_k\over k}x^k
  \right\}$ and $0\leqslant u_k \leqslant A$, then the following estimates hold:
$$
\leqno1)\quad\int_0^1{{x^{j-1}}\over {u(x)}}\,dx \ll {1\over
{ju(e^{-1/j})}} \quad \hbox{if} \quad j\geqslant 1;
$$
$$
\leqno2)\quad\int_0^{e^{-1/n}}{{x^{j-1}}\over {u(x)}}\,dx \ll
{{e^{-j/n}}\over {ju(e^{-1/n})}} \quad \hbox{if} \quad j\geqslant
n.
$$
\end{lem}
\begin{proof} 1) For $j\geqslant 1$ we have

\begin{equation*}
\begin{split}
\int_0^1{{x^{j-1}}\over {u(x)}}\,dx
&=\int_0^{e^{-1/j}}{{x^{j-1}}\over {u(x)}}\,dx
+\int_{e^{-1/j}}^1{{x^{j-1}}\over {u(x)}}\,dx
\\
&\leqslant {1\over
{u(e^{-1/j})}}\int_0^{e^{-1/j}}{{u(e^{-1/j})}\over {u(x)}}
x^{j-1}\,dx
+{{1-e^{-1/j}}\over {u(e^{-1/j})}}
\\
&\leqslant {1\over
{u(e^{-1/j})}}\int_0^{e^{-1/j}} x^{j-1}\exp \left\{
\sum_{k=1}^{\infty}{u_k\over k}(e^{-k/j}-x^k)
  \right\}\,dx+ {1\over {ju(e^{-1/j})}}
\\
&\leqslant {1\over {u(e^{-1/j})}}\int_0^{e^{-1/j}} x^{j-1} \left(
{{1-x}\over {1-e^{-1/j}}} \right)^A   \,dx+ {1\over
{ju(e^{-1/j})}}
\\
&\leqslant {{e^{A/j}j^A}\over {u(e^{-1/j})}}\int_0^{e^{-1/j}}
x^{j-1}(1-x)^A\,   dx+ {1\over {ju(e^{-1/j})}}
\\
&= {{e^{A/j}j^A}\over
{u(e^{-1/j})}}\int_0^{\infty } (1-e^{-y})^A e^{-jy}\,   dx
+ {1\over {ju(e^{-1/j})}}
\\
&\leqslant {{e^{A/j}j^A}\over
{u(e^{-1/j})}}\int_0^{e^{-1/j}} y^A e^{-jy}\,   dx + {1\over
{ju(e^{-1/j})}} = {{e^{A/j} \Gamma (A+1)+1}\over {ju(e^{-1/j})}},
\end{split}
\end{equation*}

here we have used the inequalities $e^{-y}y\leqslant
1-e^{-y}\leqslant y$, for $y\geqslant 0$.

2) Suppose now that $j\geqslant n$, then
\begin{equation*}
\begin{split}
\int_0^{e^{-1/n}}{{x^{j-1}}\over {u(x)}}\,dx
&\leqslant {1\over
{u(e^{-1/n})}}\int_0^{e^{-1/n}} x^{j-1} \left( {{1-x}\over
{1-e^{-1/n}}} \right)^A \,  dx
\\
&\leqslant
 {{e^{A/n}n^A}\over {u(e^{-1/n})}}\int_0^{e^{-1/n}}
x^{j-1}(1-x)^A  \, dx
\\
&={{e^{A/n}n^A}\over {u(e^{-1/n})}}\int_{1/n}^{\infty }
(1-e^{-y})^Ae^{-jy}  \, dy
\\
&\leqslant {{e^{A/n}n^A}\over
{u(e^{-1/n})}}\int_{1/n}^{\infty } y^Ae^{-jy}  \, dy
\\
&={{e^{A/n}n^A}\over {u(e^{-1/n})}}{1\over
{j^{A+1}}}\int_{j/n}^{\infty } y^Ae^{-y} \,  dy \ll
{{e^{-j/n}}\over {ju(e^{-1/n})}},
\end{split}
\end{equation*}
since $\int_{w}^{\infty } y^Ae^{-y}  \, dy \ll w^Ae^{-w}$, as
$w\to \infty$.

The lemma is proved.
\end{proof}
%-------------------------END----------------------------------

%----------------------------------------------------------------
%--------------------Lemma 6-------------------------------------
%----------------------------------------------------------------

%{\bf Lemma 6.} {\it
\begin{lem}
\label{boundofg}
$$
\int_0^{e^{-1/n}} \left| g_{x,n} -{{p_n}\over {p(x)}}
\right|x^{j-1}\, dx \ll {{p(e^{-1/n})}\over n } \left( {{j^{\theta
-1}}\over{n^\theta}} \right) {1\over {p(e^{-1/j})}},
$$
when $1\leqslant j\leqslant n$.
\end{lem}

\begin{proof} Since $p(z)=p(xz){{p(z)}\over p(xz)}=p(xz)G_x(z)$,
therefore $p_n=\sum_{k=0}^np_kx^kg_{n-k,x}$ and
\begin{equation*}
\begin{split}
\left| g_{x,n} -{{p_n}\over {p(x)}}  \right| &=
 {1\over
{p(x)}}\left|p(x)g_{n,x} - \sum_{k=0}^np_kx^kg_{n-k,x} \right|
\\
&\leqslant {1\over {p(x)}} \left| \sum _{k=0}^{n}p_{k}x^{k}\left(
g_{n,x}- g_{n-k,x} \right) \right| +{{g_{n,x}}\over
{p(x)}}\sum_{k>n}p_{k}x^{k}
\\
&\leqslant {1\over {p(x)}} \sum _{k\leqslant n/8}p_{k}x^{k}\left|
g_{n,x}- g_{n-k,x} \right|  +{{g_{n,x}}\over
{p(x)}}\sum_{k>n/8}p_{k}x^{k}
\\
&\quad+{1\over {p(x)}}\sum_{n/8< k
\leqslant n}p_{k}x^{k}g_{n-k,x},
\end{split}
\end{equation*}
Suppose $0\leqslant x\leqslant e^{-1/n}$. Applying here Lemma \ref{diffgnm} we
have
\begin{equation*}
\begin{split}
\left| g_{x,n} -{{p_n}\over {p(x)}}  \right|&\ll
{{p(e^{-1/n})}\over {np(x)^{2}}}\sum_{k\leqslant
n/8}p_{k}x^{k}\left( \left( {k\over n} \right)^\theta + {1\over
{(n(1-x))^\theta }}
 \right)
\\
&\quad+{{G_x(e^{-1/n})}\over {np(x)}}\sum_{k> n/8}p_{k}x^{k}+
{{p(e^{-1/n})}\over n}{{x^{n/8}}\over {p(x)}}G_x(e^{-1/n})
\\
&=:S_{1}(x) +S_{2}(x) +S_{3}(x).
\end{split}
\end{equation*}
Applying Lemma \ref{intbounds} with $u(x)=p(x)^2$ and
 $u(x)=p(x)^2(1-x)^\theta$, we have
\begin{multline*}
\int_{0}^{e^{-1/n}} S_{1}(x)x^{j-1}\,dx
\\
={{p(e^{-1/n})}\over n}
\sum_{k\leqslant n/8}p_k\left( \left( {k\over n}\right)^\theta
\int_{0}^{e^{-1/n}} {{x^{k+j}}\over{p(x)^2}}\,dx +{1\over
{n^\theta}} \int_{0}^{e^{-1/n}}
{{x^{k+j}dx}\over{{p(x)^2}(1-x)^\theta }} \right)
\\
\shoveleft{\ll {{p(e^{-1/n})}\over n} \sum_{k\leqslant n/8}p_k\left( \left(
{k\over n}\right)^\theta  {1\over{(k+j)p(e^{-1/(k+j)})^2}}\right.}
\\
\shoveright{ \left.+{1\over{n^\theta}}
{1 \over{(k+j){p(e^{-1/(k+j)})^2}(1-e^{-1/(k+j)})^\theta }}
\right)}
\\
\shoveleft{\ll {{p(e^{-1/n})}\over n} \sum_{k\leqslant n/8}
{{p_k}\over{(k+j)p(e^{-1/(k+j)})^2}}\left( \left( {k\over
n}\right)^\theta
 +{{(k+j)^\theta}\over{n^\theta}}
\right)}
\\
\shoveleft{\ll {{p(e^{-1/n})}\over n} \left(
 {1\over j}\left( {j\over n}\right)^\theta
{1\over{p(e^{-1/j})^2}} \sum_{k\leqslant j} {p_k}
 +\sum_{k>j} {{p_k}\over k} \left( {k\over
n}\right)^\theta {1\over{p(e^{-1/k})^2}} \right).\hfill}
\end{multline*}
Applying here inequality $\sum_{k\leqslant j} {p_k}\leqslant
ep(e^{-1/j})$ in the first sum and the estimate (\ref{p_n2}) in the second one
we have
\begin{multline*}
\int_{0}^{e^{-1/n}} S_{1}(x)x^{j-1}\,dx
\\
\ll {{p(e^{-1/n})}\over n}
\left(
 {1\over j}\left( {j\over n}\right)^\theta
{1\over{p(e^{-1/j})}}
 +{1\over{p(e^{-1/j})}}\sum_{k>j} {1\over {k^2}} \left( {k\over
n}\right)^\theta {{p(e^{-1/j})}\over{p(e^{-1/k})}} \right).
\end{multline*}

Lemma \ref{pvbound} yields the estimate
${{p(e^{-1/j})}\over{p(e^{-1/k})}}\leqslant \left( {j\over k}
\right)^{d^-}e^{d^-/j}$, therefore
\begin{multline*}
\int_{0}^{e^{-1/n}} S_{1}(x)x^{j-1}\,dx
\\
\ll {{p(e^{-1/n})}\over n}
\left(
 {1\over j}\left( {j\over n}\right)^\theta
{1\over{p(e^{-1/j})}}
 +{1\over{p(e^{-1/j})}}\sum_{k>j} {1\over {k^2}} \left( {k\over
n}\right)^\theta \left( {j\over k} \right)^{d^-}  \right)
\\
\ll {{p(e^{-1/n})}\over n } \left( {{j^{\theta
-1}}\over{n^\theta}} \right) {1\over {p(e^{-1/j})} }.
\end{multline*}

Let us now estimate $S_2(x)$
$$
S_2(x)\ll {{p(e^{-1/n})}\over {np(x)^2} } \sum_{k>n/8}p_kx^k\leqslant
{{p(e^{-1/n})x^{n/16}}\over {np(x)^2} }p(\sqrt{x}) \ll
{{p(e^{-1/n})x^{n/16}}\over {np(x)} },
$$
since $p(y)\ll p(y^2)$ uniformly for $0\leqslant y <1$. Hence,
applying Lemma \ref{intbounds} we have
$$
\int_{0}^{e^{-1/n}}S_2(x)x^{j-1}\,dx \ll {{p(e^{-1/n})}\over n }
\int_{0}^{e^{-1/n}}{{x^{n/16}}\over {p(x)}}\,dx\ll {1\over {n^2}}.
$$
In a similar  way we obtain the estimate
$$
\int_{0}^{e^{-1/n}}S_3(x)x^{j-1}\,dx \ll {{p(e^{-1/n})^2}\over n }
\int_{0}^{e^{-1/n}}{{x^{n/8}}\over {p(x)^2}} \,dx\ll {1\over {n^2}}.
$$

Collecting the obtained estimates and noticing that
$$
{{p(e^{-1/n})}\over n } \left( {{j^{\theta
-1}}\over{n^\theta}} \right) {1\over {p(e^{-1/j})}}\gg \frac{1}{nj} \geqslant {1\over
{n^2}},
$$
for $1\leqslant j \leqslant n$, we obtain the proof of the lemma.
\end{proof}
%--------------------END----------------------------------------
Let us define
$$
F_j(z)=p(z)\int_{0}^{1}{{x^{j-1}}\over {p(zx)}}\,dx
=\sum_{m=0}^{\infty }f_{m,j}z^m.
$$
Denoting
$$
q(z)={1\over{p(z)}}=\sum_{m=0}^{\infty}q_mz^m
$$
we can
easily see that
$$
F_j(z)=p(z)\int_{0}^{1}{{x^{j-1}}{q(zx)}}\,dx
=\sum_{m=0}^{\infty}p_mz^m \sum_{s=0}^{\infty}{{q_s}\over
{s+j}}z^s=
 \sum_{m=0}^{\infty }f_{m,j}z^m,
$$
thus
$$
f_{m,j}=\sum_{s=0}^m{{p_{m-s}q_s}\over {s+j}}.
$$
On the other hand,
$$
F_j(z)=\int_{0}^{1}x^{j-1}G_x(z)\,dx=\sum_{m=0}^{\infty}z^m\int_0^1g_{m,x}x^{j-1}\,dx=\sum_{m=0}^{\infty
}f_{m,j}z^m,
$$
therefore
$$
f_{m,j}=\int_0^1g_{m,x}x^{j-1}\,dx\geqslant 0 .
$$
 Then the following lemma holds

%-----------------------------------------------------------------
%---------------------Lemma 7-------------------------------------
%-----------------------------------------------------------------

%{\bf Lemma 7.}{\it
\begin{lem}
\label{f_mj}
We have
$$
f_{m,j}\ll {1\over{j^2}}\quad \hbox{when}\quad j\geqslant
m\geqslant 1
$$
and  $f_{0,j}={1\over j}$.
\end{lem}

\begin{proof}
Differentiating $F_j(z)$ we can easily obtain, that
$$
zF_j'(z)=F_j(z)\sum_{k=1}^{\infty}d_kz^k +1-jF_j(z).
$$
Putting here $z=0$, we have $f_{0,j}={1/j}$.

For $m\geqslant 1$, we have for $j\geqslant m$
\begin{equation*}
\begin{split}
f_{m,j}&={1\over {m+j}}\sum_{k=1}^m d_kf_{m-k,j}\leqslant
{{d^+eF_j(e^{-1/m})}\over {m+j}}={{d^+e}\over
{m+j}}p(e^{-1/m})\int_0^1 {{x^{j-1}dx}\over {p(xe^{-1/m})}}
\\
&={{d^+e}\over {m+j}}p(e^{-1/m})e^{j/m}\int_0^{e^{-1/m}}
{{x^{j-1}dx}\over {p(x)}}\ll {1\over {j^2}},
\end{split}
\end{equation*}
here we have applied Lemma \ref{intbounds}.

The lemma is proved.
\end{proof}
%----------------------------------------------------------------
%------------------PROOF OF THEOREM 5-------------------------------------
%------------------------------------------------------------------

{\it Proof of Theorem \ref{fundthm1}.} One can easily see, that
$$
\sum_{m=1}^{\infty}S(f;m)z^{m}=zf'(z)p(z),
$$
therefore
$$
ma_m=\sum_{k=1}^{m}S(f;k)q_{m-k},
$$
where
$$
q(z)={1\over{p(z)}}=\sum_{m=0}^{\infty}q_mz^m.
$$
Therefore we have
\begin{equation*}
\begin{split}
\lefteqn{\sum_{k=0}^na_kp_{n-k}-p_n
f(e^{-1/n})}
\\
&=\sum_{k=1}^{n}p_{n-k}{1\over
k}\sum_{j=1}^{k}S(f;j)q_{k-j}
-p_n\sum_{k=1}^{\infty} {{e^{-k/n}}\over k
}\sum_{j=1}^{k}S(f;j)q_{k-j}
\\
&= \sum_{j=1}^n S(f;j)\sum_{k=j}^n
{{p_{n-k}q_{k-j}}\over k}
-p_n\sum_{j=1}^{\infty}S(f;j)\sum_{k=j}^{\infty}{{e^{-k/n}}\over
k}q_{k-j}
\\
&={{S(f;n)}\over
n}+\sum_{j=1}^{n-1}S(f;j)\sum_{s=0}^{n-j}{{p_{n-j-s}q_s}\over
{s+j}}-p_n
\sum_{j=1}^{\infty}S(f;j)\sum_{k=j}^{\infty}{{e^{-k/j}}\over
k}q_{k-j},
\end{split}
\end{equation*}
recalling that
$$
f_{m,j}=\sum_{s=0}^m{{p_{m-s}q_s}\over
{s+j}}=\int_0^1g_{m,x}x^{j-1}\,dx
$$
we have
\begin{multline*}
\sum_{k=0}^na_kp_{n-k}-p_n
f(e^{-1/n})
-{{S(f;n)}\over n}
\\
=\sum_{j=1}^{n-1}S(f;j)\left( f_{n-j,j}
%\left[
%p(z)\int_{0}^{1}{{x^{j-1}}\over {p(zx)}}dx \right]_{(n-j)}
-p_n \int_{0}^{e^{-1/n}} {{x^{j-1}}\over {p(x)}}\,dx \right)
\\
+p_n
\sum_{j=n}^{\infty}S(f;j)\int_{0}^{e^{-1/n}} {{x^{j-1}}\over
{p(x)}}\,dx.
\end{multline*}
Let us denote
$$
R_n=\sum_{k=0}^na_kp_{n-k}-p_n f(e^{-1/n})-{{S(f;n)}\over n},
$$
then we have
\begin{equation*}
\begin{split}
|R_n|&\leqslant \sum_{1\leqslant j \leqslant n/2}|S(f;j)|\left|
\int_{0}^{1}{x^{j-1}}g_{n-j,x}
\,dx -p_{n-j} \int_{0}^{e^{-1/n}} {{x^{j-1}}\over {p(x)}}dx
\right|
\\
&\quad+\sum_{1\leqslant j \leqslant n/2}|S(f;j)| |p_n -p_{n-j}|
\int_{0}^{e^{-1/n}} {{x^{j-1}}\over {p(x)}}\,dx
\\
&\quad +
\sum_{n/2\leqslant j \leqslant n-1}|S(f;j)| f_{n-j,j}
+p_n \sum_{j\geqslant n}|S(f;j)|\int_{0}^{e^{-1/n}}
{{x^{j-1}}\over {p(x)}}\,dx.
\end{split}
\end{equation*}
Since $g_{m,x}  \leqslant ed^+{{G_j(e^{-1/m})}\over m}= {{ed^+}
\over m} {{p(e^{-1/m})}\over {p(e^{-1/m}x)}}$, therefore we have
\begin{multline*}
\int_{e^{-1/n}}^{1}{x^{j-1}}g_{n-j,x}
\,dx\leqslant {{ep(e^{-1/(n-j)})}\over {n-j}}\int_{e^{-1/n}}^{1}
{{dx}\over {p(xe^{-1/(n-j)})}}
\\
\ll {{p(e^{-1/n})}\over
{n^2}}{1\over {p(e^{-1/n})}}={1\over {n^2}},
\end{multline*}
when $j\leqslant n/2$.

 Applying here Lemma \ref{intbounds}, Lemma \ref{boundofg} and Lemma  \ref{f_mj}, we have
\begin{equation*}
\begin{split}
|R_n|&\ll \sum_{1\leqslant j \leqslant n/2}|S(f;j)|
{{p(e^{-1/(n-j)})}\over {n-j} } \left( {{j^{\theta -1}}\over{(n-j)^\theta}}
\right) {1\over {p(e^{-1/j})}}
\\
&\quad+\sum_{1\leqslant j \leqslant n/2}|S(f;j)|  p_n \left( {j\over n}
\right)^{\theta} {1\over
{jp(e^{-1/j})}}+{1\over{n^2}}\sum_{n/2\leqslant j \leqslant n-1}|S(f;j)|
\\
&\quad
+p_n\sum_{ j > n/2}|S(f;j)| {{e^{-j/n}}\over {jp(e^{-1/n})}}
\\
&\ll
p_n \left( {1\over {n^\theta}}\sum_{j=1}^n {{|S(f;j)|}\over
{p(e^{-1/j})}}j^{\theta -1}+ {1\over
{p(e^{-1/n})}}\sum_{j>n}{{|S(f;j)|}\over j}e^{-j/n}\right).
\end{split}
\end{equation*}
The theorem is proved.

%--------------------------END-----------------------------------

%----------------------------------------------------------------
%-------------------PROOF OF THEOREM 6---------------------------
%----------------------------------------------------------------

\begin{proof}[ Proof of theorem \ref{voronmean}.]

{\it 1) Sufficiency.} Applying Theorem \ref{fundthm1} one can easily see that
condition 2 of Theorem \ref{voronmean} implies that
$$
{1\over
{p_n}}\sum_{k=0}^{n}a_kp_{n-k} =f(e^{-1/n})+o(1).
$$
Condition 1
finally proves the sufficiency of conditions 1 and 2 of Theorem \ref{voronmean}.

{\it 2) Necessity.} Suppose now that
$$
\lim_{n \to \infty}{1\over {p_n}}\sum_{k=0}^{n}a_kp_{n-k}=A.
$$
Let us denote
$$
w_n=\sum_{k=0}^{n}a_kp_{n-k}.
$$
Under our assumption we have $w_n=Ap_n(1+\epsilon_n )$, where
$\epsilon_n \to \infty$ as $n\to \infty$.

We have
\begin{equation*}
\begin{split}
\sum_{m=1}^{\infty}S(f;m)z^{m}&=zf'(z)p(z)=z(p(z)f(z))'-zp'(z)f(z)
\\
&=z(p(z)f(z))'-p(z)f(z)\sum_{k=1}^{\infty}d_kz^k
\\
&=\sum_{n=1}^{\infty}nw_nz^n-\sum_{n=1}^{\infty}w_nz^n
\sum_{k=1}^{\infty}d_kz^k.
\end{split}
\end{equation*}
Hence we obtain
\begin{equation*}
\begin{split}
S(f;n)&=nw_n-\sum_{k=1}^nd_kw_{n-k}
\\
&=Anp_n-A\sum_{k=1}^nd_kp_{n-k}+Anp_n\epsilon_n-
A\sum_{k=1}^nd_kp_{n-k}\epsilon_{n-k}
\\
&=Anp_n\epsilon_n- A\sum_{k=1}^nd_kp_{n-k}\epsilon_{n-k}=o(np_n),
\end{split}
\end{equation*}
since $np_n=\sum_{k=1}^nd_kp_{n-k}$ and $np_n\gg p(e^{-1/n})\to
\infty$ as $n\to \infty$.

The necessity of condition 2) is proved.

The necessity of condition 1) is well known, see e. g.
\cite{hardy}. It is obtained by noticing that
$$
f(x)=\frac{f(x)p(x)}{p(x)}=\frac{w_0+w_2x+\cdots+w_nx^n+\cdots}{p_0+p_2x+\cdots+p_nx^n+\cdots}.
$$
Putting here estimate $w_n=Ap_n(1+o(1))$ and using the fact that $p(x)\to \infty$ as $x\nearrow 1$ we finally obtain
$$
f(x)\to A\quad\mbox{for}\quad x\nearrow 1.
$$
The theorem is proved.
\end{proof}
%------------------------------------------------------------------
%-----------------PROOF OF THEOREM 2-------------------------------
%------------------------------------------------------------------
\section{Mean value theorems}

\begin{proof}[Proof of Theorem \ref{meanf}]
Let $M_j$ be defined by
$$
F(z)=\exp \left\{ \sum_{k=0}^{\infty}d_k{{\hat f(k)}\over
k}z^k \right\} =\sum_{n=0}^{\infty}M_nz^n ,
$$
where $\hat f(k) \in C$, $|\hat f(k)|\leqslant 1$.

One can easily see that
$$
F(z)=p(z)m(z),
$$
where
$$
m(z)=\exp \left\{ \sum_{k=1}^{\infty}{{d_k(\hat f(k)-1)}\over
{k}}z^k \right\}=\sum_{j=1}^{\infty}m_jz^j,
$$
then
$$
M_n=\sum_{k=0}^np_km_{n-k}.
$$
In the notations of Theorem \ref{fundthm1} we have
\begin{equation*}
\begin{split}
\sum_{j=1}^{\infty}S(m;j)z^j&=zm'(z)p(z)=
m(z)p(z)\sum_{k=1}^{\infty}z^kd_k(\hat f(k) -1)
\\
&=F(z)\sum_{k=1}^{\infty}z^kd_k(\hat f(k) -1)
\end{split}
\end{equation*}
hence we have
$$
S(m;j)= \sum_{k=1}^{j}d_k\bigl( \hat f(k) -1 \bigr)M_{j-k}.
$$
Since $|\hat f(k)|\leqslant 1$ implies $|M_j|\leqslant p_j$,
therefore
$$
|S(m;j)|\leqslant d^+ \sum_{k=1}^{j}|\hat f(k) -1|p_{j-k}.
$$
Since $\hat f(k)$ for $k>n$ do not influence $M_n$ we assume that
$\hat f(k)=1$ for $k>n$. Applying here Theorem \ref{fundthm1} with $f(z)=m(z)$
we have
\begin{multline*}
\left| {{M_n}\over {p_n}}- m(e^{-1/n}) \right|=\left| {{M_n}\over
{p_n}}- \exp \left\{ \sum_{k=1}^{n}{{d_k(\hat f(k)-1)}\over
{k}}e^{-k/n} \right\} \right|
\\
\leqslant c \left( {1\over {np_n}}\sum_{k=1}^{n}|\hat f(k) -1|p_{n-k}+
{1\over {n^\theta }}\sum_{j=1}^n{{j^{\theta -1}}\over
{p(e^{-1/j})}} \sum_{k=1}^{j}|\hat f(k) -1|p_{j-k} \right.
\\
\left. \quad\mbox{}+ {1\over {p(e^{-1/n})}}\sum_{j>n}{{e^{-j/n}}\over
j}\sum_{k=1}^{j}|\hat f(k) -1|p_{j-k} \right) ,
\end{multline*}
here  $\theta =\min \{ 1, d^- \}$ and $c=c( d^+, d^- )$. We have
\begin{multline*}
{1\over {n^\theta }}\sum_{j=1}^n{{j^{\theta -1}}\over
{p(e^{-1/j})}} \sum_{k=1}^{j}|\hat f(k) -1|p_{j-k}= {1\over
{n^\theta }}\sum_{k=1}^{n}|\hat f(k)-1| \sum_{n\geqslant j\geqslant
k}{{j^{\theta -1}}\over {p(e^{-1/j})}}p_{j-k}
\\
\shoveleft{\ll {1\over {n^\theta }}\sum_{k=1}^{n}|\hat f(k)-1|
\sum_{n\geqslant j\geqslant 2k}{{j^{\theta -1}}\over
{p(e^{-1/j})}}{{p(e^{-1/(j-k)})}\over {j-k}} }
\\
\shoveright{+
 {1\over
{n^\theta }}\sum_{k=1}^{n}|\hat f(k)-1| k^{\theta -1}\sum_{2k\geqslant
j\geqslant k}{{p_{j-k}}\over {p(e^{-1/j})}}         }%p_{j-k}
\\
\shoveleft{\ll {1\over {n^\theta}}\sum_{k=1}^{n}|\hat f(k)-1|\sum_{n\geqslant
j \geqslant 2k} j^{\theta -2}+{1\over
{n^\theta}}\sum_{k=1}^{n}|\hat f(k)-1|k^{\theta -1} {1\over
{p(e^{-1/k})}}\sum_{s=0}^k p_s}
\\
\shoveleft{   \ll {1\over {n^\theta}}\sum_{k=1}^{n}|\hat f(k)-1|\left( k^{\theta
-1} +\int_k^nx^{\theta-2}\,dx \right) \hfill  }
\end{multline*}
and
\begin{multline*}
{1\over {p(e^{-1/n})}}\sum_{j>n}{{e^{-j/n}}\over
j}\sum_{k=1}^n|\hat f(k)-1|p_{j-k}={1\over
{p(e^{-1/n})}}\sum_{k=1}^n|\hat f(k)-1|\sum_{j>n}{{e^{-j/n}}\over
j}p_{j-k}
\\
\leqslant {1\over n}\sum_{k=1}^n|\hat f(k)-1|{{e^{-k/n}}\over
{p(e^{-1/n})}}\sum_{j>n}{{e^{-(j-k)/n}}}p_{j-k}\leqslant {1\over
n}\sum_{k=1}^n|\hat f(k)-1|.
\end{multline*}

 The theorem is proved.
 \end{proof}
%------------------------END------------------------------------------

Let us define
$$
L_n(z)=\sum_{j=1}^nd_j{{\hat f(j)-1}\over j}z^j \quad \hbox{and}
\quad \quad \rho (p)=\left( \sum_{k=1}^n{{|\hat f(j)-1|^p}\over j}
\right)^{1/p},
$$
moreover we assume that
$$
\rho (\infty)=\lim_{p\to \infty}\rho (p)=\max_{1\leqslant k \leqslant
n}|\hat f(k) -1|.
$$

%-------------------------------------------------------------
%-----------------Lemma 8-------------------------------------
%-------------------------------------------------------------
%{\bf Lemma 8.}{\it
\begin{lem}
\label{diffLmn}
For $m,n\geqslant 1$ and $1/p+1/q=1$ with $\infty
\geqslant p>1$, we have
      $$
           |L_n(1)-L_m(1)|\leqslant d^+ \rho (p) {\left| \log {n\over m}
              \right|}^{1/q}.
       $$
\end{lem}
\begin{proof}
Suppose $n\geqslant m$ and $p<\infty$ then applying
Cauchy's inequality with parameters $p,q$ we have
\begin{equation*}
\begin{split}
       |L_n(1)-L_m(1)|&\leqslant d^+\sum_{m<j\leqslant n} {{|\hat f (j)-1|}\over j}
         \leqslant d^+{\left( \sum_{m<j\leqslant n} {{|\hat f (j)-1|^p}\over j}  \right) }^{1\over p}
                   {\left( \sum_{m<j\leqslant n} {1\over j}  \right)}^{1\over q}
\\
       &\leqslant d^+
       {\left( \sum_{1\leqslant j \leqslant N} {{|\hat f (j)-1|^p}\over j}  \right) }^{1\over p}
        {\left( \int_{m}^n{{dx}\over {x}}  \right)}^{1\over q}
        \leqslant d^+ \rho (p) {\left| \log {n\over m}
              \right|}^{1/q}.
\end{split}
\end{equation*}
Allowing $p\to \infty$ we see that this inequality is true for
$p=\infty$
 also.

The lemma is proved.
\end{proof}

%-----------------------------------------------------------------
%----------------------------LEMMA 9------------------------------
%-----------------------------------------------------------------

 %{\bf Lemma 9.} {\it
 \begin{lem}
 \label{sumpn}
 For $n\geqslant 1$ and $q(d^-
-1)>-1$, we have
$$
\sum_{j=1}^{n}
                {1\over j}
       {\left| { {p_{n-j}}\over {p_{n}} }
       -1 \right|}^q\ll 1.
$$
\end{lem}
\begin{proof}
We have
$$
\sum_{j=1}^{n} {1\over j}
       {\left| { {p_{n-j}}\over {p_{n}} }-1 \right|}^q\ll
\sum_{1\leqslant j\leqslant n/4} {1\over j}
       {\left| { {p_{n-j}-p_n}\over {p_{n}} } \right|}^q+
{1\over n}\sum_{n/4<j\leqslant n }
       {\left| { {p_{n-j}}\over {p_{n}} } \right|}^q+1,
$$
applying Lemma \ref{diffpnm} and Lemma \ref{pvbound} together with (\ref{p_n2}) and (\ref{plwbound}) we have
\begin{equation*}
\begin{split}
\sum_{j=1}^{n} {1\over j}
       {\left| { {p_{n-j}}\over {p_{n}} }-1 \right|}^q&\ll
\sum_{1\leqslant j\leqslant n/4} {1\over j}
       {\left|
       {j\over n} \right|}^{q\theta}+
{1\over n}\sum_{1\leqslant s\leqslant 3n/4 }
       {\left| { {p(e^{-1/s})}\over {p(e^{-1/n})} } {n\over s}  \right|}^q+1,
\\
&\ll 1+{1\over n}\sum_{s\leqslant n/4 }
       {\left| \left( { s\over n }\right)^{d^-} {n\over s}
       \right|}^q\ll 1+{1\over n}\sum_{1\leqslant s\leqslant 3n/4 }
        \left| { s\over n }\right|^{q(d^--1)}\ll 1.
\end{split}
\end{equation*}
The lemma is proved.
\end{proof}

%-----------------------------------------------------------------------------
%----------------------------------Theorem 8 ---------------------------------
%-----------------------------------------------------------------------------
The following theorem has been proved for $d_j\equiv 1$ by E.
Manstavi\v cius \cite{manstadditive}, later generalized for $d_j\equiv \theta >0$
in \cite{zakclt}.

 %{\bf Theorem 8.} {\it
 \begin{thm}
\label{meanM1}
For any fixed $\infty \geqslant p>\max \left\{ 1, 1/d^- \right\}$,
there exists such a positive $\delta=\delta (d^-,d^+,p)$ that, if
$\rho \leqslant \delta$, then

\begin{equation}
\label{manst_general}
{{M_N}\over {p_N}}=\exp \{  L_N(1) \} \left( 1+
\sum_{j=1}^Nd_j{{\hat f(j)-1}\over j} \left( {{p_{N-j}}\over
{p_{N}}}-1 \right) +O(\rho^2)\right).
\end{equation}

\end{thm}

%{\it Proof.}
\begin{proof}
We will assume during this proof that $\hat f(j)=1$,
if $j>N$ and as before $M_j$ will be defined by (\ref{genfunc}). We will
suppose that $R_m$ are complex numbers
 satisfying the relation
\begin{equation}
\label{M_n}
{{M_m}\over {p_{m}}}=\exp \{  L_m(1) \} \left(
        1+ \sum_{j=1}^m d_j{{\hat f(j)-1}\over j}
       \left( {{p_{m-j}}\over {p_{m}}}
          -1 \right) +R_m \right). %\eqno(11)
\end{equation}

We set
      $$
          R:=\sup_{m \geqslant 0} |R_m|.
       $$
One can easily see that $R$ is finite since $L_k(z)=L_N(z)$ if
$k\geqslant N$ and $|M_m|\leqslant p_m$.

 Suppose
  $$
           h_n(z)=\exp \{ L_N(z) \}
                   -\exp \{  L_n(1) \}  L_N(z),
  $$
where $n\geqslant 1$. It is easy to see that the generating function of
$S(h_n;j)$ is
$$
\sum_{j=1}^{\infty}S(h_n;j)z^j={zh_n'(z)} p(z)=
F_N(z)zL_N'(z)-\exp \{  L_n(1) \}zL'_N(z)  p(z),
$$
therefore
$$
S(h_n;m)=\sum_{k=1}^md_k\bigl( \hat f(k)-1
\bigr)\left( M_{m-k}-p_{m-k}\ex{n} \right) .
$$
Inserting here expression (\ref{M_n}) of $M_n$, we get
\begin{multline*}
      |S(h_n;m)|\leqslant  \sum_{k=1}^m d_k\bigl| \hat f(k)-1 \bigr| \bigl|
       \ex{m-k}-\ex{n} \bigr| p_{m-k}
   \\
   +d^+\rho \sum_{k=1}^{m-1}d_k|\hat f(k)-1| |\ex{m-k} |
      { \left( \sum_{j=1}^{m-k}
                {1\over j}
       {\left| { {p_{m-k-j}}\over {p_{m-k}} }
       -1 \right|}^q \right) }^{1/q}p_{m-k}
   \\
          + R\sum_{k=1}^{m-1} d_k|\hat f(k)-1||\ex{m-k}|p_{m-k}.
\end{multline*}
  Here we have applied Cauchy's inequality with parameters
${1/p}+{1/q}=1$ and used the fact that $R_0=0$.

\def \l#1{{\left| \log {n\over #1}
\right|}^{1/q}}

%%% Gal "Similarly" nekartojant jau girdetos frazes

Applying once more Cauchy inequality, we have
\begin{multline*}
         |S(h_n;m)|\ll\rho m^{1/p}
          {\left( \sum_{k=0}^{m-1}|\ex{k}-\ex{n}|^q {p_{k}}^q
          \right)  }^{1/q}+
      \\
+(\rho +R)\rho m^{1/p} {\left(
\sum_{k=1}^{m-1}|\ex{k}|^q{p_{k}}^q\right)}^{1/q}.
\end{multline*}

Applying here lemma \ref{diffLmn} we have
\begin{multline*}
 |S(h_n;m)| \ll (\rho^2+\rho R)m^{1/p}|\ex{n}|
 \left(
\exp \{ q\delta d^+ {\left| \log {n} \right|}^{1/q}\} \right.\nonumber \\
+{\left.
 \sum_{k=1}^{m-1}
\left( 1+\left| \log {n\over k} \right| \right)
 \exp \left\{ q \delta  d^+  \l{k} \right\}
              p_k^q
\right) }^{1/q}.\nonumber
 \end{multline*}
Here and further the constant in symbol $\ll$ will depend on
$\theta$ and $p$ only. As  $q(d^-  -1)>-1$ so we can chose such
a positive
 \mbox{$\epsilon =\epsilon (p,d^-,d^+) <  \theta$}, that $q(d^- -1)-\epsilon >-1$.
 When $\theta = 1$
 we might take, for example, $\epsilon ={1 \over 2}$ and for $d^-<1$
we might take $\epsilon =\min \left\{ {{1+q(d^- -1)}\over {2}},
{\theta \over 2} \right\}$. It is easy to see that there
 exists such a positive constant
$C_{\epsilon}$, that inequality
$$
\left( 1+\left| \log {x} \right|
\right) \exp \left\{ \epsilon /2 |\log x|^{1/q} \right\}
 \leqslant C_{\epsilon}\left( x^{\epsilon}
+{1\over {x^{\epsilon}}} \right),
$$
 holds for $x>0$.
Supposing that $\rho <\delta \leqslant {{\epsilon }\over {2 qd^+
}}$ and applying this inequality we get
$$
 |S(h_n;m)| \ll (\rho^2+\rho R)m^{1/p}|\ex{n}|
{ \left( 1+ \sum_{k=1}^{m-1} \left( { \left({n\over k}
\right)}^\epsilon + {\left( {k\over n} \right)}^\epsilon  \right)
              p_k^q
\right) }^{1/q}
$$
$$
 \ll (\rho^2+\rho R)m^{1/p}|\ex{n}| { \left( 1+ p(e^{-1/m})^q\sum_{k=1}^{m-1}
\left( { \left({n\over k} \right)}^\epsilon + {\left( {k\over n}
\right)}^\epsilon  \right)
              \left( {{p(e^{-1/k})}\over {kp(e^{-1/m})}}\right)^q
\right) }^{1/q}
$$
$$
 \ll (\rho^2+\rho R)m^{1/p}|\ex{n}| { \left( 1+ p(e^{-1/m})^q\sum_{k=1}^{m-1}
\left( { \left({n\over k} \right)}^\epsilon + {\left( {k\over n}
\right)}^\epsilon  \right)
              {1\over {k^q}}\left( {{k}\over {m}}\right)^{qd^-}
\right) }^{1/q}
$$
here we have estimated $p_k$ by means of (\ref{p_n2}) and have applied
Lemma \ref{pvbound}.  Estimating the sums occurring in this estimate we have

\begin{equation}
\label{s}
|S(h_n;m)|\ll |\ex{n}|\rho (\rho +R)p(e^{-1/m})\left(
{\left( {n\over m} \right) }^{\epsilon /q}
 +{\left( {m\over n} \right)}^{\epsilon /q}
\right),
\end{equation}

Applying Theorem \ref{fundthm1} with $f(z)=h_n(z)$ and using the estimate (\ref{s})
we have
\begin{multline*}
\lefteqn{{{M_n}\over {p_{n}}}-\ex{n}\sum_{j=1}^nd_j{{\hat
f(j)-1}\over j}{{p_{n-j}}\over {p_{n}}}-}
\\
-\left( \exp \{ L_N(e^{-1/n}) \} -\ex{n}
L_N(e^{-1/n}) \right)
\\
\ll \rho (\rho +R)|\ex{n}|.
\end{multline*}

Inserting here estimates
$$
\exp \{  L_N(e^{-1/n})\} = \exp \{  L_n(e^{-1/n})\}
\left( 1+ \bigl( L_N(e^{-1/n})- L_n(e^{-1/n}) \bigr)
+O(\rho^2) \right)
$$
  and
$$\
\exp \{  L_n(e^{-1/n})\} =\ex{n}\left( 1+ \bigl(
L_n(e^{-1/n})- L_n(1) \bigr) +O(\rho^2) \right),$$
 we obtain
\begin{multline*}
{{M_n}\over {p_n}}-\exp \{ L_n(1) \} \left( 1+
\sum_{j=1}^nd_j{{\hat f(j)-1}\over j} \left( {{p_{n-j}}\over {p_n}}
-1 \right)  \right)\\
\ll |\ex{n}|\rho (\rho +R).
\end{multline*}
Recalling the definition of $R_n$, we see that there exists such a
constant $A=A(d^-,d^+,p)$, that
$$
|\ex{n}||R_n|\leqslant A(d^-,d^+,p) |\ex{n}|\rho (\rho +R).
$$
Dividing each part of this inequality by $\ex{n}$ and taking the
supremum by $n$ we obtain
$$
R\leqslant A(d^- ,d^+,p)( \rho^2 +\rho R).
$$
Supposing now that $\delta=\min \left\{ {1\over
{2A(d^-,d^+,p)}},{{\epsilon }\over {2qd^- }} \right\}$, we have
$$
R\leqslant 2 A(d^-,d^+,p)\rho^2.
$$

The theorem is proved.
\end{proof}
%------------------------END-----------------------------

For $u>0$ we define
$$
E(u):=\exp \left\{ 2\sum_{{\scriptstyle k=1}\atop {\scriptstyle
|\hat f(k)-1|>u}}^n {{|\hat f(k)-1|}\over k} \right\}.
$$

%When $\theta=1$ the estimate of the following theorem could be deduced from

%-------------------------------------------------------------------------
%----------------------------- THEOREM 9----------------------------------
%-------------------------------------------------------------------------
%{\bf Theorem 9.} {\it
\begin{thm}
\label{thmeanM1}
There exists such a constant $\eta =\eta
(d^-,d^+)$ that for any $u\leqslant \eta$ we have
$$
\left| {{M_n}{p_n}^{-1}} \right|\ll \left| \exp \left\{  L_n(1)
\right\} \right| {\bigl( E(u) \bigr)}^{d^+ } ,
$$
\end{thm}
%{\it Proof.}
\begin{proof}
Let us denote for $u>0$
$$
\hat f_u(j)=\begin{cases}
1,& \text{if   $|\hat f(j)-1|>u$;} \\
\hat f(j), &
\text{if  $|\hat f(j)-1|\leqslant u$,}
\end{cases}
$$
and
$$
F_u(z):=\exp \left\{  \sum_{j=1}^{\infty}d_j{{\hat f_u(j)}\over
{j}}z^j \right\} ={{\exp \{ L_n^{(u)}(z)\} }p(z)  }
=\sum_{k=0}^{\infty} M_k^{(u)}z^k
$$
here
$$
L_n^{(u)}(z)=\sum_{j=1}^{n}d_j{{\hat f_u(j)-1}\over {j}}z^j=
\sum_{\substack{j\leqslant n \\|\hat f(j)-1|\leqslant u}}d_j{{\hat f(j)-1}
\over {j}}z^j.
$$
 Then
$$
F(z)=F_u(z)\exp \left\{  \sum_{|\hat f(j)-1|>u}d_j{{\hat
f(j)-1}\over {j}}z^j \right\}.
$$
Therefore
$$
M_n=\sum_{k=0}^n M_k^{(u)}m_{n-k}^{(u)},
$$
where $m_k^{(u)}$ are defined by
$$
m^{(u)}(z)=\sum_{j=0}^{\infty} m_j^{(u)}z^j=\exp\left\{ \sum_{|\hat
f(j)-1|>u} d_j{{\hat f(j)-1}\over {j}}z^j \right\}.
$$
Differentiating $m^{(u)}(z)$ one can easily see that $m_k$
satisfy the recurrent relationship
$$
m_j^{(u)}= {1\over {j}}\sum_{{\scriptstyle 1\leqslant k \leqslant
j} \atop {\scriptstyle |\hat f(j)-1| >u}}d_j(\hat
f(j)-1)m_{j-k}^{(u)},
$$
for $j\geqslant 1$. From which we have
\begin{equation*}
\begin{split}
|m_j^{(u)}|&=\left| {1 \over {j}}\sum_{{\scriptstyle 1\leqslant k
\leqslant j} \atop {\scriptstyle |\hat f(j)-1| >u}}d_j(\hat
f(j)-1)m_{j-k}^{(u)}\right| \leqslant {{2d^+}\over
j}\sum_{k=0}^{\infty}|m_k^{(u)}|= {{2d^+}\over j}
\sum_{j=0}^{\infty} |m_j^{(u)}|\\
&\leqslant  {{2d^+}\over j}\exp\left\{ d^+\sum_{|\hat f(j)-1|>u}
{{|\hat f(j)-1|}\over {j}} \right\}={{2d^+}\over j}E^{d^+ /2}(u).
\end{split}
\end{equation*}

Suppose that $\eta\leqslant \delta(d^-,d^+, \infty )$, then
 Theorem \ref{meanM1} implies that
$M_k^{(u)}=p_k\exp \{ L_k^{(u)}(1) \}(1+O(\eta))$ and we have

\begin{equation*}
\begin{split}
|M_n|&=\sum_{k=0}^n |M_k^{(u)}m_{n-k}^{(u)}|
\\
&\ll
E^{d^+/2}(u)\sum_{k\leqslant n/2}{{p_k}\over {n-k}} |\exp \{
L_{k}^{(u)}(1) \}| +|\exp \{  L_{n}^{(u)}(1) \}|p_n\left(
\sum_{k=0}^{\infty}|m_k^{(u)}| \right)
\\
 &\ll E^{d^+/2}(u)|\exp
\{ L_n^{(u)}(1) \} | \left( {1\over n}\sum_{k\leqslant n/2}
{{p(e^{-1/k})}\over {k}} {\left( {n\over k} \right) }^{ud^+} +1
\right)
\\
&\ll {{p(e^{-1/n})}\over n}|\exp \{ L_n^{(u)}(1) \} |
E^{d^+ /2}(u)\left( \sum_{k\leqslant n/2} {{1}\over k}{\left(
{n\over k} \right) }^{ud^+-d^-} +1 \right),
\end{split}
\end{equation*}

here we have used the estimate $|\exp \{
(L_k^{(u)}(1)-L_n^{(u)}(1)) \} | \leqslant {\left( {n\over
k}\right)}^{ud^+ }$ for $k\leqslant n$ and Lemma \ref{pvbound}.
 Assuming that $\eta$ fixed such that
$u\leqslant\eta < {{d^-}\over {d^+}}$, we obtain
$$
|M_n|\ll   p_n|\exp \{ L_n^{(u)}(1) \} | E^{d^+ /2}(u) \leqslant
p_n|\exp \{  L_n(1) \} | E^{d^+ }(u).
$$
Thus we have that the theorem holds with $\eta =\min
\{\delta(d^-,d^+,\infty ),d^+/(2d^-)\}$.
\end{proof}

%---------------------------------------------------------------------
%---------------Proof of theorem 4 -----------------------------
%---------------------------------------------------------------------

{\it Proof of theorem \ref{clt1}.} Putting
 $\hat f(k)=e^{ith_{n}(k)}$ in Theorem \ref{meanM1}, for
$|t|\leq \delta L_{n,p}^{-1/p}$ we have
\begin{equation}
\label{phi}
\begin{split}
\phi_n (t)&:=\int_{-\infty}^{\infty}e^{itx}\,dF_n(x)
\\
&=e^{-itA(n)}\exp \{  L_n(1) \} \left( 1+
\sum_{j=1}^nd_j{{f(j)-1}\over j}\left( {{p_{n-j}}\over {p_{n}}} -1
\right) +O(\rho^2)\right)
\\
&=e^{-t^2/2+O(|t|^3L_{3,n})} \left(
1+C_nit+O\biggl(|t|^2\sum_{j=1}^nd_j{{a_{nk}^2}\over j}\left|
{{p_{n-j}}\over {p_{n}}} -1 \right| \biggr)+O\bigl(
|t|^2L_{n,p}^{2/p}\bigr) \right)
\\
&=e^{-t^2/2}\left( 1+C_nit+ O\bigl(
 |t|^2(1+|t|^3)(L_{n,p}^{2/p}+L_{n,3} +L'_{n,2})\bigr)
\right).
\end{split}
\end{equation}

As in \cite{manstberry}, from Theorem \ref{thmeanM1} we deduce the existence of some
sufficiently small  $c=c(d^-,d^+)$, that if
 $|t|\leq cL_{n,3}^{-1}=:T$ then

\begin{equation}
\label{phibound}
|\phi (t)|\ll e^{-c_1t^2},
\end{equation}

here $c_1=c_1(d^-,d^+)$ is some fixed positive constant. Applying
the generalized Eseen inequality (see for example \cite{petrov}), we obtain
$$
\sup_{x\in \sym R}\left| F_n(x)-\Phi (x)+{1\over
{\sqrt{2\pi}}}e^{-x^2/2}
          C_n\right|\ll \int_{- T}^T{{\left| \phi_n(t)- e^{-t^2/2}\left( 1+C_nit
\right) \right|}\over {|t|}}\,dt+{1\over T}.
$$
Representing the integral on the right hand side of this
inequality as a sum of integrals over the intervals $|t|\leq
\delta L_{n,p}^{-1/p}$ and $\delta L_{n,p}^{-1/p}<|t|\leq T$ and
 applying estimates (\ref{phi}) and (\ref{phibound}) in those intervals we obtain
 the proof of the theorem.

The theorem is proved.

\section{Example of a special function}
\label{sec_example}
Let us consider now the uniform probability measure  $\nu_{n,1}$  on $S_n$.
% which means that for any subset $A\subset S_n$ we
% have $\nu_n (A)={|A|}/{n!}$, here $|A|$ stands for the number of
% elements of the set $A$.

As noted in
\cite{manstberry}, the Cauchy inequality yields that $L_n\gg
\frac{1}{\sqrt{\log n}}$ for any sequence $\hat h_{n}(k)$
satisfying the normalizing condition (\ref{normalization}). Thus,
the convergence rate is at most of the logarithmic order.

In the work \cite{manstbabubrown} G. J. Babu and E. Manstavi\v
cius  investigated the additive function $h_n(\sigma)$ with $\hat
h_n(j)=d(j)\left({j}/{n}\right)^{1/2}$, where
$$
d(j)=\begin{cases}
\Phi^{-1} (\{j\sqrt{2}\}),&\text{if  $|\Phi^{-1} (\{j\sqrt{2}\})| \leqslant \log j$}, \\
0,&\text{in other cases.}
\end{cases}
$$
Here $\Phi(x)$ is the standard normal distribution and $\{x\}$ is
the fractional part of a real number $x$. They showed that,
although for  this function $L_n\gg 1$, distribution converges to
the normal law. Now the summands corresponding to the long cycles
are not negligible for the asymptotic distribution. The in this section we will  estimate the convergence rate of the distribution
$F_n(x)=\nu_n(h_n(\sigma)<x)$ of this special additive function.

\begin{thm}
\label{supdiff}
There exists  a  positive constant $c>0$
such that
$$
\sup_{x\in \sym R}|F_n(x)-\Phi(x)|\ll \frac{1}{n^c}.
$$
\end{thm}
  The proof allows to obtain some numerical estimate of $c$.

%\section{Proofs}

%-----------------------------------------------------------------------------------------------------------
%---------------Lemma 1---------------------------------------------------------------------------------
%-----------------------------------------------------------------------------------------------------------
\begin{lem}
\label{peronlike}
Suppose $F(z)=\sum_{k=0}^{\infty}a_kz^k$ is
analytic for $|z|<1$. Then we have for $T\geqslant 2$
$$
\sum_{k=0}^n a_k ={1\over {2\pi i}}{\int_{1-iT}^{1+iT}}{{e^s F(e^{-{s/n}})}\over s}ds+O\left( \sum_{k=0}^{\infty}{{|a_k|e^{-{k/n}}}\over{1+T\left|{k\over n}-1 \right|}} \right).
$$
\end{lem}

\begin{proof}
 The proof of this result is the same as that of
Peron's formula for Dirichlet series (see e. g.  \cite{tenenbaum}).

Applying the well known estimates

\begin{equation}
\label{contourint}
{1\over {2\pi i}}{\int_{1-iT}^{1+iT}}{{e^{xs} }\over
s}ds=\begin{cases}
 1+O\left( {{e^x}\over {T|x|}}\right)&\text{if $x>0$},\\
O\left( {{e^x}\over {T|x|}} \right)&\text{if $x<0$}
\end{cases}
\end{equation}
and noticing that

\begin{eqnarray}
\frac{1}{2\pi i}{\int_{1-iT}^{1+iT}}{{e^{xs} }\over s}\,ds&=&\frac{e^x}{2\pi }\int_{|t|<T}\frac{dt}{1+it}+\frac{e^x}{2\pi }
\int_{|t|<T}\frac{e^{ixt}-1}{1+it}\,dt\nonumber
\\ \label{intestim}
&\ll& e^x(1+|x|T),
\end{eqnarray}
we have
\begin{eqnarray*}
&&{1\over {2\pi i}}{\int_{1-iT}^{1+iT}}{{e^s F(e^{-{s/n}})}\over s}ds=\sum_{k=0}^\infty a_k
\frac{1}{2\pi i}{\int_{1-iT}^{1+iT}}{{e^{s(1-k/n)} }\over s}\,ds
\\
&=&\sum_{\left( 1-k/n \right)T>1  } a_k
\frac{1}{2\pi i}{\int_{1-iT}^{1+iT}}{{e^{s(1-k/n)} }\over s}\,ds
\\
&&\mbox{}+
\sum_{\left| 1-k/n \right|T\leqslant 1  } a_k
\frac{1}{2\pi i}{\int_{1-iT}^{1+iT}}{{e^{s(1-k/n)} }\over s}\,ds
\\
\mbox{}
&&+\sum_{\left( 1-k/n \right)T<-1  } a_k
\frac{1}{2\pi i}{\int_{1-iT}^{1+iT}}{{e^{s(1-k/n)} }\over s}\,ds=I_1+I_2+I_3.
\end{eqnarray*}
Estimating $I_1$ and $I_3$ we  use (\ref{contourint}). Estimating $I_2$    by means of
(\ref{intestim}), we finally obtain the desired result.

The lemma is proved.
\end{proof}
Suppose $k_j(\sigma)$ is the number of cycles in the $\sigma$
whose length is equal to $j$.
 Then obviously $\sum_{j=1}^njk_j(\sigma)=n$, and any additive function can be represented
 as $h_n(\sigma)=\sum_{j=1}^n\hat h_n(j)k_j(\sigma)$. Therefore
$$
\phi_n(t)={\bf M_n}e^{ith(\sigma)}=\frac{1}{n!} \sum_{\sigma \in S_n}e^{it h_n(\sigma)}=
\sum_{s_1+2s_2 +\cdots+ns_n=n}
\prod_{j=1}^n\left( \frac{e^{it \hat h_n(j)}}{j} \right)^{s_j}\frac{1}{s_j!},
$$
here we have used the well known fact that
$$
\nu(k_1(\sigma)=s_1,k_2(\sigma)=s_2,\ldots,
k_n(\sigma)=s_n)=\prod_{j=1}^n\frac{1}{j^{s_j}s_j!},
$$
for $s_1+2s_2+\cdots+ns_n=n$. Hence
$$
1+\sum_{n=1}^{\infty}\phi_n(t)z^n=\exp
\biggl\{\sum_{j=1}^\infty\frac{e^{it \hat h_n(j)}}{j}z^j \biggr\}.
$$
Suppose
$$
F(z)=\exp \left\lbrace \sum_{k=1}^{\infty}\frac{\hat f(k)}{k}z^k
\right\rbrace = \frac{1}{1-z}\exp \left\lbrace
\sum_{k=1}^{\infty}\frac{\hat f(k)-1}{k}z^k
\right\rbrace=\sum_{k=0}^{\infty}M_kz^k.
$$
Let us define $m_j$ by relationship
$$
m(z)=\exp \left\lbrace \sum_{k=1}^{\infty}\frac{\hat f(k)-1}{k}z^k \right\rbrace =\sum_{k=0}^{\infty}m_kz^k.
$$
Hence we have $M_n=m_0+m_1+\cdots +m_n$.

%-----------------------------------------------------------------------------------------------------------
%---------------Lemma 2---------------------------------------------------------------------------------
%-----------------------------------------------------------------------------------------------------------
\begin{lem}
\label{peronmean}
Suppose $2\leqslant T\leqslant n$ and $|\hat f(j)|\leqslant 1$, then
$$
M_n ={1\over {2\pi i}} {\int_{1-iT}^{1+iT}}{{e^s }\over s} \exp
\left\lbrace \sum_{k=1}^{n}\frac{\hat f(k)-1}{k}e^{-sk/n}
\right\rbrace ds+ O\left( \frac{\log T}{T}\exp \left\lbrace
\sum_{k=1}^{n}\frac{|\hat f(k)-1|}{k} \right\rbrace \right).
$$
\end{lem}
\begin{proof} Since $\hat f(j)$ for $j>n$ do not influence the
value of $M_n$ we may assume that $\hat f(j)=1$ for $j>n$. One can
easily see that

\begin{equation}
\label{summ_ke^-k/n}
\sum_{k=0}^{\infty}|m_k|e^{-k/n}\leqslant \exp
\left\lbrace \sum_{k=1}^{n}\frac{|\hat f(k)-1|}{k}e^{-k/n}
\right\rbrace.
\end{equation}

Differentiating $m(z)$ one can easily verify that
$m_j$ satisfy the recurrence relationship
$$
m_N=\frac{1}{N}\sum_{k=1}^{N} (\hat f(j)-1)m_{N-k},
$$
for $N\geqslant 1$. Hence, recalling that $|\hat f(j)|\leqslant
1$, we have

\begin{equation}
\label{m_Nleq}
|m_N|\leqslant \frac{2}{N}\sum_{k=1}^{N}|m_{N-k}|\leqslant \frac{2}{N}\exp \left\lbrace \sum_{k=1}^{n}\frac{|\hat f(k)-1|}{k} \right\rbrace.
\end{equation}

These estimates yield
\begin{eqnarray*}
\sum_{k=0}^{\infty}{{|m_k|e^{-{k/n}}}\over{1+T\left|{k\over n}-1 \right|}}&\ll& \sum_{n/2<k<3n/2}{{|m_k|e^{-{k/n}}}\over{1+T\left|{k\over n}-1 \right|}}+
\frac{1}{T}\sum_{k=0}^{\infty}|m_k|e^{-{k/n}}
\\
&\ll& \exp \left\lbrace \sum_{k=1}^{n}\frac{|\hat f(k)-1|}{k} \right\rbrace\left(
\frac{\log T}{T}+ \frac{1}{T} \right)
\end{eqnarray*}
here we have applied (\ref{summ_ke^-k/n}) to estimate the sum over $k$ such that $|n-k|\geqslant n/2$
and used (\ref{m_Nleq}) to estimate the sum over $k$ such that $n/2<k<3n/2$.

Applying now lemma \ref{peronlike} with $a_n=m_n$ together with the last
estimate we complete the proof of the lemma.

The lemma is proved.
\end{proof}

\begin{proof}[Proof of Theorem \ref{supdiff} ]
In the work \cite{manstbabubrown} it has been shown that
$$
\frac{1}{n}\sum_{\scriptstyle 1 \leqslant j \leqslant n \atop
\scriptstyle d(j)<x}1=\Phi(x)+
O(n^{-1/2})\quad\hbox{and}\quad\frac{1}{n}\sum_{j=1}^n|d(j)|=\int_{-\infty}^{\infty}|x|\,d\Phi(x)+O\left(\frac{\log^2
n}{\sqrt{n}}\right).
$$
Let us denote $\hat f(j)=\exp \left\lbrace it d(j){\left( \frac{j}{n}\right) }^{1/2} \right\rbrace$
for $1\leqslant j \leqslant n$.
Then we have for $1<|t|\leqslant \sqrt{ n}$
$$
\sum_{j=1}^{n}\frac{|\hat f(j)-1|}{j}
%\sum_{j=1}^{n}\frac{\left| \exp \left\lbrace it d(j){\left( \frac{j}{n}\right) }^{1/2} \right\rbrace %-1\right|}{j}
\leqslant |t| \sum_{1\leqslant j <\frac{n}{|t^2|}}
\frac{|d(j)|}{j}{\left( \frac{j}{n}\right) }^{1/2}
+2\sum_{\frac{n}{|t^2|}\leqslant j \leqslant n}\frac{1}{j} \leqslant 4 \log |t| +O(1).
$$
For $|t|\leqslant 1$, we have $\sum_{j=1}^{n}\frac{|\hat
f(j)-1|}{j}=O(|t|)$.

Using these estimates together with lemma \ref{peronmean} we have

\begin{equation}
\label{charassympt}
M_n={1\over {2\pi i}} {\int_{1-iT}^{1+iT}}{{e^z }\over z} \exp
\bigl\{ S_n(t,z)  \bigr\} \,dz +O\left(\frac{\log T}{T}(1+|t|^4)
\right),
\end{equation}
here, as in \cite{manstbabubrown}, we denote $S_n(t,z)=\sum_{j=1}^n\frac{\exp
\bigl\lbrace it d(j){\left( {j}/{n}\right) }^{1/2}\bigr\rbrace
-1}{j}e^{-zj/n} $ where $z=1+iu$, and $u \in  R$. Further we
assume that $p$ is a natural number such that
 $p<n$ and we put $H=n/p$.
 We have
\begin{eqnarray*}
S_n(t,z)
&=&\sum_{k=1}^{p-1}\sum_{Hk<j\leqslant H(k+1)}\frac{\exp \bigl\{ it d(j){\left( {j}/{n}\right) }^{1/2}\bigr\} -1}{j}e^{-zj/n}+O\left(\frac{|t|}{p^{1/2}}
\right).
\\
 &=&\sum_{k=1}^{p-1}\sum_{Hk<j\leqslant H(k+1)}\frac{\exp \bigl\{ it d(j){\left( {Hk}/{n}\right) }^{1/2}\bigr\} -1}{Hk}e^{-zHk/n}
 \\
 &&\mbox{}+R_n + O\left(\frac{|t|}{p^{1/2}}\right),
\end{eqnarray*}
where
\begin{eqnarray*}
R_n&\ll& \sum_{k=1}^{p-1}\sum_{Hk<j\leqslant H(k+1)}\left|\frac{\exp \bigl\{ it d(j){\left( {j}/{n}\right) }^{1/2}\bigr\} -1}{j}\right||e^{-zj/n}-e^{-zHk/n}|
\\
&&\mbox{}+\sum_{k=1}^{p-1}\sum_{Hk<j\leqslant H(k+1)}\left|\frac{\exp \bigl\{ it d(j){\left( {j}/{n}\right) }^{1/2}\bigr\} -1}{j}-\frac{\exp \bigl\{ it d(j){\left( {Hk}/{n}\right) }^{1/2}\bigr\} -1}{Hk}\right|
\\
&=:&R_{n1}+R_{n2}.
\end{eqnarray*}
Now
\begin{eqnarray*}
R_{n1}&\ll& \sum_{k=1}^{p-1}\sum_{Hk<j\leqslant H(k+1)}\frac{|z|}{j}\frac{j-Hk}{n}\ll
\frac{|z|}{n}\sum_{k=1}^{p-1}\frac{1}{Hk}\sum_{Hk<j\leqslant H(k+1)}|j-Hk|
\\
&\ll&
\frac{|z|}{n}H\log p=|z|\frac{\log p}{p}.
\end{eqnarray*}
In a similar way we obtain
\begin{eqnarray*}
R_{n2}&\ll& \sum_{k=1}^{p-1}\sum_{Hk<j\leqslant H(k+1)}\frac{\bigl|\exp \bigl\{ it d(j){\left( {j}/{n}\right) }^{1/2}\bigr\} -\exp \bigl\{ it d(j){\left( {Hk}/{n}\right) }^{1/2}\bigr\}\bigr|}{j}
\\
&&\mbox{}+ \sum_{k=1}^{p-1}\sum_{Hk<j\leqslant H(k+1)}
\bigl|\exp \bigl\{ it d(j){\left( {Hk}/{n}\right) }^{1/2}\bigr\}-1\bigr|\left|\frac{1}{j}- \frac{1}{Hk}\right|
=:R_{n2}'+R_{n2}''
\end{eqnarray*}
We have
\begin{eqnarray*}
R_{n2}'&\ll&|t|\sum_{k=1}^{p-1}\left(\frac{Hk}{n}\right)^{1/2}\sum_{Hk<j\leqslant H(k+1)}
\frac{|d(j)|}{j}\left| \left(\frac{j}{Hk}\right)^{1/2}-1\right|
\\
&\ll&
|t|\sum_{k=1}^{p-1}\left(\frac{Hk}{n}\right)^{1/2}\frac{1}{Hk}\frac{1}{k}
\sum_{Hk<j\leqslant H(k+1)}|d(j)|
\\
&\ll&
|t|\log n\left( \frac{H}{n}\right)^{1/2}=|t|\frac{\log n}{p^{1/2}},
\end{eqnarray*}
since $|d(j)|\leqslant \log j$. Using similar considerations we obtain
\begin{eqnarray*}
R_{n2}''&\ll&|t|\sum_{k=1}^{p-1}\sum_{Hk<j\leqslant H(k+1)}
\left(\frac{Hk}{n}\right)^{1/2}|d(j)|\left|\frac{1}{j}- \frac{1}{Hk}\right|
\\
&\ll&|t|\log n\sum_{k=1}^{p-1}\left(\frac{Hk}{n}\right)^{1/2}\frac{1}{(Hk)^2}\sum_{Hk<j\leqslant H(k+1)}
|j-Hk|
\\
&\ll&|t|\log n\sum_{k=1}^{p-1}\left(\frac{H}{n}\right)^{1/2}\frac{1}{k^{3/2}}\ll|t|\frac{\log n}{p^{1/2}}.
\end{eqnarray*}
Therefore
$$
R_n\ll|z|\frac{\log p}{p}+|t|\frac{\log n}{p^{1/2}}.
$$

One can  see that
$$
V_k(x):=\frac{1}{H}\sum_{\scriptstyle Hk<j\leqslant H(k+1) \atop \scriptstyle d(j)<x}1=\Phi(x)+\alpha_k(x),
$$
where $|\alpha_k(x)|=O(H^{-1/2})$, because
$$
\frac{1}{H}\sum_{\scriptstyle Hk<j\leqslant H(k+1) \atop \scriptstyle \{j\sqrt{2}\}<x}1=
\frac{1}{H}\sum_{\scriptstyle 0<s\leqslant H \atop \scriptstyle \{[Hk]\sqrt{2}+\{s\sqrt{2}\}\}<x}1
+O\left( \frac{1}{H} \right)=x +O\left( \frac{1}{H^{1/2}} \right),
$$
for $0<x<1$.
Now we can estimate
\begin{eqnarray*}
\lefteqn{\frac{1}{H}\sum_{Hk<j\leqslant H(k+1)}\left( \exp \bigl\{ it d(j){\left( {Hk}/{n}\right) }^{1/2}\bigr\} -1\right)}
\\
&=&\int_{-\log n}^{\log n}\left( \exp \bigl\{ it{\left( {Hk}/{n}\right) }^{1/2}x\bigr\} -1\right)
dV_k(x)
%\biggl( \frac{1}{H}\sum_{\scriptstyle Hk<j\leqslant H(k+1) \atop \scriptstyle d(j)<x}1 \biggr)
\\
&=&\int_{-\log n}^{\log n}\left( \exp \bigl\{ it{\left( {Hk}/{n}\right) }^{1/2}x\bigr\} -1\right)d\Phi(x)
\\
&&\mbox{}+
\int_{-\log n}^{\log n}\left( \exp \bigl\{ it{\left( {Hk}/{n}\right) }^{1/2}x\bigr\} -1\right)d\alpha_k(x)
\\
&=&\int_{-\infty}^{\infty}\left( \exp \bigl\{ it{\left(
{Hk}/{n}\right) }^{1/2}x\bigr\} -1\right)
d\Phi(x)
\\
&&\mbox{}+O\left(\frac{1}{n} \right) +O(H^{-1/2})+O\left(|t|
\left(\frac{Hk}{n} \right)^{1/2} \frac{\log n}{H^{1/2}}   \right).
\end{eqnarray*}
Hence recalling that $H=n/p$ and $|z|>1$, we have
\begin{eqnarray*}
S(t,z)&=&\sum_{k=1}^{p}\frac{e^{-zk/p}}{k}\int_{-\infty}^{\infty}\left( \exp \bigl\{ it{\left( {k}/{p}\right) }^{1/2}x\bigr\} -1\right)\,d\Phi(x)
\\
&&\mbox{}+O\left( |z|\frac{\log p}{p}+|t|\frac{\log n}{p^{1/2}}
+(1+|t|)\frac{\log n}{H^{1/2}} \right).
\end{eqnarray*}
Now we will apply the well-known formula
$$
\sum_{n=a}^bf(n)=f(a)+\int_{a}^bf(x)\,dx+\int_a^b\{x\}f'(x)\,dx,
$$
where $a,b\in {\sym Z}$. Putting here
$$
f(y)=\frac{e^{-zy/p}}{y}\int_{-\infty}^{\infty}\left( \exp \bigl\{
it{\left( {y}/{p}\right) }^{1/2}x\bigr\} -1\right)\,d\Phi(x),
$$
we have
\begin{eqnarray*}
&&\sum_{k=1}^{p}\frac{e^{-zk/p}}{k}\int_{-\infty}^{\infty}\left( \exp \bigl\{ it{\left( {k}/{p}\right) }^{1/2}x\bigr\} -1\right)\,d\Phi(x)
\\
&=&\int_1^p\frac{e^{-zy/p}}{y}\int_{-\infty}^{\infty}\left( \exp \bigl\{ it{\left( {y}/{p}\right) }^{1/2}x\bigr\} -1\right)\,d\Phi(x)\,dy
\\
&&\mbox{}+O\left( \int_1^p|f'(y)|\,dy \right)+O\left( \frac{|t|}{p^{1/2}} \right)
\\
&=&\int_{0}^1\frac{e^{-zy}}{y}\int_{-\infty}^{\infty}\left( \exp \bigl\{ it{y }^{1/2}x\bigr\} -1\right)\,d\Phi(x)\,dy+O\left( |t|\frac{|z|}{p}+\frac{|t|}{p^{1/2}} \right),
\end{eqnarray*}
since $|f'(y)|\ll\frac{ |t||z|}{p^{3/2}y^{1/2}}+\frac{|t|}{y^{3/2}p^{1/2}}$.

As noted in \cite{manstbabubrown}
$$
\int_{0}^1\frac{e^{-zy}}{y}\int_{-\infty}^{\infty}\left( \exp \bigl\{ it{y }^{1/2}x\bigr\} -1\right)\,d\Phi(x)\,dy
=\log\frac{z}{z+t^2/2}-\int_{z+t^2/2}^z\frac{e^{-w}}{w}\,dw,
$$
therefore, inserting the obtained estimates into
(\ref{charassympt}) and taking $p=n^{1/2}$ and $T=n^{1/4}$, we
have for $|t|<\frac{n^{1/4}}{\log n}$
\begin{eqnarray*}
M_n&=&\frac{1}{2\pi i}\int_{1-iT}^{1+iT}\frac{e^z}{z+\frac{t^2}{2}}
\exp\left\lbrace -\int_{z+\frac{t^2}{2}}^z\frac{e^{-w}}{w}dw \right\rbrace \,dz
\\
&&\mbox{}+O\left( (1+|t|)\frac{\log^2 n}{n^{1/4}}+(1+|t|^4)\frac{\log n}{n^{1/4}} \right).
\end{eqnarray*}
Since as shown in \cite{manstbabubrown}
$$
\frac{1}{2\pi
i}\int_{1-in^{1/4}}^{1+in^{1/4}}\frac{e^z}{z+\frac{t^2}{2}}
\exp\left\lbrace -\int_{z+\frac{t^2}{2}}^z\frac{e^{-w}}{w}dw
\right\rbrace dz= e^{-t^2/2}+\left( \frac{\log n}{n^{1/4}}
\right),
$$
we finally have the following estimate for $M_n$

\begin{equation}
\label{M_nis}
M_n=e^{-t^2/2}+O\left( (1+|t|)\frac{\log^2 n}{n^{1/4}}+(1+|t|^4)\frac{\log n}{n^{1/4}} \right).
\end{equation}

We will also need some crude estimate of $g_n(t)=M_n$ for small $|t|\leq 1$.

\begin{equation}
\label{M_n-1}
|M_n-1|\leqslant\sum_{k=1}^{\infty}|m_k|\leqslant \exp \left\lbrace \sum_{k=1}^{n}\frac{|\hat f(k)-1|}{k} \right\rbrace-1\ll |t|.
\end{equation}

The Berry-Eseen inequality (see e. g. \cite{petrov}) gives
$$
\sup_{x\in \sym R}|F_n(x)-\Phi(x)|\ll \int_{-U}^U\frac{|g_n(t)-e^{-t^2/2}|}{|t|}\,dt+\frac{1}{U}.
$$
Putting here $U=n^{1/17}$, we split the integration contour into
two parts $|t|\leqslant \frac{1}{n}$ and $\frac{1}{n}\leqslant
|t|\leqslant U$ and applying in each of these interval
correspondingly the estimates (\ref{M_n-1}) and (\ref{M_nis}) we complete the proof of
the theorem.

The theorem is proved.
\end{proof}

\chapter{Erd\H os Tur\'an law}
%\section{Introduction}

%\input{erd.tex}

\section{Proofs}

Let  $f\colon\ ({\sym Z}^+)^n \to {\sym C}$ be a function of the form
$$
f((a_1,a_2,\ldots ,a_n))=\prod_{j=1}^n\hat f(j)^{a_j},
$$
where $\hat f(j)\in {\sym C}$, and we set $0^0=1$.
 Then we have the identity

\begin{equation}
\label{identity_for_Mn}
\sum_{n=0}^{\infty}{\bf M_n}f(\xi)z^n=
\prod_{j=0}^{\infty}{\left(1+\sum_{n=1}^{\infty}
{{\left( {z\over q}\right)}^{nj}\hat f(j)^n}\right) }^{I_j}=
 \prod_{j \geqslant 1}{{{\left( 1-\hat f(j){\left({z \over q} \right)}^j
 \right)}^{-I_j}}},
\end{equation}
where
$$
{\bf M_n}f(\xi)={1\over {q^n}}\sum_{P\in E_n}f(\xi(P)), \quad n\geqslant 1,
$$
and ${\bf M_0}f(\xi)=1$; here
$\xi=\xi(P)=(\xi_1(P),\xi_2(P),\ldots ,\xi_n(P))$, $\xi_k(P)$
denotes the number of normed prime polynomials of degree $k$
in the canonical decomposition of $P$,
and $I_n$ is the number of prime polynomials in $E_n$.
Relation (\ref{identity_for_Mn}) can be obtained by calculating the
coefficient at $z^n$ in the Taylor expansion of the infinite product on the
right-hand side of (\ref{identity_for_Mn}) %(see, e.g., \cite{knopfmacher}).
Putting here
 $\hat f(j)\equiv 1$, we obtain the well-known relation
(see, e.g., \cite{nicolas_poly})

 \begin{equation}
\label{product_for_In}
\prod_{n \geqslant 1}{{{\left( 1-{\left({z \over q} \right)}^n
 \right)}^{-I_n}}}={1 \over {1-z}},
\end{equation}
from which it follows, in particular, that

\begin{equation}
\label{assympt_In}
I_n={{q^n}\over n}+A_n {{q^{n/2}}\over n},
\end{equation}
where $-2\leqslant A_n\leqslant 0$.
Putting, in (\ref{identity_for_Mn}),
$$
\hat f(j)=\begin{cases}
{\rm e}^{it},&\hbox{\rm if }
j=k, \\
1, &\hbox{\rm if } j\not=k,
\end{cases}
$$
and using (\ref{product_for_In}), we obtain

\begin{eqnarray}
\sum_{n=0}^{\infty}{z^n}{\bf M_n}{\rm e}^{it\xi_k}
&=&
{{{\left( 1-{\left({z \over q} \right)}^k{\rm e}^{it}
 \right)}^{-I_k}}}\prod_{{\scriptstyle n\geqslant 0},\
{\scriptstyle n\not= k}}{{{\left( 1-{\left({z \over q} \right)}^n
 \right)}^{-I_n}}}\nonumber\\
& =&
{1\over {1-z}}{ {{\left( 1-{\left({z \over q} \right)}^k
 \right)}^{I_k}}\over {{\left( 1-{\left({z \over q} \right)}^k{\rm e}^{it}
 \right)}^{I_k}}  }.\nonumber
\end{eqnarray}

Differentiating the obtained formula with respect to
 $t$ and putting $t=0$, we obtain

\begin{equation}
\label{gen_func_meanxi}
\sum_{n=1}^{\infty}{z^n}{\bf M_n}\xi_k={ {z^k}\over
{(1-z)\left( {1-{\left({z \over q} \right)}^k}\right) }}
 {{I_k}\over {q^k}}.
\end{equation}
Hence, it follows that, for $k\leqslant n$, we have

 \begin{equation}
\label{meanxi_assympt}
{\bf M_n}\xi_k={{I_k}\over {q^k}}\sum_{j\colon\ 1\leqslant kj\leqslant n}{1 \over {q^{k(j-1)}}}=
{1\over k}+{\rm O}\left( {1 \over {q^{k/2}}} \right).%\eqno(9)
\end{equation}

Similarly, we obtain
$$
\sum_{m\geqslant 1}z^m{\bf M_m}\xi_k^2=
{1\over{(1-z)}}
\left[
      { {z^k}\over {\left(1-\left({{z}\over q}\right)^k\right)}  }
       {{I_k}\over {q^k}}
       +{{z^{2k}}\over {\left(1-\left({{z}\over q}\right)^k\right)}}
        {{I_k}\over {q^k}} {{I_k+1}\over {q^k}}
\right]
$$
and, for $k\leqslant n$,

\begin{equation}
\label{mean_xi_square}
{\bf M_n}\xi_k^2={1\over k}+{\rm O}\left({1\over {k^2}}\right).% \eqno(10)
\end{equation}

 Estimating the closeness of $\log P_n(\xi)$ to $\log O_n(\xi)$
we will use, as in \cite{barbour_tavare}, the formula

\begin{equation}
\label{logP_minus_logO}
\log P_n(a)-\log O_n(a)=\sum_p \sum_{s\geqslant 1}(D_{np^s}-1)^+\log p, %\eqno(11)
\end{equation}

where the sum is taken over all prime numbers,
$a\in ({\sym Z}^+)^n$, $(d-1)^+=d-1+I[d=0]$, and
$$
D_{nk}=D_{nk}(a)=\sum_{j \leqslant n \colon\  k|j}a_j,
$$

\begin{eqnarray}
\label{M_nD_nk}
{\bf M_n}D_{nk}(\xi)
&=&\sum_{j\leqslant n \colon\  k|j}
{\bf M_n}\xi_j=\sum_{j\leqslant n \colon\  k|j}\left({1\over j}+
{\rm O}\left(
{1\over {q^{j/2}}}
\right)
\right)\nonumber\\
&=&
{1\over k}\sum_{j=1}^{[n/k]}{1\over j}+{\rm O}\left(
{1\over {q^{k/2}}}
\right). %\eqno(12)
\end{eqnarray}

Let us find the generating function of
${\nu_n}(D_{nk}=0)$.
 Taking, in (\ref{identity_for_Mn}),
$$
\hat f(j)=\begin{cases}
0,&\hbox{\rm if }
k|j, \\
1, &\hbox{\rm if } k\nmid j,
\end{cases}
$$
and using (\ref{product_for_In}), we get
\begin{equation}
\label{nu_nD_nk_is_0}
\begin{split}
\sum_{n=1}^{\infty}{\nu_n}(D_{nk}=0)z^n&=
\prod_{n\colon k {\not \vert} n}
{{{\left( 1-{\left({z \over q} \right)}^n
 \right)}^{-I_n}}}={1\over {(1-z)}}\prod_{n\colon k | n}
{{\left( 1-{\left({z \over q} \right)}^n
 \right)}^{I_n}}\\
&={1\over {(1-z)} }\exp
\left\{
-\sum_{n\colon k|n}{{z^n}\over n}+
\sum_{n\colon k|n}I_n
\left[
 \log{\left( 1-{\left({z \over q}\right)}^n \right)}\right.\right.\\
&\quad +\left.\left.
{{\left({z\over q}\right)}^n}
\right]-
\sum_{n\colon k|n}{   {A_n}\over {nq^{n/2}}  }z^n
\right\}
={ { {(1-z^k)}^{1\over k} }\over {1-z} }\exp\{F_k(z)\},%\eqno(13)
\end{split}
\end{equation}

where
$$
F_k(z)=\sum_{n\colon\ k|n}I_n
\left[
 \log{\left( 1-{\left({z \over q}\right)}^n \right)}+
{{\left({z\over q}\right)}^n}
\right]-
\sum_{n\colon\ k|n}{   {A_n}\over {nq^{n/2}}  }z^n.
$$

We further  use the following notation:
if $f(z)=\sum_{n=0}^{\infty}a_nz^n$, then we denote
 $[f(z)]_{(n)}=a_n$.
We shall often use the following elementary property of this notation.
Let $f_1(z)$, $f_2(z)$, $g_1(z)$, and $g_2(z)$ be analytic functions in a
 neighborhood of zero.  If
$|[f_i(z)]_{(s)}|\leqslant [g_i(z)]_{(s)}$
for $s\geqslant 0$ and $i=1,2$, then
$$
\big|[f_1(z)f_2(z)]_{(s)}\big|\leqslant [g_1(z)g_2(z)]_{(s)}\quad
\hbox{\rm for}\
s\geqslant 0.
$$

The notation of the form $u(\ldots )\ll v(\ldots )$
means that
$u(n,q,t\ldots )={\rm O}( v(n,q,t\ldots ))$.
The constants in the symbols  ${\rm O}(\ldots )$ and $\ll $
are always assumed to be absolute and independent of $q$.

\begin{lem}
\label{coefF_k(z)}
If $k\geqslant 2$, then
$$
\big|[\exp \{ F_k (z) \} ]_{(s)}\big|\leqslant
\left[
 {\left(
 1-{\left( {z\over { {\sqrt 2}}}\right) }^{k}
        \right)}^{-1}
\right]_{(s)}, \quad  s \geqslant 0.   \eqno({\rm i})
$$
$$
\left[
(1-z^k)^{1/k}\exp \{ F_k(z) \}
 \right]_{(m)}\ll
\left[ 1+\sum_{n=1}^{\infty}{{z^{kn}}\over {kn}} \right]_{(m)},
\quad m\geqslant 0. \eqno({\rm ii})
$$
\end{lem}

\begin{proof} Using estimate (\ref{assympt_In}), we obtain
\begin{equation*}
\begin{split}
F_k(z)&=-\sum_{n\colon\ k|n}{I_n}
\sum_{j=2}^{\infty}
{1\over j}{{\left({z\over q}\right)}^{nj}}-
\sum_{n\colon\ k|n}{   {A_n}\over {nq^{n/2}}  }z^n
 \\
&=-\sum_{m=1}^{\infty}{{\left( z\over {\sqrt q} \right)}^{km}}
{1\over {km}} \left({1\over {q^{km/2}}}
\sum_{{\scriptstyle l \geqslant 1, j\geqslant 2},\ {\scriptstyle lj=m}}
{{kmI_{kl}}\over j}+A_{km} \right).
\end{split}
\end{equation*}
Since $0<I_n\leqslant {{q^n}\over n}$ and $-2\leqslant A_{j}\leqslant 0$,
we have
$$
\left| {1\over {q^{km/2}}}
\sum_{{\scriptstyle l \geqslant 1, j\geqslant 2},\ {\scriptstyle lj=m}}
{{kmI_{kl}}\over j}+A_{km}\right| \leqslant 2,
$$
whence, for  $s\geqslant 0$, we have
\begin{equation*}
\begin{split}
\big|[F_k(z)]_{(s)}\big|
&\leqslant
\left[
{2\over k}\sum_{m=1}^{\infty}{1\over {m}}{\left(
{{z}\over {\sqrt {q}}} \right) }^{km}
\right]_{(s)}
\leqslant \left[
\sum_{m=1}^{\infty}{1\over {m}}{\left(
{{z}\over {\sqrt {2}}} \right) }^{km}
\right]_{(s)}
\\
&=\left[
 \log {\left(
 1-{\left( {z\over {\sqrt {2}}}\right) }^{k}
        \right)}^{-1}
\right]_{(s)}.
\end{split}
\end{equation*}

Hence, for $s\geqslant 0$, we obtain

\begin{eqnarray}
&&\big|[\exp \{ F_k (z) \} ]_{(s)}\big|=\left| \sum_{j=0}^{\infty}
{{[{ F_k (z) }^j ]_{(s)}}\over {j!}}\right|\cr
&&\quad \leqslant
\sum_{j=0}^{\infty}
 {1\over {j!}}
\left[ \left(
 \log {\left(
 1-{\left( {z\over {\sqrt {2}}}\right) }^{k}
        \right)}^{-1} \right)^j
\right]_{(s)}
=\left[  {\left(
 1-{\left( {z\over { {\sqrt 2}}}\right) }^{k}
        \right)}^{-1} \right]_{(s)}.\nonumber
\end{eqnarray}

In the proof of estimate (ii), we shall use estimate (i).
Since
$$
\big|[(1-z)^{1/k}]_{(j)}\big|=
{1\over k}\prod_{l=2}^j\left( 1-{{1+1/k}\over l} \right)
\leqslant {1\over k}\prod_{l=2}^j\left( 1-{{1}\over l} \right)={1\over {kj}}
$$
for $j\geqslant 1$, we have that
$$
\big|[(1-z^k)^{1/k}]_{(j)}\big|\leqslant
\left[ 1+\sum_{s=1}^{\infty}{{z^{ks}}\over {ks}} \right]_{(j)}
$$
for  $j\geqslant 0$. Using this estimate and $(i)$, we obtain
\begin{eqnarray}
&&\left| \left[
(1-z^k)^{1/k}\exp \{ F_k(z) \}
 \right]_{(m)}\right| \leqslant {\left[
\left( 1+\sum_{s=1}^{\infty} {{z^{ks}}\over {ks}} \right)
\left( \sum_{s=0}^{\infty} {{z^{ks}}\over {2^{ks/2}}} \right)
\right] }_{(m)}\nonumber\\
&&\quad =\left[ 1+ \sum_{n=1}^{\infty}z^{kn}\left( {1\over {2^{kn/2}} }+
{1\over k}\sum_{s=1}^{n}{1\over {s2^{k(n-s)/2}}}\right) \right]_{(m)}
\ll \left[ 1+\sum_{n=1}^{\infty}{{z^{kn}}\over {kn}} \right]_{(m)},\nonumber
\end{eqnarray}
and the lemma is proved.
\end{proof}
\begin{lem}
 \label{nu_nD_nk_iszero}
 If $k\geqslant 2$, then
$$
{\nu_n}(D_{nk}(\xi)=0)=\exp \left\{ {-{1\over k}\sum_{j=1}^{[n/k]}{1\over j} }
\right\} +{\rm O}\left(
{1\over {k^2}} \right).
$$
\end{lem}

\begin{proof} In \cite{erdos_turan_II}, an exact expression has been obtained for

\begin{equation*}
\begin{split}
&\left[{   { {(1-z^k)}^{1/ k} }\over {1-z}   }\right]_{(n)}=
\sum_{j \leqslant n}\left[   { {(1-z^k)}^{1/ k} }   \right]_{(j)}=
\sum_{j\colon\ jk\leqslant n}\left[ { {(1-z)}^{1/ k} }\right]_{(j)}
\\
&\quad =\left[{   { {(1-z)}^{1/ k} }\over {1-z}   }\right]_{([n/k])}=
\prod_{j=1}^{[n/k]}\left( 1- {1 \over {jk} } \right)=
\exp \left\{ {-{1\over k}\sum_{j=1}^{[n/k]}{1\over j}} \right\}
\left( 1+{\rm O}\left( {1\over {k^2}}\right) \right).
\end{split}
\end{equation*}
Hence, using (\ref{nu_nD_nk_is_0}) and Lemma \ref{coefF_k(z)}, we obtain
\begin{equation*}
\begin{split}
{\nu_n}(D_{nk}(\xi)=0)&=\left[
{ { {(1-z^k)}^{1/k} }\over {1-z} }\exp\{F_k(z)\} \right]_{(n)}=
\left[
{ { {(1-z^k)}^{1/ k} }\over {1-z} } \right]_{(n)}
\\
&\quad +\sum_{m=1}^n
\left[
{ { {(1-z^k)}^{1/ k} }\over {1-z} }\right]_{(n-m)}
\left[ \exp\{F_k(z)\} \right]_{(m)}
\\
&=
\left[
{ { {(1-z^k)}^{1/k} }\over {1-z} } \right]_{(n)}+{\rm O}\left(
{1\over {2^{k/2}}} \right)
=\exp \left\{ {-{1\over k}\sum_{j=1}^{[n/k]}{1\over j} }\right\} +{\rm O}\left(
{1\over {k^2}} \right).
\end{split}
\end{equation*}
The lemma is proved.
\end{proof}

\begin{proof}[Proof of theorem \ref{meanoxi}]
  From  (\ref{mean_xi_square}) and Lemma \ref{nu_nD_nk_iszero} it follows that
$$
\mu_n:={\bf M_n}\big(\log P_n(\xi)-\log O_n(\xi)\big)=
\sum_p \sum_{s\geqslant 0}\big({\bf
M_n}D_{np^s}-1+{\nu_n}(D_{np^s}=0)\big)\log p,
$$

\begin{eqnarray}
\mu_n&=&\sum_{m\leqslant n}\Lambda(m)
\left[
\sum_{j=1}^{[n/m]}{1\over {jm}}
-\left( 1-\er^{-\sum_{j=1}^{[n/m]}{1\over {jm}}} \right)
\right]
+{\rm O}(1)\cr
&=&\sum_{m\leqslant n}\Lambda(m)
\bigg[
{
{\log {n\over {m}}+\gamma +{\rm O}\left( {{m}\over n} \right) }
\over {m}
}\cr
&&\quad -\left.\left( 1-{\rm e}^{ -{{\log {n\over {m}} }\over {m}} }
\left( 1-{{\gamma}\over {m}} +
{\rm O}\left( {1\over {m^{2}}}+{1\over n} \right)
\right)
\right)
\right]
+{\rm O}(1)\cr
&=&
\sum_{m\leqslant n}\Lambda(m)
\left[
{
{\log {n\over {m}} }
\over {m}
}
-\left( 1-\exp \left\{ { -{{\log {n\over {m}} }\over {m}} }\right\}
\right)
\right]\cr
&&\quad +
\gamma \sum_{m\leqslant n}\Lambda(m)
% \sum_{1\leqslant s \leqslant {{\log n}\over {\log p}}}
{{ 1-\exp \left\{ {-{{\log {n\over {m}} }\over {m}} }\right\} }\over {m}}
+{\rm O}(1),\nonumber
\end{eqnarray}
where $\gamma $ is the Euler constant, and $\Lambda(m)$  is the Mangoldt
function.
We further denote by $B$ bounded constants (not necessarily the
same in different places). Since
\begin{multline*}
\sum_{m\leqslant n}\Lambda(m)
{{ 1-\exp\left\{ { -{{\log {n\over {m}} }\over {m}} }\right\} }\over {m}}
=
\sum_{m\leqslant \log n} \Lambda(m)
{{ 1-\exp \left\{ { -{{\log {n\over m} }\over m} }\right\} }\over m}
\\
\quad +B\sum_{m> \log n} {{\Lambda(m)} \over {m^2}} \log n+{\rm O}(1)
={\rm O}(\log \log n),
\end{multline*}
we have
\begin{equation*}
\begin{split}
\mu_n&=\sum_{m\leqslant n}\Lambda(m)
\left[
{
{\log {n\over {m}} }
\over {m}
}
-\left(
1-\exp\left\{   -{{\log {n\over {m}} }\over {m}}  \right\}
\right)
\right]+{\rm O}(\log \log n)
\\
&=\sum_{m\leqslant n}\Lambda(m)
\left[
{ {\log n}\over {m} }-
{ {\log {m}}\over {m} }
-\left( 1-\exp\left\{ { -{{\log n }\over {m}} }\right\}\right.\right.
\\
&\quad \times\left.\left.\left( 1+{ {\log {m}}\over {m} }+
{\rm O}\left( { {\log {m}}\over {m} } \right)^2\right)
\right)
\right]
+{\rm O}(\log \log n)
\\
&=
\sum_{m\leqslant n}\Lambda(m)
%\sum_{1\leqslant s \leqslant {{\log n}\over {\log p}}}
\left[
{ {\log n}\over {m} }
-\left( 1-\exp \left\{ { -{{\log n }\over {m}} }\right\}
\right)
\right]
\\
&\quad +B
\sum_{m\leqslant n}\Lambda(m)
% \sum_{1\leqslant s \leqslant {{\log n}\over {\log p}}}
{ {\log {m}}\over {m} }
\left( 1-\exp \left\{ { -{{\log n }\over {m}} }\right\}
\right)+{\rm O}(\log \log n)
\\
&=S(x)+{\rm O}\bigl( (\log x)^2 \bigr) ,
\end{split}
\end{equation*}
where
$$
S(x)=\sum_{m=1}^{\infty}\Lambda(m)\phi \left( {m\over x} \right),
$$
$\phi (y)={\rm e}^{-{1\over y}}-1+{1\over y}$, and $x=\log n$.
As in the proof of one formula of Linnik (see, e.g., \cite{karatsuba}, p.~83,
Problem 6),
 we  represent $S(x)$ by a sum over nontrivial zeros of  $\zeta(s)$.
 Calculating the Mellin transform of the function $\phi (y)$, we obtain
$$
\int\limits_{0}^{\infty}\phi (x)x^{s-1}{\rm d}x=\Gamma (-s)
$$
for $1<\Re s <2$; here $\Gamma (s)$ is the Euler Gamma function.
Then we have
$$
{1\over {2\pi i}}\int\limits_{\sigma -i \infty}^{\sigma +i \infty}
\left( -{{\zeta'(s)}\over {\zeta(s)}} \right) \Gamma(-s)x^s{\rm d}s=
\sum_{m=1}^{\infty}\Lambda(m){1\over {2\pi i}}
\int\limits_{\sigma -i\infty}^{\sigma +i \infty}\Gamma(-s)
\left( {m\over x} \right)^{-s}{\rm d}s=S(x)
$$
for $1< \sigma <2$.
Let us change the integration line $\Re s=\sigma$
in the above integral by the line $\Re s=-{1\over 2}$ in the following way.
Let us consider the integral over the rectangle with vertices
 $-{1\over 2}+ iT$, $-{1\over 2}- iT$, $\sigma + iT$, and $\sigma - iT$
with  $T$ chosen so that the distance from the closest zeros of $\zeta(s)$
is $\gg {1\over {\log T}}$.
Letting now  $T$ to infinity and applying the well-known estimates of
 ${{\zeta'(s)}\over {\zeta(s)}}$ in the critical strip,
 we see that the integrals over $[-{1\over 2}+iT, \sigma +i T]$ and
$[-{1\over 2}-iT, \sigma -i T]$ tend to zero as $T\to\infty$.
Applying now the residue theorem, we obtain

\begin{equation*}
\begin{split}
S(x)&={1\over {2\pi i}}\int\limits_{\sigma -i \infty}^{\sigma +i \infty}
\left( -{{\zeta'(s)}\over {\zeta(s)}} \right) \Gamma(-s)x^s{\rm d}s=
{1\over {2\pi i}}\int\limits_{-{1\over 2} -i \infty}^{-{1\over 2} +i \infty}
\left( -{{\zeta'(s)}\over {\zeta(s)}} \right) \Gamma(-s)x^s{\rm d}s
\\
&\quad +\sum_{\rho}res_{s=\rho}
\left( -{{\zeta'(s)}\over {\zeta(s)}} \right) \Gamma(-s)x^s
+\left( res_{s=0}+res_{s=1} \right)
\left( -{{\zeta'(s)}\over {\zeta(s)}} \right) \Gamma(-s)x^s
\\
&=x(\log x -1)+{{\zeta'(0)}\over {\zeta(0)}}
-\sum_{\rho}\Gamma(-\rho)x^{\rho}+{\rm O}\left( {1\over {\sqrt x}}\right) ,
\end{split}
\end{equation*}
where $\sum_{\rho}$ is the sum over nontrivial zeros of the Riemann
zeta-function. Now using estimate~(\ref{gen_func_meanxi}), we
calculate
$$
{\bf M_n}\log P_n(\xi(P))=\sum_{j=1}^n{\bf M_n}\xi_j\log j=\sum_{j=1}^n
{{\log j}\over j}+{\rm O}(1)={1\over 2}\log^2 n + {\rm O}(1).
$$
The theorem is proved.
\end{proof}

We now estimate the sum over the nontrivial zeros of the Riemann
zeta-function.
 Since $\Re \rho \leqslant 1-{c\over {\log T}}$
for ${|\Im \rho | \leqslant T}$ (see \cite{karatsuba} or \cite{tenenbaum}) and
$|\Gamma(-(\sigma +it))|\ll \er^{-{{\pi} \over 2}|t|}|t|^{-1/2}$
 uniformly in $0\leqslant \sigma \leqslant 1$,
putting $T=\log \log n$, we get

\begin{equation*}
\begin{split}
R_n&:=\left| \sum_{\rho}\Gamma(-\rho)(\log n)^{\rho}\right| \ll
(\log n)^{1-{c\over {\log T}}}
\left| \sum_{|\Im \rho |\leqslant T}\Gamma(-\rho)\right| +
\log n \sum_{|\Im \rho |> T} |\Gamma(-\rho)|
\\
&\ll
\log n\left( (\log n)^{-{c\over {\log T}}} + {\rm e}^{-T} \right)
\ll (\log n) {\rm e}^{-c{{\log \log n}\over {\log \log \log n}}};
\end{split}
\end{equation*}
here we used the well-known fact that
$\sum_{n<|\Im \rho |\leqslant n+1}1\ll \log n$.

To estimate the covariances of $(D_{nk}-1)^+$
the following elementary lemma will be useful.

\begin{lem}
\label{coeff(z)g(z)}
  Let $f(z)$ and $g(z)$ be analytic functions. Then
$$
\left[
{{f(z)g(z)}\over {1-z}}
\right]_{(n)}
-\left[
{f(z)\over {1-z}}
\right]_{(n)}
\left[
{g(z)\over {1-z}}
\right]_{(n)}=
-\sum_{{\scriptstyle {1\leqslant i,j \leqslant n}}\atop{\scriptstyle {i+j>n}}}
[f(z)]_{(i)}[g(z)]_{(j)}.
$$
\end{lem}

 In the following two lemmas, we prove the same estimates for
 the covariances of  $(D_{nk}(\xi(P))-1)^+$
as those obtained in  \cite{arat_bar_tav} for $(D_{nk}(Z)-1)^+$,
where $Z$ is a vector with components that are
independent random Poisson variables with
parameters ${\theta / j}$.

\begin{lem}
\label{covD_nk+}
For $(k,l)=1$, $k,l\geqslant 2$, we have the estimate
$$
{\rm cov}\bigl( (D_{nk}(\xi)-1)^+,(D_{nl}(\xi)-1)^+\bigr)
\ll {{\log n}\over {kl}}.
$$
\end{lem}

\begin{proof} Since $(d-1)^+=d-1+I(d=0)$, we have
\begin{multline*}
{\rm cov}((D_{nk}-1)^+,\ (D_{nl}-1)^+)={\rm cov}(D_{nk},D_{nl})+
{\rm cov}(D_{nk},I[D_{nl}=0])
\\+
{\rm cov}(I[D_{nk}=0],D_{nl}=0)+{\rm cov}(I[D_{nk}=0],I[D_{nl}=0]).
\end{multline*}
Putting, in (\ref{identity_for_Mn}),
$$
\hat f(m)=\begin{cases}
0,&\text{if
 $k|m$   or\ $l|m$,} \cr
1, &\text{otherwise},
\end{cases}
$$
and using (\ref{identity_for_Mn}), we get
\begin{equation*}
\begin{split}
U_{k,l}(z)&:=\sum_{n=0}^{\infty}
{\nu_n}\big(D_{nk}(\xi)=0,D_{nl}(\xi)=0\big)z^n=
 {1\over {1-z}}\prod_{n\geqslant 1\colon\  k|n \vee l|n}{\left( 1-{ \left({z\over q}\right)}^n
\right)}^{I_n}
\\
&={1\over {1-z}}{{\prod_{n\geqslant 1\colon\ k|n}{\left( 1-{ \left({z\over q}\right)}^n
\right)}^{I_n} \prod_{n\geqslant 1\colon\ l|n}{\left( 1-{ \left({z\over q}\right)}^n
\right)}^{I_n} }\over {\prod_{n\geqslant 1\colon\ kl|n}{\left( 1-{ \left({z\over q}\right)}^n
\right)}^{I_n} }}
\\
&={1\over {1-z}}
{  {(1-z^k)^{1/k}(1-z^l)^{1/l}}\over { (1-z^{kl})^{1/kl} } }
\exp \{ F_k(z)+F_l(z)-F_{kl}(z) \}.
\end{split}
\end{equation*}
Hence, we have

\begin{eqnarray}
&&{\nu_n}\big(D_{nk}(\xi)=0,D_{nl}(\xi)=0\big)=\left[
{  {(1-z^k)^{1/k}(1-z^l)^{1/l}}\over { (1-z)} } \exp \big\{ F_k(z) +F_l(z) \big\}
\right]_{(n)}
\cr
&&\qquad +\left[
U_{k,l}(z) \big(1- \exp \{ F_{kl}(z) \} \big)
\right]_{(n)}+
\left[
U_{k,l}(z) \exp \{ F_{kl}(z) \}\big(1-(1-z^{kl})^{1/kl}\big)
%{  {(1-z^k)^{1/k}(1-z^l)^{1/l}}\over { (1-z)(1-z^{kl})^{1/kl} } }
% \exp \{ F_k(z) +F_l(z) \}(1-(1-z^{kl})^{1/kl})
\right]_{(n)}
\cr
&&\quad =[S_1(z)+S_2(z)+S_3(z)]_{(n)}.\nonumber
\end{eqnarray}

Since  $0\leqslant [U_{k,l}(z)]_{(j)}\leqslant [ {1\over {1-z}}
]_{(j)}$  and, by Lemma \ref{coefF_k(z)},
$|[1- \exp \{ F_{kl}(z) \} ]_{(j)}|\leqslant [
 {(1-{( {z/{ {\sqrt 2}}}) }^{kl})}^{-1}-1]_{(j)}$
 for $j\geqslant 0$, we have

\begin{eqnarray}
\big|[S_2(z)]_{(n)}\big|&\leqslant&
\left[
{1\over {1-z}}
 \left( {\left(
 1-{\left( {z\over { {\sqrt 2}}}\right) }^{k}
        \right)}^{-1}-1 \right)
\right]_{(n)}\cr
&=&
\left[
{1\over {1-z} }\left(
\sum_{j\geqslant 1}{ \left( {z\over {\sqrt 2}} \right)}^{klj}
\right)
\right]_{(n)}\ll {1\over{ {2}^{kl/2}}}.\nonumber
\end{eqnarray}
Similarly, since  $|[(1-z^{kl})^{1\over {kl}}-1]_{(j)}|\leqslant [
\sum_{s=1}^{\infty}{{z^{kls}}\over {kls}} ]_{(j)}$
for $j\geqslant 0$, we have
$$
\big|[S_3(z)]_{(n)}\big|\leqslant \left[
{1\over {1-z} }
{\left(
 1-{\left( {z\over { {\sqrt 2}}}\right) }^{k}
        \right)}^{-1}
\left(
\sum_{j\geqslant 1}{{z^{klj}}\over klj}
\right) \right]_{(n)}
\ll {{\log n}\over {kl}}.
$$

Applying the estimates obtained and using (\ref{M_nD_nk}), we get

\begin{eqnarray}
&&{\rm cov}\big(I[D_{nk}(\xi)=0],I[D_{nl}(\xi)=0]\big)\cr
&&\quad =\left[
{  { (1-z^k)^{1/k}\exp \{ F_k(z) \} (1-z^l)^{1/l}\exp \{ F_l(z) \} }
\over {1-z} }                \right]_{(n)}
\cr
&&\qquad -\left[
{  { (1-z^k)^{1/k}\exp \{ F_k(z) \} }
\over {1-z} }                \right]_{(n)}
\left[
{  { (1-z^l)^{1/l}\exp \{ F_l(z) \} }
\over {1-z} }                \right]_{(n)}
+{\rm O}\left( { {\log n} \over {kl} } \right).\nonumber
\end{eqnarray}

Applying estimate (ii)  of Lemma \ref{coefF_k(z)} and Lemma \ref{coeff(z)g(z)}, we obtain
\begin{eqnarray}
&&{\rm cov}\big(I[D_{nk}(\xi)=0],I[D_{nl}(\xi)=0]\big)
\cr
&&\quad =-
\sum_{  {\scriptstyle i+j>n}\atop {\scriptstyle 1\leqslant i,j \leqslant n}}
\left[
(1-z^k)^{1/k}\exp \{ F_k(z) \}
 \right]_{(i)}\left[
(1-z^l)^{1/l}\exp \{ F_l(z) \}
 \right]_{(j)}+
{\rm O}\left( { {\log n} \over {kl} } \right)\cr
&&\quad \ll
\sum_{  {\scriptstyle i+j>n}\atop {\scriptstyle 1\leqslant i,j \leqslant n}}
\left[
\left( 1+\sum_{s=1}^{\infty} {{z^{ks}}\over {ks}} \right)
 \right]_{(i)}
\left[
\left( 1+\sum_{s=1}^{\infty} {{z^{ls}}\over {ls}} \right)
 \right]_{(j)}
 +{{\log n} \over {kl}}\cr
&&\quad =
\sum_{  {\scriptstyle i,j\colon\  ik+jl>n}\atop {\scriptstyle 1\leqslant ik,jl \leqslant n}}
\left[
\left( 1+\sum_{s=1}^{\infty} {{z^s}\over {ks}} \right)
 \right]_{(i)}
\left[
\left( 1+\sum_{s=1}^{\infty} {{z^s}\over {ls}} \right)
 \right]_{(j)}
+{  {\log n}\over {kl} }.\nonumber
\end{eqnarray}

Therefore, we have
$$
{\rm cov}\big(I[D_{nk}(\xi)=0],I[D_{nl}(\xi)=0]\big)\ll { {\log n}\over {kl}}+
\sum_{  {\scriptstyle i,j\colon\  ik+jl>n}\atop {\scriptstyle 1\leqslant ik,jl \leqslant n}}
{1\over {kilj}}.
$$

Since
\begin{equation*}
\begin{split}
{1\over {kl}}\sum_{  {\scriptstyle ik+jl>n}\atop {\scriptstyle 1\leqslant ik,jl \leqslant n}}
{1\over {ij}}&=
{1\over {kl}}\sum_{i=1}^{[n/k]}{1\over i}
\sum_{{{n-ki}\over l}<j \leqslant {n \over l}}{1\over j}
\\
&={1\over {kl}}
\sum_{1\leqslant i \leqslant {n\over {2k}}}{1\over i}
\sum_{{{n-ki}\over l}<j \leqslant {n \over l}}{1\over j}
 +{1\over {kl}}\sum_{{n\over {2k}} \leqslant i \leqslant {n\over k} }{1\over i}
\sum_{{{n-ki}\over l}<j \leqslant {n \over l}}{1\over j}
\\
&\ll {1\over {kl}}
\sum_{1\leqslant i \leqslant {n\over {2k}}}{1\over i}
+{1\over {kl}}\sum_{{n\over {2k}}
\leqslant i \leqslant {n\over k} }{{\log n}\over i}
\ll { {\log n}\over {kl}},
\end{split}
\end{equation*}
we finally have
\begin{equation}
\label{covI_nk_is_0andI_nk_is_0}
{\rm cov}\big(I[D_{nk}(\xi)=0],I[D_{nl}(\xi)=0]\big)
\ll { {\log n}\over {kl}}.%\eqno(14)
\end{equation}

To estimate ${\rm cov}(I[D_{nk}(\xi)=0],\xi_j)$
we consider two cases: $k \nmid j$ and $k|j$.

1) If $k \nmid j$, then
$$
{\rm cov}\big(I[D_{nk}(\xi)=0],\xi_j\big)
={\bf M_n}\big(\xi_j I[D_{nk}(\xi)=0]\big)-{\bf M_n}\xi_j
{\nu_n}(D_{nk}(\xi)=0).
$$
Putting, in $(\ref{identity_for_Mn})$,
$$
\hat f(s)=\begin{cases} \er^{it},&\text{for $s=j$}, \cr 0, &\text{if
$k|s$}, \cr 1,&\text{for remaining  $s$}
\end{cases}
$$
we obtain
$$
\sum_{n\geqslant 0}{\bf M_n}\big({\rm e}^{it\xi_j}I[D_{nk}(\xi)=0]\big)z^n
={{(1-z^k)^{1\over k}}\over {1-z}}
\exp \{ F_k(z) \}
{  {\left( {1-{\left({z \over q} \right)}^j} \right)^{I_j}}
\over {\left( {1-{\left({z \over q} \right)}^j}{\rm e}^{it} \right)^{I_j}}  }.
$$
Differentiating the obtained equality with respect to $t$
and putting $t=0$, we have

\begin{equation}
\label{M_nxi_jID_nk}
\sum_{n\geqslant 0}{\bf M_n}\big(\xi_jI[D_{nk}(\xi)=0]\big)z^n=
{{(1-z^k)^{1\over k}}\over {1-z}} \exp \{ F_k(z) \}
{   {z^j} \over {1-{\left({z \over q} \right)^j} }  }{ I_j \over {q^j} }.
%\eqno(15)
\end{equation}

Hence, applying Lemma \ref{coeff(z)g(z)} and (\ref{gen_func_meanxi}), (\ref{nu_nD_nk_is_0}), and (\ref{M_nxi_jID_nk}), we have

\begin{eqnarray}
\label{covI_D_nk_and_xi_j}
&&{\rm cov}\big(I[D_{nk}(\xi)=0],\xi_j\big)
=\left[
{{(1-z^k)^{1\over k}}\over {1-z}} \exp \{ F_k(z) \}
{   {z^j} \over {1-{\left({z \over q} \right)^j} }  }{ I_j \over {q^j} }
\right]_{(n)}
\cr
&&\qquad -\left[  {1\over {1-z}}{   {z^j} \over {1-{\left({z \over q} \right)^j} }  }
{ I_j \over {q^j} }\right]_{(n)}
\left[
      {{(1-z^k)^{1\over k}}\over {1-z}} \exp \{ F_k(z) \}
\right]_{(n)}
\cr
&&\quad =-\sum_{  {\scriptstyle r+s>n}\atop {\scriptstyle 1\leqslant r,s \leqslant n}}
\left[
{   {z^j} \over {1-{\left({z \over q} \right)^j} }  }{ I_j \over {q^j} }
\right]_{(r)}
\left[
{(1-z^k)^{1\over k}}
\exp \{ F_k(z) \}
\right]_{(s)}
\cr
&&\quad \ll \sum_{  {\scriptstyle s,r\colon\ rj+sk>n}\atop {\scriptstyle 1\leqslant rj,sk \leqslant n}}
{1\over j}\left[ {   z \over {\left(1-{z \over{ q^j}} \right)} }
\right]_{(r)}
\left[
\left(
1+{1\over k}\sum_{m=1}^{\infty}{{z^m}\over m}
\right)
%\left(
%\sum_{m=0}^{\infty}{{z^m}\over {2^{mk/2}}}
%\right)
\right]_{(s)}
\cr
&&\quad ={1\over {jk}}
\sum_{  {\scriptstyle s,r\colon\ rj+sk>n}\atop {\scriptstyle 1\leqslant rj,sk \leqslant n}}
{1\over {q^{j(r-1)}s}}=
{1\over {jk}}\sum_{   {{n-j}\over k } <s\leqslant {n\over k}  }
{1\over s}+
{1\over {jk}}
\sum_{  {\scriptstyle rj+sk>n}\atop {\scriptstyle 1\leqslant rj,sk \leqslant n, r\geqslant
2}}
{1\over {q^{j(r-1)}s}}
\cr
&&\quad \ll {1\over {jk}}\left(
1+\log {n\over {n-j+1} }
\right)+{ {\log n } \over {jq^jk} }.
\end{eqnarray}
    %\eqno(16)

2) Now suppose that $k|j$.
Since $I[D_{nk}(\xi)=0]\xi_j\equiv 0$, in this case,  we have

\begin{equation}
\label{cov_xi_jD_nk}
{\rm cov}\big(I[D_{nk}(\xi)=0]\xi_j\big)
=-{\nu_n}(D_{nk}(\xi)=0){\bf M_n}\xi_j\ll {1\over
j}.%\eqno(17)
\end{equation}

Now we can estimate
 ${\rm cov}(I[D_{nk}(\xi)=0],D_{nl}(\xi))$:
\begin{eqnarray}
&&{\rm cov}\big(I[D_{nk}(\xi)=0],D_{nl}(\xi)\big)\cr
&&\quad =\sum_{j\leqslant n\colon\ l|j,k|j}
{\rm cov}\big(I[D_{nk}(\xi)=0],\xi_j\big)+
\sum_{j\leqslant n\colon\ l|j,k \not| j}{\rm
cov}\big(I[D_{nk}(\xi)=0],\xi_j\big)
\cr
&&\quad \ll \sum_{j\leqslant n\colon\ kl|j}{1\over j}+
\sum_{j\leqslant n\colon\ l|j}\left(
{1\over {jk}}\left(
1+\log {n\over {n-j+1} }
\right)+{ {\log n } \over {jq^jk} }
\right)
\cr
&&\quad \ll {{\log n}\over {kl}}
+\sum_{j\leqslant n\colon\ l|j}{1\over {kj}}\log {n\over {n-j+1} }
\cr
&&\quad \ll {{\log n}\over {kl}}
+\sum_{j\leqslant n/2\colon\ l|j}{1\over {kj}}{j\over n}
+\sum_{n/2<j\leqslant n\colon\ l|j}{1\over {kn}}\log n\ll {{\log n}\over {kl}};\nonumber
\end{eqnarray}

here we used estimates (\ref{covI_D_nk_and_xi_j}) and (\ref{cov_xi_jD_nk}).

Let  $i \neq j$. Then, as before, putting, in (\ref{identity_for_Mn}),
 $\hat f(i)={\rm e}^{it_1}$,
$\hat f(j)={\rm e}^{it_2}$, and $\hat f(m)=1 $
for the remaining  $m$, we obtain the generating function of
 ${\bf M_n}{\rm e}^{it_1\xi_i+it_2\xi_j}$.
Differentiating the obtained formula with respect to
 $t_1$ and $t_2$ and putting $t_1=t_2=0$, we get
$$
\sum_{n=1}^{\infty}{\bf M_n}\xi_i\xi_jz^n={ I_i \over {q^i} }{ I_j \over {q^j} }
{1\over {1-z}}
{   {z^{i+j}} \over {\left(1-{\left({z \over q} \right)^i}\right)
{\left(1-{\left({z \over q} \right)^j}\right) }  }  }.
$$
   From this and from (\ref{gen_func_meanxi}), applying Lemma \ref{coeff(z)g(z)}, we have

\begin{eqnarray}
{\rm cov}(\xi_i,\xi_j)
&=&\left[
{ I_i \over {q^i} }{ I_j \over {q^j} }
{1\over {1-z}}
{   {z^i} \over {1-{\left({z \over q} \right)^i} }  }
{   {z^j} \over {1-{\left({z \over q} \right)^j} }  }
                 \right]_{(n)}
\cr
&&\quad              - \left[
{ I_i \over {q^i} }
{1\over {1-z}}
{   {z^i} \over {1-{\left({z \over q} \right)^i} }  }
              \right]_{(n)}
                  \left[
                { I_j \over {q^j} }
{1\over {1-z} }
{   {z^j} \over {1-{\left({z \over q} \right)^j} }  }
                 \right]_{(n)}\cr
&=&
-{ I_i \over {q^i} }{ I_j \over {q^j} }
\sum_{  {\scriptstyle is+jr>n}\atop {\scriptstyle 1\leqslant is,rj \leqslant n}}
{1\over {q^{i(s-1)}q^{j(r-1)}}}.\nonumber
\end{eqnarray}

Hence, it follows that, for  $i\neq j$, we have

\begin{equation}
\label{cov_xi_ixi_j}{\rm cov}(\xi_i,\xi_j)\ll
\begin{cases}
{1\over {ij}}{1\over {q^{\min \{ i,j \} }}},&\hbox{\rm if }
i+j\leqslant n, \cr
{1\over {ij}}, &\hbox{\rm if } i+j>n.% \cr}\eqno(18)
\end{cases}
\end{equation}

For $i=j$, by (\ref{mean_xi_square}) we have ${\rm cov}(\xi_i,\xi_i)=D\xi_i\ll {1/i}$.

Applying the estimates obtained, we have
\begin{eqnarray}
\label{cov_D_nkD_nl}
&&{\rm cov}\big(D_{nk}(\xi),D_{nl}(\xi)\big)\cr
&&\quad =
\sum_{ {\scriptstyle 1\leqslant i,j\leqslant n}
\atop { {\scriptstyle i+j\leqslant n}
\atop {\scriptstyle k|i, l|j, i\not= j}}}
 {\rm cov}(\xi_i,\xi_j)+
\sum_{ {\scriptstyle 1\leqslant i,j\leqslant n}
\atop {{\scriptstyle i+j> n}
\atop {\scriptstyle k|i, l|j, i\not= j}} }
 {\rm cov}(\xi_i,\xi_j)+
\sum_{  {\scriptstyle kl|j}\atop {\scriptstyle 1\leqslant j\leqslant n}}
 {\rm cov}(\xi_i,\xi_j)\cr
&&\quad \ll \sum_{   ki+lj\leqslant n}{1\over {kilj}}{1\over {q^{\min \{ ki,lj \} }}}+
\sum_{  {\scriptstyle ki+lj>n}\atop {\scriptstyle 1\leqslant ki,lj \leqslant n}}
{1\over {kilj}}
+\sum_{j=1}^{\left[ {n/ {kl}} \right]}
{1\over {jkl}}\ll {{\log n}\over {kl}}.
\end{eqnarray}
%\eqno(19)

Using estimates (\ref{covI_nk_is_0andI_nk_is_0}), (\ref{covI_D_nk_and_xi_j}), (\ref{cov_xi_jD_nk}), and (\ref{cov_D_nkD_nl}), we obtain the required
estimate
$$
{\rm cov}\bigl( (D_{nk}(\xi)-1)^+,(D_{nl}(\xi)-1)^+\bigr)
\ll {{\log n}\over {kl}}.
$$
The lemma is proved.
\end{proof}

\begin{lem}
\label{covD_np^s}
Let $p$ be a prime number, and let $1\leqslant s\leqslant t$.
Then we have
$$
{\rm cov}\bigl( (D_{np^s}(\xi)-1)^+,(D_{np^t}(\xi)-1)^+\bigr)
\ll {{\log n}\over {p^t}}.
$$
\end{lem}

\begin{proof} Applying Lemma \ref{nu_nD_nk_iszero}, we have
\begin{eqnarray}
&&{\rm cov}\big(I[D_{np^s}(\xi)=0],I[D_{np^t}(\xi)=0]\big)
\cr
&&\quad
={\nu_n}\big(D_{np^s}(\xi)=0\big)-{\nu_n}\big(D_{np^s}(\xi)=0\big){\nu_n}\big(D_{np^t}(\xi)=0\big)\cr
&&\quad =
{\nu_n}\big(D_{np^s}(\xi)=0\big)\big(1-{\nu_n}(D_{np^t}(\xi)=0)\big)\cr
&&\quad \ll
1-\exp\left\{ {-{1\over {p^t}}\sum_{j=1}^{[n/p^t]}{1\over j} }\right\}
+{1\over {p^{2t}}}\ll
{{\log n}\over {p^t}};\nonumber
\end{eqnarray}
here we have used the inequality
$1-{\rm e}^{-x}\leqslant x$ for $x\geqslant 0$.

Since, for  $s\leqslant t$, from $p^t|j$  it follows
that $I[D_{np^s}(\xi)=0]\xi_j=0$, we have
$$
{\rm cov}\big(I[D_{np^s}(\xi)=0],D_{np^t}(\xi)\big)=
-\sum_{j\leqslant n \colon\  p^t|j}{\nu_n}(D_{np^s}=0){\bf M_n}\xi_j
\ll {{\log n}\over {p^t}}.
$$

Applying estimates (\ref{covI_D_nk_and_xi_j}) and (\ref{cov_xi_jD_nk}), we have
\begin{multline*}
{\rm cov}\big(I[D_{np^t}(\xi)=0],D_{np^s}(\xi)\big)
=\sum_{j\leqslant n \colon\  p^s |j}
{\rm cov}\big(I[D_{np^t}(\xi)=0],\xi_j\big)
\\
\ll \sum_{j\leqslant n \colon\  p^t|j}{1\over j}+{{\log n}\over {p^t}}+
\sum_{j\leqslant n \colon\  p^s|j}{1\over {jp^t}}{\log {n\over {n+1-j}}}
\ll {{\log n}\over {p^t}}.\nonumber
\end{multline*}

Applying estimates (\ref{cov_xi_ixi_j}), we have
\begin{multline*}
{\rm cov}\big(D_{np^s}(\xi),D_{np^t}(\xi)\big)=
\sum_{ {\scriptstyle 1\leqslant i,j \leqslant n}\atop {\scriptstyle p^s|i,p^t|j} }
{\rm cov}(\xi_i,\xi_j)=
\sum_{ j\leqslant n \colon\ p^t|j}{\rm cov}(\xi_j,\xi_j)
\\
 +
\sum_{ {{\scriptstyle 1\leqslant i,j \leqslant n}
         \atop {\scriptstyle p^s|i,p^t|j}}
         \atop {\scriptstyle i\not= j}    }
{\rm cov}(\xi_i,\xi_j)
\ll {{\log n}\over {p^t}}+
\sum_{        {     {\scriptstyle p^s|i,p^t|j}
              \atop {\scriptstyle 1\leqslant i,j \leqslant n}    }
              \atop {\scriptstyle i\not= j}                 }
{1\over {ij}}{1\over q^{\min \{ i,j \} }}+
\sum_{         {        {\scriptstyle p^s|i,p^t|j}
                  \atop {\scriptstyle i+j>n}       }
                  \atop {\scriptstyle 1\leqslant i,j \leqslant n}       }
{1\over {ij}}\ll
{{\log n}\over {p^t} }.\nonumber
\end{multline*}

   From the estimates obtained it follows  that
$$
{\rm cov}\bigl( (D_{np^s}-1)^+,(D_{np^t}-1)^+\bigr) \ll {{\log n}\over {p^t}}.
$$
The lemma is proved.
\end{proof}

The following proposition, which is completely similar to
Proposition  2.3 of \cite{barbour_tavare}, gives the desired estimate of closeness of the
random variables $\log O_n(\xi(P)) $ and $\log P_n(\xi(P)) $.

\begin{prop}  For every fixed $K>0$, we have
$$
{\nu_n}\left(
{  { |\log P_n(\xi)-\log O_n(\xi)-\mu_n|} \over {\log^{3/2} n} } >
K\left( {{\log \log n}\over \log n}\right) ^{2/3}
\right)\ll \left( {{\log \log n}\over \log n}\right) ^{2/3},
$$
where $\mu_n={\bf M}(\log P_n(\xi)-\log O_n(\xi))$.
\end{prop}

\begin{proof} In the proof of this proposition, we repeat the arguments of
Barbour and Tavare  in the proof of the corresponding result.
We have
\begin{multline*}
\log P_n(\xi)-\log O_n(\xi)=
\left(
\sum_{p\leqslant \log^2 n }+\sum_{p>\log^2n}
\right) (D_{np}(\xi)-1)^+\log p
\\ \quad +
\sum_p\sum_{s\geqslant 2}(D_{np^s}(\xi)-1)^+\log p
=V_1+V_2+V_3.\nonumber
\end{multline*}
Applying Lemmas \ref{covD_nk+} and \ref{covD_np^s}, we have
$$
{\bf D} V_1 \ll  \sum_{p\leqslant \log^2 n}{{\log n}\over p}{\log^2 p}+
\sum_{{\scriptstyle p\not= q}\atop {\scriptstyle p,q\leqslant \log ^2 n}}
{{\log n}\over {pq}}\log p\log q\ll
\log n (\log \log n)^2.
$$
   From Lemma \ref{nu_nD_nk_iszero} it follows that
$$
{\bf M}(D_{nk}-1)^+\ll \min \left\{
\left( {{\log n}\over k}\right),\left( {{\log n}\over k}\right)^2
\right\}.
$$
Therefore,
$$
{\bf M}V_2\ll \sum_{p>\log^2n}{{\log^2 n}\over {p^2}}\log p\ll 1.
$$
We estimate
$$
{\bf D}V_3\ll \sum_p \log^2 p\sum_{s,t\geqslant 2}
{{\log n}\over {p^{\max \{s,t\} }}}+
\sum_{p\not= q}\log p \log q \sum_{s,t\geqslant 2}
{  {\log n} \over {p^sq^t}  }\ll \log n.
$$

Applying the Chebyshev inequality, we get
$$
{\nu_n}\left( {{|V_1-{\bf M}V_1|}\over {\log^{3/2} n}}>{1\over 3}K
\left( {{\log \log n}\over \log n}\right) ^{2/3} \right)
\ll \left( {{\log \log n}\over \log n}\right) ^{2/3},
$$
$$
{\nu_n}\left( {{|V_2-{\bf M}V_2|}\over {\log^{3/2} n}}>{1\over 3}K
\left( {{\log \log n}\over \log n}\right) ^{2/3} \right)
\ll  {1\over { {(\log \log n)}^{2/3} \log ^{5/6}n} },
$$
$$
{\nu_n}\left( {{|V_3-{\bf M}V_3|}\over {\log^{3/2} n}}>{1\over 3}K
\left( {{\log \log n}\over \log n}\right) ^{2/3} \right)
\ll  {1\over { {(\log \log n)}^{4/3} \log ^{2/3}n } }.
$$
Hence, we obtain the proof of the proposition.
\end{proof}

The proved proposition enables us to replace the investigation
of closeness of $\log O_n(\xi(P))$ to the normal law
by the investigation of the simpler quantity
$$
\log P_n(\xi(P))=\sum_{i=1}^n\xi_i(P)\log i,
$$
which is a linear combination of the variables
 $\xi_i(P)$.

% In the case of the group $S_n$, Manstavi\v cius~\cite{manstberry}
% proved a theorem allowing one to get an asymptotic
% expansion of the characteristic functions of  random variables
% of the form
% $a_1\alpha_1 (\sigma)+a_2\alpha_2 (\sigma)+\cdots +a_n\alpha_n(\sigma)$.
% Estimating then the closeness of the characteristic functions
%  of $\log P_n(\alpha)$ and
%  $\log P_n(\xi)$ and applying the Esseen inequality, we will obtain the
% the needed estimate.
%
%
% In the case of the group $S_n$, we have the following identity,
% which is analogous to (\ref{identity_for_Mn}):
%
% \begin{equation}
% \label{sum_M_nfzn}
% \sum_{n=0}^{\infty}{\bf M_n}f(\alpha)z^n=\exp \left\{
% \sum_{j=1}^{\infty}{{\hat f(j)}\over j}z^j \right\}, % \eqno(20)
% \end{equation}
%
% where
% $$
% {\bf M_n}f(\alpha)={1\over {n!}}\sum_{\sigma \in S_n}f(\alpha(\sigma)),\quad
% n\geqslant 1,
% $$
% and ${\bf M_0}f(\alpha)=1$.

Further we use the same notation
  ${\bf M_n}$ to denote the mean values of
$f(\alpha)$ and $f(\xi)$ on $S_n$ and $E_n$, respectively.

We  denote by $\phi_{n,\alpha}(t)$ and $\phi_{n,\xi}(t)$
the characteristic functions of the random variables
$$
{ {\log P_n(\alpha (\sigma))-{1\over 2}\log^2 n}
\over {(1/\sqrt 3 )\log^{3/2}n}}\quad \hbox{\rm and}\quad
{ {\log P_n(\xi (P))-{1\over 2}\log^2 n}
\over {(1/\sqrt 3 )\log^{3/2}n}},
$$
respectively. % We denote
% $$
% \rho (n,p)=\left( \sum_{k\leqslant n}{{|\hat f(k)-1|^p}\over
% k}\right)^{1/p}, \qquad
% L(z)=\sum_{k\leqslant n}{{\hat f(k)-1}\over k}z^k,
% $$
% $$
% I_j(n)={1\over {2\pi n i}}\int\limits_{|z|=r}{{(L(z)-L(1))^j}\over
% {(1-z)z^{n+1}}}dz, \quad j=0,1,\ldots\,.
% $$
%
% %\thm{5\ {\rm (Manstavi\v cius~\cite{manstberry})}}
%
% \begin{thm}
% \label{Manstav}
% Let $p>1$.  There exists a sufficiently small $\delta=\delta(p)$ such
% that, for $\rho:=\rho (n,p)\leqslant \delta$, we have
% $$
% {\bf M_n}f(\alpha)=\exp \{ L(1) \}\left(
%        1+\sum_{j=2}^{N-1}{{I_j(n)}\over {j!}}+{\rm O}(\rho^N+n^{-c})
%                    \right)
% $$
% for each  $N\geqslant 2$ with some fixed constant $c=c(p)>0$.
% \end{thm}
%The constant in the symbol $O$ also depends on $p$ only. % LMR nera!
%\endthm

To estimate $\phi_{n,\alpha}(t)$ we apply Theorem \ref{meanM1} with $d_j\equiv 1$ and $p=\infty$.
Putting in (\ref{manst_general})
$$
\hat f(k)=\exp \left\{ {it\log k}\over {{( {1\over3}\log^3n) }^{1/2}}
\right\},\quad  p=\infty,
$$
we get
$$
\rho=\rho(\infty)\ll {\frac{|t|}{\log^{1/2}n}}
$$
For $|t|\leqslant  \varepsilon \log^{1/2}n$, where
$\varepsilon$ is a sufficiently small fixed number,
we have $\rho \leqslant \delta$ and

\begin{eqnarray}
&&{\bf M_n}\exp\left\{ it { {\log P_n(\alpha (\sigma))}
\over {(1/\sqrt 3 )\log^{3/2}n}}\right\}\cr
&&\quad  =
\exp \left\{ S_n\left(
 {t\over { (1/ \sqrt{3})\log^{3/2}n} }
           \right) \right\} \left(
    1+{\rm O}\left({ {|t|^2}\over {\log n }  }   \right)
                    \right);\nonumber
\end{eqnarray}
here
$$
S_n(t)=\sum_{m=1}^n {{{\rm e}^{it\log m}-1}\over m}.
$$

 Let us estimate the quantity $\exp \{ S_n(t) \}$. We have

\begin{eqnarray}
S_n(t)
&=&\sum_{m=1}^n {{{\rm e}^{it\log m}-1}\over m}
=\sum_{m=1}^n {1\over {m^{1-it}} }-\log n-\gamma +
{\rm O}\left( {1\over n} \right)
\cr
&=&\zeta (1-it)+{{n^{it}}\over {it}}-\log n -\gamma +{\rm O}\left( {1\over n}\nonumber
\right)
\end{eqnarray}

for $|t|\leqslant 1$. Here we have used the formula
$$
\zeta (s)=\sum_{m\leqslant n}{1\over {m^s}}+{{n^{1-s}}\over {s-1}}+
{\rm O}\left( {{|s|}\over {\sigma}}{1\over {n^{\sigma}}} \right)
$$
for $\Re s=\sigma >0$ (see, e.g., \cite{karatsuba}).

Since $\zeta (1-it)=-{1\over {it}}+\omega (t)$, where $\omega (t)$ is
an analytic function and $\omega (0)=\gamma$, for  $|t|\leqslant 1$,
we have

\begin{equation}
\label{S_n(t)}
S_n(t)={{n^{it}-1}\over {it}}-\log n+
{\rm O}\left( {1\over n}+|t| \right).%\eqno(21)
\end{equation}

Applying the Taylor expansion for $|t|\leqslant 1$, we have

\begin{equation}
\label{S_n(t)2}
S_n(t)={1\over {it}}\left(
\sum_{k=1}^4{{(it\log n)^k}\over {k!}}+
%  it\log n+{{(it)^2}\over {2!}}+{{(it)^3}\over {3!}}+{{(it)^4}\over {4!}}+
{\rm O}(|t|^5 \log^5 n))
     \right)-
\log n +{\rm O}\left( {1\over n}+|t| \right).%\eqno(22)
\end{equation}

Let
$$
D_n(t)=S_n\left(
 {t\over { (1/ \sqrt{3})\log^{3/2}n} }
           \right)
-{it {{\sqrt 3}\over 2}\log^{1/2} n}.
$$
Then, for $|t|\leqslant \log^{3/2}n $, we have
$$
D_n(t)=-{{t^2}\over 2}+(it)^3{  {3^{3/2}}\over {4!} }{1\over {\log^{1/2}n} }
+{\rm O}\left(
{1\over n} +{{|t|^4}\over {\log n}} +{{|t|}\over {\log^{3/2}n }}
  \right).
$$

If $|t|\leqslant\log^{1/6}n$, then
\begin{multline}
\label{phi_alphan}
\phi_{\alpha,n} (t)
=\exp \{ D_n(t) \} \left(
    1+{\rm O}\left({ {|t|^2}\over {\log n }  }   \right)
                    \right)
={\rm e}^{-{{t^2}\over 2}}\left(
       1+(it)^3{{3^{3/2}}\over {4!}}{1\over {\log^{1/2}n}}\right.
       \\
+\left.{\rm O}\left(
  {1\over n^{c}}+ {{|t|^4+|t|^6}\over {\log n}}
+{{|t|}\over {\log^{3/2} n}}+{{|t|^2}\over {\log n}}
           \right)
    \right).
\end{multline}
   From (\ref{S_n(t)})  we have, for $\varepsilon_1 >0$,
\begin{multline}
\label{exp_S_n(t)}
\left| \exp \left\{ S_n\left(
 {t\over { (1/ \sqrt{3})\log^{3/2}n} }
           \right) \right\} \right|
       \\
=\exp \left\{  \left(
                 {
  { \sin \left(   { {t\sqrt{3}} \over {\log^{1/2} n} } \right) }\over
  { {t\sqrt{3}} \over {\log^{1/2} n} }
                 }
-1  \right) \log n+{\rm O}(1)
      \right\}\ll {1\over {n^{\lambda }}}
\end{multline}
for $\varepsilon_1\log^{1/2}n \leqslant t \leqslant \log^{3/2}n$;
here $\lambda=\lambda (\varepsilon_1)
=1-\sup_{u>\varepsilon_1\sqrt{3}}{{\sin u}\over u}$.

 Theorem 5 gives an estimate of the characteristic function of
 $\phi_{n,\alpha}(t)$
for $|t|\leqslant \log^{1/6}n$. We further  estimate
 $\phi_{n,\alpha}(t)$ for $|t|\leqslant \log^{3/2}n$.
The general Theorem~A of Manstavi\v cius~\cite{manstberry}   gives the desired
estimate for $|t|\leqslant \log^{1/2}n$
but its application becomes difficult in the region
$\log^{1/2}n\leqslant |t|\leqslant \log^{3/2}n$,
since, for $|t|\geqslant \log^{1/2}n$,
its remainder term has an increasing  multiplier.

\begin{thm}
\label{th_turlarge}
\begin{equation}
\label{phi_nalpha_smalt}
\phi_{n,\alpha}(t)\ll
\er^{-{{t^2}\over {4}}}\quad \hbox{for}\quad
|t|\leqslant \delta_1 \log^{1/2}n,%\eqno(25)
\end{equation}

\begin{equation}
\label{phi_nalpha_larget}
\phi_{n,\alpha}(t)\ll {1\over {n^u}}\quad \hbox{for}\quad
\delta_1\log^{1/2}n<|t|\leqslant \log^{3/2}n,
%\eqno(26)
\end{equation}

where $u$ and $\delta_1$ are some fixed positive constants.
\end{thm}

\begin{proof} From (\ref{S_n(t)})  and (\ref{phi_alphan}) for $|t|\leqslant \delta\sqrt{\log n}$ we have
\begin{equation}
\begin{split}
\label{phi_alphan}
|\phi_{\alpha,n} (t)|
&\ll|\exp \{ D_n(t) \}|=\left| \exp \left\{ S_n\left(
 {t\over { (1/ \sqrt{3})\log^{3/2}n} }
           \right) \right\} \right|
       \\
       &=\exp \left\{  \left(
                 {
  { \sin \left(   { {t\sqrt{3}} \over {\log^{1/2} n} } \right) }\over
  { {t\sqrt{3}} \over {\log^{1/2} n} }
                 }-1
  \right) \log n+{\rm O}(1)
      \right\}
\end{split}
\end{equation}
Applying here the inequality $\frac{\sin u}{u}\leqslant 1-\frac{u^2}{9}$, which is true for $|u|\leqslant 2$, we have
$$
|\phi_{\alpha,n} (t)|\ll e^{-t^2/3},
$$
for $|t|\leqslant \delta\sqrt{\log n}$.

The estimate (\ref{phi_nalpha_smalt}) is proven.

In the proof of  estimate (\ref{phi_nalpha_larget}), we  use some ideas of~\cite{manstberry} and~\cite{manstadditive}.
Let
$$
\sum_{n\geqslant 0}N_n(t)z^n
=\exp \left\{ \sum_{k=1}^n{{{\rm e}^{it\log k}}\over k}z^k
 \right\}.
$$
 Differentiating this identity with respect to $z$,
 we can easily note that $N_n$
satisfy the recurrent relation
$$
N_n(t)={1\over n}\sum_{k=1}^n \er^{it\log k}N_{n-k}(t).
$$

Applying the Cauchy inequality and the Parseval identity, we have

\begin{equation}
\label{|N_n(t)|}
|N_n(t)|\leqslant \left( {1\over n}\sum_{k=1}^{\infty} |N_k(t)|^2 \right)^{1/2}
%\cr&
\leqslant
\left(
{1\over {2\pi n}}\int\limits_{-\pi}^{\pi}
\left|
\exp \left\{ \sum_{k=1}^n{{{\rm e}^{it\log k}}\over k}{\rm e}^{ixk}
 \right\}
\right|^2\dr x
\right)^{1/2}.
%\cr}
%\eqno(27)
\end{equation}

Integrating by parts  for
$\pi\geqslant |x| \geqslant{1\over n}$ and $t\in {\sym R}$,
we have
$$
\sum_{k=1}^n{1\over {k^{1-it}}}{\rm e}^{ixk}
=
\int\limits_{1/\pi}^n{1\over {y^{1-it}}}
{\rm d}\left( \sum_{k<y}{\rm e}^{ixk} \right)
=(1-it)\int\limits_{1/\pi}^n{\rm e}^{ix}
{{1-{\rm e}^{ix[y]}}\over {1-{\rm e}^{ix}} }{1\over
{y^{2-it}}}{\rm d}y+{\rm O}(1)
$$
\begin{eqnarray}
&&\quad ={{1-it}\over {1-{\rm e}^{ix}} }{\rm e}^{ix}
\int\limits_{1/\pi}^{1/ {|x|}}
{ {1-{\rm e}^{ix[y]}}\over {y^{2-it}} }{\rm d}y+{\rm O}\left(
  { { |1-it|}\over {|{\rm e}^{ix}-1|} }\int\limits_{1/ {|x|}}^n
{{{\rm d}y}\over {y^2}}
\right)\nonumber
\\
&&\quad ={{1-it}\over {1-{\rm e}^{ix}} }{\rm e}^{ix}\int\limits_{1/\pi}^{1/ {|x|}}
{ {1-{\rm e}^{ixy}}\over {y^{2-it}} }{\rm d}y+{\rm O}\left(
{{|1-it|}\over {|1-{\rm e}^{ix}|} }\int\limits_{1/\pi}^{1/{|x|}}
{ { |1-{\rm e}^{ ix\{ y\} } | }\over {|y^2|} }{\rm d}y \right) +
{\rm O}(|1-it|)\nonumber
\\
&&\quad =-{{1-it}\over {1-{\rm e}^{ix}} }{\rm e}^{ix}\int\limits_{1/\pi}^{1/ {|x|}}
{ {ixy}\over {y^{2-it}} }{\rm d}y+
{{1-it}\over {1-{\rm e}^{ix}} }{\rm e}^{ix}\int\limits_{1/\pi}^{1/ {|x|}}
{ {1-{\rm e}^{ixy}+ixy}\over {y^{2-it}} }{\rm d}y+{\rm O}(|1-it|)\nonumber
\\
&&\quad ={{1-it}\over {{\rm e}^{ix}-1} }{\rm e}^{ix}
\int\limits_{1/\pi}^{1/ {|x|}}{ {ixy}\over {y^{2-it}} }{\rm d}y+
{\rm O}\left(
 {{|1-it|}\over {|x|}}\int\limits_{1/\pi}^{1/ {|x|}}{{|xy|^2}\over
 {y^2}}{\rm d}y
\right)
+{\rm O}(|1-it|)\nonumber
\\
&&\quad ={{1-it}\over {{\rm e}^{ix}-1} }{\rm e}^{ix}ix
\int\limits_{1/\pi}^{1/ {|x|}}{ {{\rm d}y}\over {y^{1-it}} }+{\rm O}(|1-it|)
={{1-it}\over {{\rm e}^{ix}-1} }{\rm e}^{ix}ix{{|x|^{-it}-1}\over {it}}+
{\rm O}(|1-it|)\nonumber
\\ \label{S(t)assympt}
&&\quad ={{{\rm e}^{it\log{1\over |x|}}-1}\over {it} }+{\rm O}(|1-it|).
\end{eqnarray}
Hence, it follows that

\begin{equation}
\label{expsum}
\left|
\exp \left\{ \sum_{k=1}^n{{{\rm e}^{it\log k}}\over k}{\rm e}^{ixk}
 \right\}
\right|\ll \exp \left\{  {{
\sin \left( t\log {1\over |x|} \right)
                       }\over t} \right\}% \eqno(28)
\end{equation}

for $|t|\leqslant 1$ and $|x|>{1\over n}$.

Since
$$
\big|\phi_{n,\alpha}(t)\big|=
\bigg\vert   N_n\bigg( {t\over {\big( {1\over 3}\log^3n \big)^{1/2} } }
\bigg)
\bigg|,
$$
putting
$$
t'={t\over {\left( {1\over 3}\log^3n \right)^{1/2} } }\quad \hbox{for}\
\log^{3/2}n \geqslant |t| \geqslant  \delta_1 \log^{1/2}n,
$$
from (\ref{|N_n(t)|}) and (\ref{expsum}) we get
\begin{multline*}
|\phi_{n,\alpha}(t)|^2
\ll
{1\over n}\left( \int\limits_{\pi \geqslant |x| \geqslant {1\over \sqrt{n}}}+
{1\over n}\int\limits_{ {1\over \sqrt{n}}>|x|\geqslant {1\over n} }+
{1\over n}\int\limits_{|x|<{1\over n}}\right)
\\
 \times\left|
\exp \left\{ \sum_{k=1}^n
{{{\rm e}^{it'\log k}}\over k}
{\rm e}^{ixk} \right\}
\right|^2{\rm d}x=
I_1+I_2+I_3.
\end{multline*}

Since $ {{
\sin( t\log {(1/ |x|)})}\over t}\leqslant \log {(1/ |x|)}$,
we have
\begin{equation*}
\begin{split}
I_1&\leqslant{1\over n}\int\limits_{\pi \geqslant |x| \geqslant {1\over \sqrt{n}}}
\exp \Big\{ 2\log {1\over {|x|}}\Big\}{\rm d}x\ll {1\over
{\sqrt{n}}}.
\\
I_2&\ll {1\over n}\int\limits_{ {1\over \sqrt{n}}>|x|\geqslant {1\over n} }
\exp \left\{ 2\log {1\over {|x|}}\left(
  {{
\sin \left( t'\log {1\over |x|} \right)
                       }\over {t'\log {1\over {|x|}}}}
\right)  \right\} {\rm d}x\cr
\\
&\leqslant {1\over n}\int\limits_{ {1\over \sqrt{n}}>|x|\geqslant {1\over n} }
\exp \left\{ 2\log {1\over {|x|}}(1-\varepsilon)  \right\} {\rm d}x
\leqslant {1\over n}
\int\limits_{1>|x|>{1\over n}}
{{{\rm d}x}\over {|x|^{2(1-\varepsilon )}}}\ll {1\over
{n^{2\varepsilon }} },\nonumber
\end{split}
\end{equation*}
where
$1-\varepsilon:=\max_{u\geqslant (\delta_1/ 2){\sqrt 3}}{{\sin u}\over u}$.

Finally, since
$
\exp \{ \sum_{k=1}^n
{{{\rm e}^{it'\log k}}\over k}
{\rm e}^{ixk} \}
=
n\exp \left\{
% \sum_{k=1}^n
%{{{\rm e}^{it'\log k}}\over k}
S_n(t')+{\rm O}(1)
\right\}$
for $|x|\leqslant {1\over n}$,  applying (\ref{exp_S_n(t)})
with $\varepsilon_1=\delta_1$, we have
$$
I_3\ll \exp \{ 2\Re S_n(t') \} \ll  {1\over {n^{\lambda (\delta_1)}}}.
$$
The theorem is proved.
\end{proof}
\begin{thm}
\label{difff_S_n-f_E_n}
Let $a_1, a_2,\ldots ,a_n$ be real numbers.
We  denote
$$
f_{S_n}(t)={\bf M_n}{\rm e}^{it\sum_{k=1}^{n}a_k\alpha_k(\sigma) }\quad
\hbox{and}\quad
f_{E_n}(t)={\bf M_n}{\rm e}^{it\sum_{k=1}^{n}a_k\xi_k(P)}.
$$
If
$\max_{1\leqslant k\leqslant n}{{|a_k|}/ {q^{k/4}}}\leqslant \beta_n,$
then, for $|t|\leqslant {1\over {\beta_n}}$, we have
$
|f_{S_n}(t)-f_{E_n}(t)|\ll \beta_n |t|.
$
\end{thm}

\begin{proof}
Let $a_k=0$ for $k>n$.
Taking $f(k)={\rm e}^{ita_k}$  in (\ref{identity_for_Mn}), we have
$$
F_{E_n}(z)=\sum_{m=0}^{\infty}f_{E_m}(t)z^n=\prod_{k=1}^{\infty}
\left(
  {{ 1-\left( {z\over q} \right)^k{\rm e}^{ita_k} } }
\right)^{-I_k},
$$
$$
F_{S_n}(z)=\sum_{m=0}^{\infty}f_{S_m}(t)z^m=\exp
 \left\{ \sum_{k=1}^{\infty} {{{\rm e}^{ita_k}}\over k}z^k \right\},
$$
\begin{equation}
\label{F_E_n}
F_{E_n}(z)=\exp \left\{
 \sum_{k=1}^{\infty}{ {{\rm e}^{ita_k}}\over k}z^k+H(z,t)\right\}
=F_{S_n}(z)\exp \{ H(z,t) \},%\eqno(29)
\end{equation}
where
\begin{multline*}
H(z,t)=-\sum_{k=1}^{\infty }
{{q^k}\over k} \left[ \log \left( 1-\left( {z\over q}
\right)^k{\rm e}^{ita_k}\right)
 +\left( {z\over q} \right)^k
 {\rm e}^{ita_k} \right]
 \\
  +\sum_{k=1}^{\infty }A_k{{q^{k/2}}\over k}\log \left( 1- \left( {z\over q}
\right)^k
{\rm e}^{ita_k} \right);\nonumber
\end{multline*}
here  $A_k$  are the same numbers as in $(\ref{assympt_In})$.
For $t=0$, we have
$$
{1\over {1-z}}=
\prod_{k=1}^{\infty}{{\left( 1-\left( {z\over q} \right)^k \right)^{-I_k}}}=
{{\exp \{ H(z,0) \} }\over {1-z} },
$$
and, therefore, $H(z,0)=0$. Using this identity,
from (\ref{F_E_n}) we have
\begin{equation*}
\begin{split}
F_{E_n}(z)
&=F_{S_n}(z)\exp \{ H(z,t)-H(z,0) \}
\\
&=F_{S_n}(z)\exp \left\{
\sum_{k=1}^{\infty} {{q^k}\over k}\sum_{j=2}^{\infty}\left( {z\over q} \right)^{jk}
{{({\rm e}^{itja_k}-1)}\over j}\right.\cr
&\quad \left.+
\sum_{k=1}^n {{q^{k/2}}\over k}A_k
\sum_{j=1}^{\infty}\left( {z\over q} \right)^{jk}
{{({\rm e}^{itja_k}-1)}\over j}\right\}
\\
&=F_{S_n}(z)\exp \left\{
 |t|\sum_{k=1}^{\infty}
 {{q^k}\over k}\sum_{j=2}^{\infty}\left( {z\over q} \right)^{jk}
a_kb_{kj}(t)\right.
\\
&\quad \left.+
|t|\sum_{k=1}^{\infty} {{q^{k/2}}\over k}A_k
\sum_{j=1}^{\infty}\left( {z\over q} \right)^{jk}
a_kb_{kj}(t)\right\},
\end{split}
\end{equation*}
where $|b_{kj}|\leqslant 1$. Therefore, we have
$$
F_{E_n}(z)=F_{S_n}(z)\exp \left\{\beta_n |t|\sum_{m=1}^{\infty} \gamma_m(t)z^m
 \right\},
$$
where $|\gamma_m(t)|\leqslant {A\over {q^{m/4}}}$ with an
absolute constant $A$.
Recalling that
 $\beta_n|t|\leqslant 1$  for $s\geqslant 1$, we get
\begin{multline*}
\left[   \exp \left\{\beta_n |t|\sum_{m=1}^{\infty} \gamma_m(t)z^m
 \right\}      \right]_{(s)}
\leqslant
  \left[  1+\sum_{k=1}^{\infty}{1\over {k!}}
 \left( \beta_n|t|\sum_{m=1}^{\infty} \gamma_m(t)z^m
 \right)^k     \right]_{(s)}
\\
 \leqslant  \left[  1+\beta_n|t|\sum_{k=1}^{\infty}{1\over {k!}}
 \left( \sum_{m=1}^{\infty} {A\over {q^{m/4}}}z^m
 \right)^k     \right]_{(s)}
=
\left[  \beta_n|t| \exp \left\{ \sum_{m=1}^{\infty} {A\over {q^{m/4}}}z^m
 \right\}      \right]_{(s)}
 \\
\ll {{\beta_n|t|}\over {q^{s/8}}}.
\end{multline*}

Since $[F_{S_n}(z)]_{(k)}\leqslant 1$,
using the inequality obtained before, we have
\begin{equation*}
\begin{split}
f_{E_n}(t)&=[F_{E_n}(z)]_{(n)}=
\left[ F_{S_n}(z)\exp \left\{ \beta_n|t|\sum_{m=1}^{\infty} \gamma_m(t)z^m
 \right\} \right]_{(n)}
\\
&=f_{S_n}(t)+{\rm O}\left(\beta_n|t|\sum_{k=1}^n {1\over {q^{k/8}}} \right)
=f_{S_n}(t)+{\rm O}(\beta_n|t|).
\end{split}
\end{equation*}

The theorem is proved.
\end{proof}

Applying Theorem~\ref{difff_S_n-f_E_n} with
$a_k={{\log k }\over {({1/{\sqrt 3})}}\log^{3/2} n}$,
we have
\begin{equation}
\label{phi_alpha_minus_phi_xi}
\big|\phi_{n,\alpha}(t)-\phi_{n,\xi}(t)\big|\ll
{{|t|}\over {\log^{3/2}n}}%\eqno(30)
\end{equation}
for $|t|\leqslant \log^{3/2}n$.

Let
$$
G_n(x)=\Phi(x)
+{{3^{3/2}}\over {24\sqrt{2\pi}}}
{{(1-x^2)\er^{-{x^2/2}}}\over \sqrt{\log n}}.
$$
Then we have
$$
\int\limits_{-\infty}^{+\infty}{\rm e}^{itx}\,\dr G_n(x)={\rm e}^{-{{t^2}\over 2}}
\left( 1+(it)^3{{3^{3/2}}\over {4!}} {1\over {\log^{1/2}n}} \right)
=:g_n(t).
$$
 Let us also denote
$$
F_n(x)={\nu_n}\left(
{{\log P_n(\xi)-{1\over 2}\log^2n }
\over {\left( {1\over 3}\log^3 n \right)^{1/2}  }}<x
\right).
$$
Applying the generalized Esseen inequality (see e. g. \cite{petrov}), we have
\begin{multline*}
\sup_{x\in {\sym R}}\big|G_n(x)-F_n(x)\big|\ll \int\limits_{-\log^{3/4}n}^{\log^{3/4}n}
{{|\phi_{n,\xi}(t)-g_n(t)|}\over {|t|}}{\rm d}t+{1\over {\log^{3/4}n}}
\\
\leqslant \int\limits_{-\log^{3/4}n}^{\log^{3/4}n}
{{|\phi_{n,\alpha}(t)-g_n(t)|}\over {|t|}}{\rm d}t+
\int\limits_{-\log^{3/4}n}^{\log^{3/4}n}
{{|\phi_{n,\alpha}(t)-\phi_{n,\xi }(t)|}\over {|t|}}\,\dr t
+{1\over {\log^{3/4}n}}.
\end{multline*}
Using~(\ref{phi_alpha_minus_phi_xi}), we can estimate the second integral
in this equality by ${\rm O}(\log^{-3/4}n)$.
We estimate the first integral using, in the interval
$|t|\leqslant {1\over {\log^2 n}}$,
the estimate
\begin{equation*}
\begin{split}
\big|\phi_{n,\alpha}(t)-1\big|
&\leqslant \big|\phi_{n,\alpha}(t)-\phi_{n,\xi}(t)\big|
+\big|1-\phi_{n,\xi}(t)\big|
\\
&\ll |t|\left( {{1}\over {\log^{3/2}n}}
+  {\bf M_n}\left| { {\log P_n(\xi (P))-{(1/ 2)}\log^2 n}
\over {(1/\sqrt 3 )\log^{3/2}n}}\right|    \right)
\ll |t|\log^{1/2} n
\end{split}
\end{equation*}
and, in the intervals
${1\over {\log^2n}}<|t|\leqslant  \log^{1/6}n$,
$ \log^{1/6}n\leqslant |t| <\delta_1 \log^{1/2}n$, and
$\delta_1 \log^{1/2}n<|t|<\log^{2/3}n$,
estimates  (\ref{phi_alphan}), (\ref{phi_nalpha_smalt}), and (\ref{phi_nalpha_larget}), respectively.

Finally, we have

\begin{equation}
\label{supG_n-F_n}
\sup_{x\in {\sym R}}|G_n(x)-F_n(x)|\ll {1\over {\log^{3/4}n}}.%\eqno(31)
\end{equation}

To estimate the closeness between the distribution functions of
$\log O_n(\xi)$ and $\log P_n(\xi)$,
we  use the following lemma, which is a generalization of
Lemma~2.5 of~\cite{barbour_tavare}.

\begin{lem}
\label{cmp_x_and_u}
 Let $U$ and $X$ be random variables.
Suppose that $\sup_{x\in {\sym R}}|P(U<x)-G(x)|\leqslant \eta$,
 where $G(x)$ is a differentiable function
satisfying the condition
 $|G'(x)|\leqslant C$. Then, for each  $\varepsilon >0$, we have
$$
\sup_{x\in {\sym R}}|P(U+X<x)-G(x)|\ll \eta +\varepsilon +
P(|X|>\varepsilon),
$$
where the constant in the symbol
 $\ll$ depends on $C$ only.
\end{lem}
 \begin{proof}[Proof of theorem \ref{cltxith}] Putting, in Lemma~\ref{cmp_x_and_u},
$$
X=-{  { \log P_n(\xi)-\log O_n(\xi)-
\mu_n}
\over {(1/\sqrt 3)\log^{3/2} n} }\quad \hbox{\rm and}\quad
U={  { \log P_n(\xi)-{\bf M_n}\log P_n(\xi)} \over {(1/\sqrt 3)\log^{3/2} n} },
$$
 $\varepsilon=( (\log \log n)/ \log n) ^{2/3}$,
$\eta ={1/{\log^{3/4}n}}$ and then applying
 (\ref{supG_n-F_n}) and Proposition 1, we obtain the proof of the theorem.
\end{proof}

\begin{proof}[Proof of theorems \ref{cltaltha} and \ref{meanoalpha} ]
 In order to prove that the results of this paper hold in the case of the
 group $S_n$, one can either repeat
the proofs given above in a simplified way, or pass to the limit
as $q\to \infty$ for fixed $n$ and using the facts that
$( \xi_1(P),\ldots ,\xi_n(P) ) \to ( \alpha_1(\sigma),\ldots,
\alpha_n(\sigma) )$ weakly as $q\to \infty$ and  that
the constants in the symbols ${\rm O}$ and $\ll$ do not depend on $q$
in all the proofs given above.
\end{proof}
\chapter{Functions on $S_n^{(k)}$}
\label{ch_sub}
\section {Means of multiplicative functions on $S_n^{(k)}$}
%-------------------------------------------------------------------
%\midinsert
%\hbox{
%\vbox{}\hskip 5truecm
%\vbox{
%\special{em:graph fractex.pcx}%}}
%\vskip5truecm
%\hskip 2cm
%\special{em: graph pict.pcx}
%\endinsert
%\vskip 2cm
%----------------------------------------------------------------
As before we denote
$$
S_n^{(k)}=\{ \sigma=x^k  | x\in S_n  \}.
$$
 In \cite{minpav} Mineev and Pavlov proved the following criterion to determine whether $\sigma$ belongs to the set $S_n^{(k)}$.
\begin{thmabc}
Suppose $k$ has the following decomposition into the product of prime numbers
 $k=p_1^{l_k(p_1)}p_2^{l_k(p_2)}\ldots p_s^{l_k(p_s)}$.  Let us define the function
 $$
 q_k(j)=\prod_{p|(k,j)}p^{l_k(p)}.
 $$
 Then $\sigma \in S_n^{(k)}$  if and only if
 $$
 q_k(j)|\alpha_j(\sigma)
 $$
 for all $1\leqslant j \leqslant n$.
\end{thmabc}
 We have
\begin{equation*}
\begin{split}
{\bf M_n^{(k)}}f
&=\frac{1}{|S_n^{(k)}|} \sum_{\sigma \in S_{n}^{(k)}}f(\sigma)=
\frac{1}{|S_n^{(k)}|}\sum_{\scriptstyle s_1+2s_2 +\cdots +ns_n=n
\atop \scriptstyle q_k(l)|s_l}\prod_{j=1}^n\hat f(j)^{s_j}n!\prod_{j=1}^n{\frac{1}{j^{s_j}s_j!}}
\\
&=\frac{n!}{|S_n^{(k)}|}\sum_{\scriptstyle s_1+2s_2 +\cdots +ns_n=n
\atop \scriptstyle q_k(l)|s_l}\prod_{j=1}^n{\left(\frac{\hat f(j)}{j}\right)^{s_j}\frac{1}{s_j!}}.
\end{split}
\end{equation*}
Here we have used the well known fact that the quantity of
  $\sigma \in S_n$ such that
$\alpha_j(\sigma)=s_j$ for $1\leqslant j \leqslant n$, equals
$$
n!\prod_{j=1}^{n} {\frac{1}{s_j!j^{s_j}}},
$$ when
$s_1+2s_2+\cdots +ns_n=n$.
Hence one can easily see that the following identity holds
\begin{equation}
\label{f_gen_func}
\begin{split}
\sum_{n=0}^\infty\frac{|S_n^{(k)}|}{n!}{\bf M_n^{(k)}}fz^n&=\prod_{j=1}^\infty\left(
1+\sum_{{\scriptstyle j\geqslant 1}\atop{\scriptstyle q_k(j)|s}}\left(\frac{\hat f(j)z^j}{j}\right)^s\frac{1}{s!} \right)\nonumber
\\
&=\exp \biggl\{ \sum_{{\scriptstyle j\geqslant 1}\atop{ \scriptstyle (j,k)=1}}\frac{\hat f(j)}{j}z^j \biggr\}\prod_{(j,k)>1}\left(
1+\sum_{{s\geqslant 1}\atop{q_k(j)|s}}\left(\frac{\hat f(j)z^j}{j}\right)^s\frac{1}{s!} \right)
\nonumber
\\
&=\exp \biggl\{   \sum_{{j \geqslant 1}\atop{(j,k)=1}}\frac{\hat f(j)}{j}z^j \biggr\}H_k(f;z).
\end{split}
\end{equation}

Further we will assume that $|\hat f(j)|\leqslant 1$.
  We will denote by  ${\bf M}_n(f)={\bf M}_n^{(k)}(f)$
the mean value of $f(\sigma)$ on the subset $S_n^{(k)}$.
 Let us define
$$
\mu_n(p)=\biggl(\frac{1}{n}\sum_{\scriptstyle 1\leqslant j\leqslant n \atop \scriptstyle (j,k)=1 }|\hat f(j)-1|^p\biggr)^{1/p}.
$$
In the works \cite{manstberry},  \cite{manstadditive} and \cite{zakvoronoi} there have been obtained the estimates of the mean values of the multiplicative functions on whole group $S_n$.

The following theorem establishes analogous estimate for ${\bf M_n^{(k)}}f$.

\begin{thm}
\label{mean}
%{\bf Theorem 1.} {\it
Suppose $|\hat f(j)|\leqslant 1$. Then %there exists such fixed $\epsilon =\epsilon(k)>0$, that
$$
{\bf M_n^{(k)}}{f}=\exp \biggl\{ \sum_{\substack{1\leqslant j\leqslant n \\ \scriptstyle (j,k)=1} }{\frac{\hat f(j)-1}{ j}} \biggr\}\prod_{\substack{ j\leqslant n \\ \scriptstyle (j,k)>1}}
\left(\frac{1+\displaystyle{\sum_{s \geqslant 1:q_k(j)|s}}\frac{\hat f(j)^s}{j^ss!}}{1+\displaystyle{\sum_{s \geqslant 1:q_k(j)|s}}\frac{1}{j^ss!}} \right)  +O\left(\mu_n(p)+\frac{1}{n^{\beta}}\right),
$$
for $p>\frac{1}{\beta}$. Here $\beta=\frac{\phi(k)}{k}$,
if $k$ is prime, and
$$
\beta=\min_{d:d|k,d>1}\frac{\phi(k)}{k}\left(1- {\mu(d)}\prod_{p|d}\frac{1}{p-1}\right),
$$
if  $k$ is composite, where $\mu(d)$ -- Moebius function, $\phi(k)$ -- Euler function.
%}
\end{thm}
 The next theorem is the analog of the well known Halsz - Wirsing result for the multiplicative functions on natural numbers. Analogous result for multiplicative functions on permutations has been obtained in \cite{manstadditive}.
\begin{thm}
\label{halash} Supose we have a fixed sequence of complex numbers
$\hat f(j)$ such that $|\hat f(j)|\leqslant 1$. Then there are two
possible cases concerning the asymptotic behavior of the
corresponding sequence of means ${\bf M_n^{(k)}}f$.
\begin{enumerate}
\item
If the series
\begin{equation}
\label{halseries}
\sum_{\substack{j\geqslant 1\\(j,k)=1}}\frac{1-\Re (\hat f(j)e^{-ixj})}{j}
\end{equation}
diverges for every $x\in [-\pi ,\pi ]$ then
$$
\lim_{n\to \infty}{\bf M_n^{(k)}}f=0
$$
\item
If
$$
\sum_{\substack{j\geqslant 1\\(j,k)=1}}\frac{1-\Re (\hat f(j)e^{-ix_0j})}{j} <\infty,
$$
for some $x_0\in [-\pi, \pi]$, then
$$
{\bf M_n^{(k)}}f=e^{ix_0n}\exp \biggl\{ \sum_{\scriptstyle 1\leqslant j\leqslant n \atop \scriptstyle (j,k)=1 }{\frac{\hat f(j)e^{-ix_0j}-1}{ j}} \biggr\}\prod_{\scriptstyle j\leqslant n \atop \scriptstyle (j,k)>1}
\left(\frac{1+\displaystyle{\sum_{s \geqslant 1:q_k(j)|s}}\frac{(\hat f(j)e^{-ix_0j})^s}{j^ss!}}{1+\displaystyle{\sum_{s \geqslant 1:q_k(j)|s}}\frac{1}{j^ss!}} \right)  +o(1).
$$
\end{enumerate}
\end{thm}
In what follows we assume that $k$ is a fixed natural number, therefore
we will often omit the index $k$ in the notation of the mean value  ${\bf M}_nf={\bf M}_n^{(k)}f$ and measure $\nu_n=\nu_n^{(k)}$.

%\section{Proofs}
Further we will denote $k_0=\prod_{p|k}p$.
Putting in (\ref{f_gen_func}) $\hat f(j)\equiv 1$, we obtain
$$
F(z)=\sum_{n=0}^\infty\frac{|S_n^{(k)}|}{n!}z^n=p(z)H_k(1;z)
$$
where
$$
p(z)=\sum_{j=0}^\infty p_jz^j=\exp \biggl\lbrace \sum_{{\scriptstyle j\geqslant 1}\atop{\scriptstyle (j,k)=1}}\frac{z^j}{j} \biggr\rbrace=
\prod_{m|k}\frac{1}{(1-z^m)^{\frac{\mu(m)}{m}}}=\prod_{j=0}^{k_0-1}
\frac{1}{\bigl(1-ze^{-2\pi i \frac{j}{k_0}}\bigr)^{\gamma_j}},
$$
and $\gamma_j=\frac{\phi(k)}{k}{\mu(l_j)}\prod_{p|l_j}\frac{1}{p-1}$, where
$l_j=\frac{k_0}{(j,k_0)}$
for $1\leqslant j <k_0$, $\gamma_0=\frac{\phi(k)}{k}$.
In the work of Mineev and Pavlov it has been proved that $\gamma_j<\gamma_0=\frac{\phi(k)}{k}$, if  $j\not= 0$.
One can easily see also that $|\gamma_j|\leqslant \gamma_0$, moreover,
$|\gamma_j|=\gamma_0$ if and only if $l_j=2$, that is when $\frac{j}{k_0}=\frac{1}{2}$.

Further we will denote by $\epsilon$
some positive, fixed number, not necessarily the same in different places.

Suppose $f(z)=\sum_{j=0}^\infty a_jz^j$,  we will the following
notation for the $n$-th coefficient in the Taylor expansion of
$f$: $[f(z)]_{(n)}=a_n$. Further we will often use some simple
properties of this notation, which we formulate in a form of the
lemma.
%
% {\bf Lemma 1.} {\it
\begin{lem}
\label{coef}
Suppose $u(z)$, $v(z)$, $U(z)$, $V(z)$, $\psi (z)$
are analytic in the vicinity of zero and such that
  $|[u(z)]_{(n)}|\leqslant [U(z)]_{(n)}$, $|[v(z)]_{(n)}|\leqslant [V(z)]_{(n)}$ and
$[\psi (z)]_{(n)}\geqslant 0$ for $n\geqslant 0$.
Then for $n\geqslant 0$ the following inequalities hold:

\begin{enumerate}
\item
 $$
\bigl|[u(z)v(z)]_{(n)}\bigr|\leqslant [U(z)V(z)]_{(n)}, %\leqno1)
$$
\item
$$
\bigl|[e^{u(z)}]_{(n)}\bigr|\leqslant [e^{U(z)}]_{(n)}, %\leqno2)
$$
\item
$$
\bigr|[e^{\epsilon u(z)}-1]_{(n)}\bigl|\leqslant \epsilon [e^{U(z)}-1]_{(n)} ,\quad \hbox{\it for}\quad 0\leqslant \epsilon\leqslant 1, %\leqno3)
$$
\item
$$
[U(z)]_{(n)}\leqslant [U(z)(1+V(z))]_{(n)},% \leqno4)
$$
\item
$$
0\leqslant\frac{1}{2}\left[\psi(z)(e^{\psi(z)}-1)\right]_{(n)} \leqslant\left[1+e^{\psi(z)}\psi(z)-e^{\psi(z)} \right]_{(n)}\leqslant \left[\psi(z)(e^{\psi(z)}-1)\right]_{(n)}. %\leqno5)
$$
\end{enumerate}
\end{lem}
% }
%

The estimates which are similar to those of the lemmas
\ref{hcoef}, \ref{md} and \ref{meanalph}, have been obtained in
the work of Pavlov. For the sake of completeness we give here
somewhat more elementary their proofs. We will estimate the $n$-th
Taylor coefficient of the function $H_k(f;z)$.
%--------------------------------------------------------------------------------------------------
%------------------------Lemma 1--------------------------------------------------------------
%----------------------------------------------------------------------------------------------------
%
% {\bf Lemma 2.} {\it
\begin{lem}
\label{hcoef}
For $n\geqslant 1$ and $|\hat f(j)|\leqslant 1$ we have
$$
[H_k(f;z)]_{(n)}=O\left( \frac{1}{n^2}\right).
$$
\end{lem}
% }
%
% {\it Proof.}
\begin{proof}
Applying lemma \ref{coef} one can easily see that for $n\geqslant 0$ and $|\hat f(j)|\leqslant 1$
\begin{equation*}
\begin{split}
|[H_k(f;z)]_{(n)}|&\leqslant \left[  \prod_{{\scriptstyle j\geqslant 1}\atop{\scriptstyle (j,k)>1}}\left(
1+\sum_{q_k(j)|s}\left(\frac{z^j}{j}\right)^s\frac{1}{s!} \right)\right]_{(n)}
\\
&\leqslant \left[
 \prod_{{\scriptstyle j\geqslant 1}\atop{\scriptstyle (j,k)>1}}\left(
1+\sum_{l=1}^\infty \left(\frac{z^j}{j}\right)^{lq_k(j)}\frac{1}{l!} \right)\right]_{(n)}
=\left[ \exp \biggl\lbrace \sum_{{\scriptstyle j\geqslant 1}\atop{\scriptstyle (j,k)>1}}\frac{z^{jq_k(j)}}{j^{q_k(j)}}   \biggr\rbrace \right]_{(n)}
\\
&\leqslant
\left[ \exp \biggl\lbrace \sum_{{\scriptstyle j\geqslant 1}\atop{\scriptstyle (j,k)>1}}\frac{z^{jq_k(j)}}{j^{2}}   \biggr\rbrace \right]_{(n)}
\leqslant \left[ \exp \biggl\lbrace k^2 \sum_{j=1}^\infty\frac{z^{j}}{j^2}   \biggr\rbrace \right]_{(n)};
\end{split}
\end{equation*}
here we have used the fact that $2\leqslant q_k(j)\leqslant k$ for $(j,k)>1$.
Therefore $|[H_k(f;z)]_{(n)}|\leqslant g_n$, where
$$
g(z)=\sum_{n=0}^\infty g_nz^n=  \exp \biggl\lbrace k^2 \sum_{j=1}^\infty\frac{z^{j}}{j^2}   \biggr\rbrace.
$$
One can easily see that $g''(x)\ll \frac{1}{1-x}$ for $0\leqslant x \leqslant 1$.
Therefore
$$
\sum_{j=1}^\infty j^2g_jx^j\ll \frac{1}{1-x}.
$$
Putting here $x=e^{-1/n}$, we obtain
$$
\sum_{n\leqslant j\leqslant 2n}g_j \ll \frac{1}{n}.
$$

Since $zg'(z)=k^2g(z)\sum_{m=1}^\infty \frac{z^m}{m}$, then
$$
ng_n=k^2\sum_{j=0}^{n-1}\frac{g_j}{n-j}\leqslant \frac{2k^2}{n}\sum_{j\leqslant n/2}g_j +
k^2\sum_{n/2< j \leqslant n}g_j,
$$
therefore $g_n=O\left(\frac{1}{n^2} \right)$, whence it follows that$[H_k(f;z)]_{(n)}=O\left(\frac{1}{n^2} \right)$.
%-----------------------------------END OF PROOF---------------------------------------------

The lemma is proved.
\end{proof}

The Lemma \ref{hcoef} shows that the main contribution to the value of ${\bf M_n^{(k)}}f$
is done by the coefficients of the function
$$
F(z)=\sum_{m=0}^\infty M_mz^m=\exp \biggl\lbrace \sum_{(j,k)=1}\frac{\hat f(j)}{j}z^j \biggr\rbrace % =p(z)\exp\{ L_n(z) \},
$$
This generating function may be regarded as a special case of more general generating function (\ref{genfunc2}) of Chapter 1  if we put
\begin{equation}
\label{d_jkcoprime}
d_j=\begin{cases}
1,&\text{if $(j,k)=1$}\\
0,&\text{if $(j,k)>1$}
\end{cases}.
\end{equation}
Unfortunately, Theorem \ref{meanf} is not directly applicable
here, as in our case the parameters $d_j$ are not bounded from
bellow by a positive constant. The proof of Theorem \ref{meanf}
was based on estimate of  Theorem \ref{fundthm1}, in the proof of
which we used the condition $d_j\geqslant d^->0$. In a general
case this condition can hardly be removed, because the behavior of
corresponding $p_j$ might become irregular. Let us take for
example in
$$
p(z)=\exp \left\lbrace \sum_{j=1}^{\infty}d_j\frac{z^j}{j} \right\rbrace
$$
$$
d_j=\begin{cases}
1,&\text{if $2|j$}\\
0,&\text{if $2\not|j$}
\end{cases},
$$
we will have then $p(z)=\sum_{j=0}^\infty p_j=\frac{1}{(1-z^2)^{1/2}}$ which gives  $p_{2j+1}=0$ for $j\geqslant 0$.

However, in our case when $d_j$ are defined by (\ref{d_jkcoprime})
the function $p(z)$ has a good analytic continuation beyond the
unit disc and good behavior of its special points which enables us
to establish the analog of Theorem \ref{fundthm1}.

As in Chapter 1 we define $g_{j,x}$ by means of relation
$$
G_x(z)=\frac{p(z)}{p(zx)}=\sum_{j=0}^\infty g_{j,x}z^j.
$$
%{\bf Lemma 15.} {\it
\begin{lem}
\label{g_nx}
For $0\leqslant x \leqslant e^{-1/n}$ we have
$$
g_{n,x}=\frac{p_n}{p(x)}
+O\left(n^{\gamma' -1}(1-x)^{\gamma'} +n^{\gamma -2}(1-x)^{\gamma-1}+\frac{1}{n} \right),
$$
where $\gamma'=\max_{j\not=0}\bigl( \max\left\lbrace \gamma_j ,0 \right\rbrace \bigr)$.
%}
\end{lem}

%{\it Proof}.
\begin{proof}
Applying Cauchy formula we have
$$
g_{n,x}=\frac{1}{2\pi i}\int _{C_{\phi,R}}\frac{p(z)}{p(zx)}\frac{dz}{z^{n+1}}
=\sum_{j=0}^{k-1}
\frac{1}{2\pi i}\int _{C_{\phi,R,j}}\frac{p(z)}{p(zx)}\frac{dz}{z^{n+1}}
+O\left( \frac{1}{R^n} \right).
$$
The integration contour in this formula is
$C_{\phi,R}=C_R\cup\cup_{j=0}^{k_0-1}C_{\phi,R,j}$. Here for
$\gamma_j>0$, $C_{\phi,R,j}$ consists of the intervals on the
complex plain, connecting the points
  $e^{2\pi i j/k}$ and $Re^{2\pi i j/k}$. For $\gamma_j<0$, $C_{\phi,R,j}$
  consists of the intervals
$\{z=e^{2\pi i j/k}(1+re^{i\phi})\}$ and $\{z=e^{2\pi i j/k}(1+re^{-i\phi})\}$, where $1\leqslant r \leqslant r_0$, and
 $r_0$
 corresponds to the point of intersection with the circle $|z|=R$,  $0<\phi <\frac{\pi}{2}$ -- fixed, sufficiently small angle.

Here $2 \leqslant R \leqslant 6$
is chosen in such a way that $|1-Rx|\geqslant 1/2$.

Let $\gamma_j>0$, then
\begin{equation*}
\begin{split}
\frac{1}{2\pi i}\int _{C_{\phi,R,j}}\frac{p(z)}{p(zx)}\frac{dz}{z^{n+1}}
&\ll
\int_1^R\left|\frac{1-xy}{1-y} \right|^{\gamma_j} \frac{dy}{y^{n+1}}=
\frac{1}{n}\int_0^{Rn}\left|\frac{1-xe^{-u/n}}{1-e^{-u/n}} \right|^{\gamma_j}
e^{-u}\,du
\\
&\ll \frac{1}{n}\int_0^{Rn}\left(\frac{n}{u} \right)^{\gamma_j} |1-x+x(1-e^{-u/n})|^{\gamma_j}
e^{-u}\,du
\\
&\ll n^{\gamma_j -1}(1-x)^{\gamma_j} +\frac{1}{n}
\end{split}
\end{equation*}
and for $\gamma_j<0$
\begin{equation*}
\begin{split}
\frac{1}{2\pi i}\int _{C_{\phi,R,j}}\frac{p(z)}{p(zx)}\frac{dz}{z^{n+1}}&\ll\int _{C_{\phi,R,j}}\left|\frac{1-z}{1-xz} \right|^{(-\gamma_j)} \frac{dz}{z^{n+1}}
\\
&\ll
\int_0^{R}\left|\frac{r}{1-x(1+re^{i \phi})} \right|^{(-\gamma_j)} \frac{dr}{|1+re^{i \phi}|^{n+1}}
\\
&\ll \int_0^{R}\frac{ dr}{(1+r \cos \phi)^{n+1}}\ll \frac{1}{n},
\end{split}
\end{equation*}
because
$$
|1-x(1+re^{i \phi})|=|(1-x)e^{-i\phi /2}-rxe^{i\phi /2}|\geqslant \left| \sin \frac{\phi}{2} \right| \bigl|
(1-x) +rx
\bigr|
$$
and $|1+re^{i \phi}|\geqslant 1+r \cos \phi$.

 Hence
$$
g_{n,x}=\frac{1}{2\pi i}\int _{C_{\epsilon,R,0}}
\frac{p(z)}{p(zx)}
\frac{dz}{z^{n+1}}
   +O\left(\sum_{\gamma_j >0}n^{\gamma_j -1}(1-x)^{\gamma_j}+ \frac{1}{n} \right).
$$
Denoting $V_{x}(z)=\prod_{l=1}^{k-1}
{\left( \frac{1-xze^{-2\pi i l/k}}{1-ze^{-2\pi i l/k}}\right)}^{\gamma_j}$, we obtain
$\frac{p(z)}{p(zx)}=V_{x}(z){\left( \frac{1-xz}{1-z}\right)}^{\gamma_0}$.
From the above estimate we have
\begin{equation*}
\begin{split}
\frac{1}{2\pi i}\int _{C_{\epsilon,R,0}}
\frac{p(z)}{p(zx)}
\frac{dz}{z^{n+1}}&=
\frac{V_{x}(1)}{2\pi i}\int _{C_{\epsilon,R,0}}{\left( \frac{1-xz}{1-z}\right)}^{\gamma_0}\frac{dz}{z^{n+1}}
\\
&\quad+\frac{1}{2\pi i}\int _{C_{\epsilon,R,0}}(V_{x}(z)-V_{x}(1)){\left( \frac{1-xz}{1-z}\right)}^{\gamma_0}\frac{dz}{z^{n+1}}
\\
&=I_1 +I_2.
\end{split}
\end{equation*}
Let us estimate $I_1$.
\begin{equation*}
\begin{split}
I_1&=
\frac{V_{x}(1)(1-x)^{\gamma_0}}{2\pi i}\int _{C_{\epsilon,R,0}} \frac{dz}{(1-z)^{\gamma_0}z^{n+1}}
\\
&\quad+\frac{V_{x}(1)}{2\pi i}\int _{C_{\epsilon,R,0}}
\frac{(1-x)^{\gamma_0}}{(1-z)^{\gamma_0}}
\left( \left(1+\frac{x(1-z)}{1-x} \right)^{\gamma_0} -1 \right) \frac{dz}{z^{n+1}}
\\
&=V_{x}(1)(1-x)^{\gamma_0}{{n+\gamma_0 -1}\choose {n}}+O(n^{-1})
\\
&\quad+O\left(\frac{|V_{x}(1)|}{2\pi i}\int _{C_{\epsilon,R,0}}
\frac{(1-x)^{\gamma_0}}{|1-z|^{\gamma_0}}
 \left|\frac{x(1-z)}{1-x} \right| \frac{dz}{z^{n+1}}\right)
 \\
&=V_{x}(1)(1-x)^{\gamma_0}{{n+\gamma_0 -1}\choose {n}}+O(n^{-1})
+O((1-x)^{\gamma_0 -1}n^{\gamma_0 -2}).
\end{split}
\end{equation*}
Because $V_x'(z)\ll 1$ in the vicinity of the point $z=1$, we have
$$
I_2\ll \int _{C_{\epsilon,R,0}}|1-z|{\left| \frac{1-xz}{1-z}\right|}^{\gamma_0}\frac{dz}{|z|^{n+1}}
\ll \int _{C_{\epsilon,R,0}}|1-z|^{1-\gamma_0}\frac{dz}{|z|^{n+1}}\ll n^{\gamma_0 -2 }.
$$
We have
$$
V_x(1)=\lim_{z\to 1}\frac{p(z)}{p(zx)}\left( \frac{1-z}{1-xz}\right)^{\gamma_0}=
\frac{A_k}{p(x)(1-x)^{\gamma_0}},
$$
where
$A_k=\lim_{z\to 1}(1-z)^{\gamma_0}p(z)$.

Applying the earlier obtained estimates we have
$$
g_{n,x}=\frac{A_k}{p(x)}{{n+\gamma_0 -1}\choose {n}}+O\left(\frac{1}{n}
+(1-x)^{\gamma_0 -1}n^{\gamma_0 -2}+
n^{\gamma' -1}(1-x)^{\gamma'}\right).
$$
Putting here $x=0$ and noting that $g_{n,0}=p_n$ we have

\begin{equation}
\label{p_nassympt}
p_n=A_k{{n+\gamma_0 -1}\choose {n}}+O(n^{\gamma' -1}).
\end{equation}

Inserting this estimate into the previous estimate we obtain the proof of the lemma.
\end{proof}
%\smallskip
The following result has been proved in the work \cite{pavlov}.
\begin{thmabc}[\cite{pavlov}]
\label{pavlovpn}
For $n\geqslant 1$ we have
$$
c_n=\frac{|S_n^{(k)}|}{n!}=\frac{n^{\gamma_0 -1}}{\Gamma(\gamma_0)}A_kH_k(1;1)
\left(1+O(n^{-\beta}) \right)=
p_nH_k(1;1)
\left(1+O(n^{-\beta}) \right),
$$
where
$A_k=\lim_{x\to 1}p(x)(1-x)^{\gamma_0}$, $\beta =\gamma_0 -\gamma'$ and $\gamma'=\max_{j\not=0}\bigl( \max\left\lbrace \gamma_j ,0 \right\rbrace \bigr)$.
\end{thmabc}
In formulation of this theorem in \cite{pavlov} the constant $\beta$ has not been written explicitly as $\beta =\gamma_0 -\gamma'$, although this  formula could be easily obtained from the proof of the theorem there. Therefore, and also in order to make our exposition self-contained we
present the proof of this theorem here.
\begin{proof}[Proof of theorem \ref{pavlovpn} ]
 As $\sum_{n=0}^\infty\frac{|S_n^{(k)}|}{n!}=p(z)H_k(f,z)$ therefore applying Lemma \ref{hcoef} and estimate (\ref{p_nassympt}) we have
\begin{equation*}
\begin{split}
\frac{|S_n^{(k)}|}{n!}&=\sum_{j=0}^np_{n-j}[H_k(f,z)]_{(j)}=
\sum_{j\leqslant n/2}A_k{{n-j+\gamma_0 -1}\choose {n-j}}[H_k(f,z)]_{(j)}
\\
&\quad+
O(n^{\gamma' -1}) +O\left(\frac{1}{n^2}\sum_{j\leqslant n/2}p_j\right)
\\
&=A_k\binom{n+\gamma_0 -1}{n}\sum_{j\leqslant n/2}[H_k(f,z)]_{(j)}+O\left( n^{\gamma_0-2}\log n \right)+O(n^{\gamma' -1})
\\
&\quad+O\left(\frac{p(e^{-1/n})}{n^2}\right)=
\frac{A_kn^{\gamma_0-1}}{\Gamma(\gamma_0)}H_k(1,1)+O\left( n^{\gamma_0-2}\log n \right)+O(n^{\gamma' -1}).
\end{split}
\end{equation*}
The theorem is proved.
\end{proof}

%{\bf Theorem 4.} {\it
\begin{thm}
\label{fundthm}
Let  $f(z)=\sum_{n=0}^{\infty}a_nz^n$
for $|z|<1$. Then for $n\geqslant 1$ we have
\begin{multline*}
\left|{\frac{1}{p_n}}\sum_{k=0}^na_kp_{n-k}-f(e^{-1/n})-{\frac{S(f;n)}{np_n}}\right|
\\
\leqslant C\left( {\frac{1}{n^\beta}}\sum_{j=1}^n {\frac{|S(f;j)|}{p(e^{-1/j})}}j^{\beta -1}+ {\frac{1}{p(e^{-1/n
})}}\sum_{j>n}{\frac{|S(f;j)|}{j}}e^{-j/n}\right),
\end{multline*}
where $S(f;m)=\sum_{k=1}^ma_kkp_{m-k}$, \ $\beta =\gamma_0 -\gamma'$,
and $C=C(k)$ -- constant which depends on  $k$ only, and $\gamma'=\max_{j\not=0}\bigl( \max\left\lbrace \gamma_j ,0 \right\rbrace \bigr)$ . %}
\end{thm}
%{\it Proof.}
\begin{proof}
The proof of this lemma is absolutely analogous to that of the theorem
  \ref{fundthm1}. Therefore we will repeat only the main steps.

From the proof of theorem \ref{fundthm1} we have
\begin{multline*}
R_n=\sum_{k=0}^na_kp_{n-k}-p_n
f(e^{-1/n})
-{\frac{S(f;n)}{n}}
\\
=\sum_{j=1}^{n-1}S(f;j)\left( f_{n-j,j}
-p_n \int_{0}^{e^{-1/n}} {\frac{x^{j-1}}{p(x)}}\,dx \right)
\\
+p_n
\sum_{j=n}^{\infty}S(f;j)\int_{0}^{e^{-1/n}} {\frac{x^{j-1}}{p(x)}}\,dx,
\end{multline*}
where
$$
f_{m,j}=\int_0^1g_{m,x}x^{j-1}\,dx\geqslant 0.
$$

Therefore
\begin{equation*}
\begin{split}
|R_n|&\leqslant  \sum_{1\leqslant j \leqslant n/2}|S(f;j)|\left|
\int_{0}^{1}{x^{j-1}}g_{n-j,x}
\,dx -p_{n-j} \int_{0}^{e^{-1/n}} {\frac{x^{j-1}}{p(x)}}\,dx
\right|
\\
&\quad+\sum_{1\leqslant j \leqslant n/2}|S(f;j)| |p_n -p_{n-j}|
\int_{0}^{e^{-1/n}} {\frac{x^{j-1}}{p(x)}}\,dx
\\
&\quad+
\sum_{n/2\leqslant j \leqslant n-1}|S(f;j)| f_{n-j,j}\,
dx
+p_n \sum_{j\geqslant n}|S(f;j)|\int_{0}^{e^{-1/n}}
{\frac{x^{j-1}}{p(x)}}\,dx.
\end{split}
\end{equation*}
In the proof of the Lemma \ref{f_mj} we did not use the condition
$d_j\geqslant d^->0$ therefore its estimate $f_{m,j}\ll
\frac{1}{j^2}$ for $j\geqslant m \geqslant1$ remains valid in the
present case also.

The proof of Theorem \ref{fundthm1}  gives
$$
\int_{e^{-1/n}}^{1}{x^{j-1}}g_{n-j,x}\,dx\ll{\frac{1}{n^2}},
$$
for $j\leqslant n/2$.

Applying these estimates together with the Lemma
 \ref{g_nx} and Theorem \ref{p_nassympt} we obtain the proof of the theorem.
 \end{proof}

Let us denote
$$
L_n(z)=\sum_{\scriptstyle 1\leqslant j\leqslant n \atop \scriptstyle (j,k)=1 }{\frac{\hat f(j)-1}{j}}z^k \quad \hbox{\it and}
\quad \quad \rho (p)=\left( \sum_{\scriptstyle 1\leqslant j\leqslant n \atop \scriptstyle (j,k)=1 }{\frac{|\hat f(j)-1|^p}{j}}
\right)^{1/p};
$$
we will also assume that
$$
\rho (\infty)=\lim_{p\to \infty}\max_{\scriptstyle 1\leqslant j \leqslant
n \atop \scriptstyle (j,k)=1}|\hat f(j) -1|.
$$

%

%{\bf Lemma 16.} {\it
\begin{lem}
\label{meanmp}
For any fixed $\infty
\geqslant p>\frac{1}{\beta}$ we have
$$
{\frac{M_n}{p_n}}=\exp \biggl\{ \sum_{\scriptstyle 1\leqslant j\leqslant n \atop \scriptstyle (j,k)=1 }{\frac{\hat f(j)-1}{j}} \biggr\} +O(\mu_n(p)).
$$%}
%\smallskip
\end{lem}
%{\it Proof.}
\begin{proof}
 We will suppose that $\hat f(j)=1$ for $j>n$.
We have
$$
F(z)=\sum_{m=0}^\infty M_mz^m=\exp \biggl\lbrace \sum_{(j,k)=1}\frac{\hat f(j)}{j}z^j \biggr\rbrace  =p(z)\exp\{ L_n(z) \}=p(z)h(z),
$$
where $h(z)=\sum_{j=0}^\infty h_jz^j=\exp\{ L_n(z) \}$. Therefore
$$
\frac{M_n}{p_n}=\frac{1}{p_n}\sum_{j=0}^nh_jp_{n-j}.
$$
Applying theorem  \ref{fundthm} with $a_j=h_j$, we have
\begin{equation*}
\begin{split}
S(m;h)&=[zp(z)h'(z)]_{(m)}=\left[zp(z)h(z)L_n'(z)\right]_{(m)}=\left[zF(z)L_n'(z)\right]_{(m)}
\\
&=\sum_{\scriptstyle 1\leqslant j \leqslant m \atop \scriptstyle (j,k)=1}{(\hat f(j)-1)}M_{m-j}.
\end{split}
\end{equation*}
Taking into account that   $|M_m|\leqslant
p_m=\frac{A_k}{\Gamma(\gamma_0)}m^{\gamma_0-1}(1+o(1))$,   and
applying  Cauchy inequality with parameters
$\frac{1}{p}+\frac{1}{q}=1$, we obtain
\begin{equation*}
\begin{split}
|S(m;h)|&\leqslant\biggl( \sum_{\scriptstyle 1\leqslant j \leqslant m \atop \scriptstyle (j,k)=1}|\hat f(j)-1|^p\biggr)^{1/p}\biggl( \sum_{j=0}^mp_j^q \biggr)^{1/q}
\ll  \mu_n(p)n^{1/p}
\biggl(1+ \sum_{j=1}^mj^{(\gamma_0-1)q} \biggr)^{1/q}
\\
&\ll\mu_n(p)n^{1/p}m^{\frac{(\gamma_0-1)q+1}{q}}\ll
\mu_n(p)\left( \frac{n}{m} \right)^{1/p}m^{\gamma_0}.
\end{split}
\end{equation*}
Inserting this estimate into the inequality of theorem \ref{fundthm}, we have
$$
\left|\frac{M_n}{p_n} -h(e^{-1/n})\right|\ll \mu_n(p).
$$
Since
$$
h(e^{1/n})=h(1)\bigl(1+O(\mu_n(p))\bigr)
$$
and
$$
h(1)=\exp\{ L_n(1)\}=\exp \biggl\{ \sum_{\scriptstyle 1\leqslant j\leqslant n \atop \scriptstyle (j,k)=1 }{\frac{\hat f(j)-1}{j}} \biggr\},$$
hence we obtain the proof of the lemma.
\end{proof}
%\smallskip
\begin{proof}[Proof of theorem \ref{mean}]
%{\it Proof of theorem 1.}
Because
$$
\frac{|S_n^{(k)}|}{n!}{\bf M_{n}}(f)=\sum_{j=0}^nM_{n-j}[H_k(f;j)]_{(j)},
$$
and $[H_k(f;z)]_{(j)}\leqslant [H_k(z)]_{(j)}=O(j^{-2})$,  $|M_m|\leqslant p_m=O(m^{\gamma_0 -1})$, then
$$
\frac{|S_n^{(k)}|}{n!}{\bf M_{n}}(f)=\sum_{j\leqslant
n/2}M_{n-j}[H_k(f;z)]_{(j)} +O\left( n^{\gamma_0-2} \right).
$$
Applying here  the asymptotic
$M_{n-j}=p_{n-j}\bigl(\exp\{L_{n-j}(1)\}+\mu_{n-j}(p)\bigr)$, we
obtain the proof of the theorem.
\end{proof}
%\smallskip
The proof of the next theorem is obtained by applying theorem \ref{fundthm} to function $h_n(z)=\exp\left\lbrace L_N(z)\right\rbrace
-\exp\left\lbrace L_n(1)\right\rbrace L_N(z)$
and using the estimates of Lemma \ref{hcoef}. The calculations are absolutely analogous to those of the proof of Theorem \ref{meanM1}, therefore we will not repeat them here
%\smallskip
\begin{thm}
\label{meanmn}
%{\bf Theorem 5.} {\it
For any fixed $\infty
\geqslant p>\frac{1}{\beta}$ there exists such positive $\delta=\delta (k,p)$ that if $\rho \leqslant
\delta$ then
$$
{\bf M_n}f=\frac{H(f;1)}{H(1,1)}\exp \{  L_N(1) \} \left( 1+
\sum_{\scriptstyle 1\leqslant j\leqslant N\atop \scriptstyle (j,k)=1}{\frac{\hat f(j)-1}{j}} \left( {\frac{p_{N-j}}{p_{N}}}-1 \right) +O\left(\rho^2 +\frac{1}{n^{\epsilon}} \right)\right),
$$
where $\epsilon>0$ -- fixed sufficiently small number.
\end{thm}
Let us define $M_m$ by means of relation
$$
M(z)=\exp \biggl\lbrace \sum_{{\scriptstyle j\leqslant n}\atop {\scriptstyle (j,k)=1}}\frac{\hat f(j)}{j}z^j \biggr\rbrace
=\exp \{ U_n(z) \}  =\sum_{m=0}^\infty M_mz^m,
$$
where $|{\hat f(j)}|\leqslant 1$; then $|M_m|\leqslant p_m$ for $m\geqslant 0$.
%\smallskip
%--------------------------------------------------------------------------------------------------------
%-------------------------------------------Lemma 13---------------------------------------------
%--------------------------------------------------------------------------------------------------------
%{\bf Lemma 13.}{\it
\begin{lem}
\label{mean_forlarge_t}
For any multiplicative function $f$ such that
$|\hat f(j)|\leqslant 1$, we have
$$
{\bf M_n}f\ll\left(\frac{1}{n^{2\gamma_0+1}}\int_{-\pi}^{\pi}|M(e^{ix})|^2|U_n'(e^{ix})|^2\,dx \right)^{\frac{\gamma_0}{2(\gamma_0 +1)}} +\frac{1}{n^{\gamma_0}}.
$$%}
\end{lem}
%{\it  Proof.}
\begin{proof}
Further we assume that $\hat f(j)=0$ for $j>n$, since $\hat f(j)$ for $j>n$ do not influence the value of  $M_n$.
Differentiating $M(z)$ by $z$
one can easily check that $M_n$ satisfy the recurrence relation
$$
 M_m=\frac{1}{m}\sum_{\scriptstyle (j,k)=1\atop \scriptstyle 1\leqslant j \leqslant m}{\hat f(j)}M_{m-j}
$$
for $m\geqslant 1$. Therefore for $1\leqslant T \leqslant n$ and $n/2\leqslant m \leqslant n$ we have
\begin{equation*}
\begin{split}
 |M_m|&\leqslant \frac{2}{n}\sum_{j=0}^{n-1}|M_{j}|\leqslant \frac{2}{n}\sum_{0\leqslant j \leqslant T}|M_{j}|+
 \frac{2}{nT}\sum_{T< j \leqslant n}j|M_{j}|
 \\
  &\leqslant 2e \frac{p(e^{-1/T})}{n}+2\frac{\sqrt{n}}{nT}
 \left( \sum_{j=1}^\infty|jM_j|^2 \right)^{1/2} .
 \end{split}
\end{equation*}
Putting here $T=[\epsilon n]$ with $1/n\leqslant\epsilon \leqslant 1$, for $n/2\leqslant m \leqslant n$ we obtain
\begin{equation*}
\begin{split}
\frac{|M_m|}{n^{\gamma_0-1}}&\ll \left( \frac{T}{n}\right)^{\gamma_0} +\frac{1}{n^{\gamma_0 -\frac{1}{2}}T}
\left( \int_{-\pi}^{\pi}|M'(e^{ix})|^2\,dx \right)^{1/2}
\\
&\ll
\epsilon^{\gamma_0} +\frac{1}{\epsilon}\left(\frac{1}{n^{2\gamma_0 +1}}\int_{-\pi}^{\pi}|M(e^{ix})|^2|U_n'(e^{ix})|^2\,dx \right)^{1/2}.
\\
&\ll\left(\frac{1}{n^{2\gamma_0+1}}\int_{-\pi}^{\pi}|M(e^{ix})|^2|U_n'(e^{ix})|^2\,dx \right)^{\frac{\gamma_0}{2(\gamma_0 +1)}} +\frac{1}{n^{\gamma_0}};
\end{split}
\end{equation*}
here we have taken the minimum by $1/n\leqslant\epsilon \leqslant 1$.

Lemma \ref{hcoef} gives $[H(f;z)]_{(m)}=O(m^{-2})$, and also  $|M_m|\leqslant p_m=O(m^{\gamma_0 -1})$, therefore
\begin{equation*}
\begin{split}
\frac{|S_n|}{n!}{\bf M_n}f&=\sum_{m=0}^nM_m[H(f;z)]_{(n-m)}
\\
&\ll  n^{\gamma_0 -1}\left(\frac{1}{n^{2\gamma_0+1}}\int_{-\pi}^{\pi}|M(e^{ix})|^2|U_n'(e^{ix})|^2\,dx \right)^{\frac{\gamma_0}{2(\gamma_0 +1)}} +\frac{1}{n}
+\frac{1}{n^2} \sum_{m\leqslant n/2}p_m
\\
&\ll
n^{\gamma_0 -1}\left(\frac{1}{n^{2\gamma_0+1}}\int_{-\pi}^{\pi}|M(e^{ix})|^2|U_n'(e^{ix})|^2\,dx \right)^{\frac{\gamma_0}{2(\gamma_0 +1)}} +\frac{1}{n} +n^{\gamma_0-2}.
\end{split}
\end{equation*}

The lemma is proved.
\end{proof}
\begin{corol}
\label{isvada}
Suppose $\hat f(j) \in \mathbb C$ and $|\hat f(j)|\leqslant 1$ for  $1\leqslant j \leqslant n$.
Let us denote
$$
J(n):=\min_{x\in[-\pi,\pi]}\sum_{\substack{1\leqslant j \leqslant n\\(j,k)=1}}\frac{1-\Re (\hat f(j)e^{-ix})}{j}
$$
Then
for the mean value of the corresponding multiplicative function $f$ we have
$$
{\bf M_n}f \ll \exp \left\{ -\frac{\gamma_0}{\gamma_0 +1}J(n)\right\} +\frac{1}{n^{\gamma_0}}
$$
\end{corol}
\begin{proof}
Applying Lemma \ref{mean_forlarge_t} and noticing that
$$
\max_{x\in[-\pi,\pi]}|M(e^{ix})|\ll n^{\gamma_0} \exp \left\{ -J(n)\right\},
$$
we have
\begin{equation*}
\begin{split}
{\bf M_n}f &\ll \exp \left\{ -
\frac{\gamma_0}{\gamma_0 +1}J(n)\right\}
\left(\frac{1}{n}\int_{-\pi}^{\pi}|U_n'(e^{ix})|^2\,dx \right)^{\frac{\gamma_0}{2(\gamma_0 +1)}}
 +\frac{1}{n^{\gamma_0}}
 \\
&\ll\exp \left\{ -
\frac{\gamma_0}{\gamma_0 +1}J(n)\right\}
 +\frac{1}{n^{\gamma_0}},
 \end{split}
\end{equation*}
because by means  Parseval identity have
$$
\int_{-\pi}^{\pi}|U_n'(e^{ix})|^2dx=2\pi\sum_{\substack{1\leqslant j\leqslant n\\
(j,k)=1}}|\hat f(j)|^2\ll n.
$$
Hence we obtain the estimate of the corollary.
\end{proof}
\begin{proof}[Proof of theorem \ref{halash}]
1) Divergence of series (\ref{halseries}) at every point $x\in[-\pi,\pi]$ implies that
$J(n)\to \infty$, whence by Corollary \ref{isvada} we have
$$
\lim_{n\to \infty} {\bf M_n^{(k)}}f=0.
$$
2) If series (\ref{halseries}) converges at some point $x_0\in[-\pi,\pi]$ then
$$
\frac{1}{n}\sum_{\substack{1\leqslant j\leqslant n\\
(j,k)=1}}|\hat f(j)e^{-ix_0j}-1|^2 \to \infty \quad \mbox{as}\quad n\to \infty .
$$
Application of Theorem \ref{mean}
to $\hat f(j) \to \hat f(j)e^{-itx_0}$
gives the desired estimate.

The theorem is proved.
\end{proof}

\section{Distribution of $\log P_n(\alpha)$ on $S_n^{(k)}$  }

We now apply the results of the previous section to study the distribution of the additive function
$$
\log P_n(\alpha(\sigma))=\sum_{j=1}^n\alpha_j(\sigma)\log j
$$
on $S_n^{(k)}$.

Let us introduce the notation
$$
u_n(t)={\bf M_n}\exp\biggl\{it{\frac{\log  P_n(\alpha)
}{   \sqrt{\frac{\phi(k)}{3k}}\log^{3/2}n   } }\biggr\}.
$$
Further we will denote
 $$
 t'=\frac{t}{   \sqrt{\frac{\phi(k)}{3k}}\log^{3/2}n   }
 $$.

%{\bf Lemma 14.}{\it
\begin{lem}
\label{ut}
$$
u_n(t)\ll\frac{1}{n^\epsilon},
$$
for $\delta\log^{1/2} n\leqslant t \leqslant \log^{3/2} n$. Here $\delta$ -- some fixed positive number.
\end{lem}

%{\it Proof.}
\begin{proof}
Putting in lemma \ref{mean_forlarge_t} $\hat f(j)=e^{it'\log j}$,
we obtain
$$
u_n(t)
\ll\left(\frac{1}{n^{2\gamma_0 +1}}\int_{0}^{1}|\exp\{U_n(e^{2\pi ix})\}|^2|U_n'(e^{2\pi ix})|^2dx \right)^{\frac{\gamma_0}{2(\gamma_0 +1)}} +\frac{1}{n^{\gamma_0}}.
$$
We have
\begin{equation*}
\begin{split}
U_n(e^{ix})&=\sum_{\scriptstyle 1\leqslant j\leqslant n \atop \scriptstyle (j,k)=1}\frac{e^{ixj}}{j^{1-it'}}
=\sum_{ j=1}^n\frac{e^{ixj}}{j^{1-it'}}\sum_{d|(j,k)}\mu(d)=
\sum_{d|k}\mu(d)
\sum_{\scriptstyle 1\leqslant j\leqslant n \atop \scriptstyle d|j}\frac{e^{ixj}}{j^{1-it'}}
\\
&=\sum_{d|k}\frac{\mu(d)}{d^{1-it'}}
\sum_{s=1}^{[n/d]}\frac{e^{ixds}}{s^{1-it'}}
=\sum_{d|k}\frac{\mu(d)}{d^{1-it'}}
\sum_{s=1}^{n}\frac{e^{ixds}}{s^{1-it'}} +O(1)
\\
&=\sum_{d|k}\frac{\mu(d)}{d^{1-it'}}
S_{t',n}(2\pi dx)+O(1),
\end{split}
\end{equation*}
where
$$
S_{t,n}(x)=\sum_{s=1}^{n}\frac{e^{2\pi ixs}}{s^{1-it}}.
$$
It is clare that $S_{t,n}(x)$ is periodic with period $1$: $S_{t,n}(x+1)=S_{t,n}(x)$.
It follows from (\ref{S(t)assympt})  that for $\frac{1}{n}\leqslant |x|\leqslant\frac{1}{2}$ we have
$$
S_{t,n}(x)=\sum_{s=1}^{n}\frac{e^{2\pi ixs}}{s^{1-it}}=\frac{|2\pi  x|^{-it}-1}{it} +O(1+|t|)
%\leqslant \log\frac{1}{2\pi |x|}+O(1+|t|)
.
$$
It follows hence
$$
|S_{t,n}(x)|= O_\eta(1),
$$
for $x\not\in \cup_{m \in \sym Z}(m-\eta,m+\eta)$, if  $\eta$ -- fixed such that $0<\eta<\frac{1}{2}$ and $|t|\leqslant 1$.

Let us put
$\eta:=\frac{1}{4k_0}$ and
$$V_\eta=[0,1]\cap\bigcup_{j=0}^{k_0}\left( \frac{j}{k_0}-\eta,\frac{j}{k_0}+\eta \right)$$.

If $x\in [0,1]$, then from $dx\in \cup_{m \in \sym
Z}(m-\eta,m+\eta)$,  for $d|k_0$ it follows that  $x\in V_\eta$.
Conversely, if $x\not\in V_{\eta}$, then $dx\not\in \cup_{m \in
\sym Z}(m-\eta,m+\eta)$ for any $d|k_0$. Therefore for $x\not\in
V_{\eta}$ we have
$$
U_n(e^{2\pi ix})=\sum_{d|k}\frac{\mu(d)}{d^{1-it'}}S_{t'}(dx)+O_\eta(1)=O_\eta(1).
$$
It follows hence that
\begin{multline*}
\frac{1}{n^{2\gamma_0 +1}}\int_{x\in [1,0]\setminus V_\eta}|\exp\{U_n(e^{2\pi ix})\}|^2|U_n'(e^{2\pi ix})|^2dx
\\
\ll \frac{1}{n^{2\gamma_0 +1}} \int_{x\in [1,0]}|U'_n(e^{2\pi ix})|^2dx
\ll \frac{1}{n^{2\gamma_0}},
\end{multline*}
the last inequality follows from the Parseval identity.

Let us estimate the integral over $x\in V_\eta$.
Let $x=\frac{s_0}{d_0}+u$,
$|u|\leqslant \eta
$, where $(s_0,d_0)=1$ and $d_0|k_0$ .
If $dx\in \cup_{m \in \sym Z}(m-\eta,m+\eta)$, then there exists such  $s$,
$1\leqslant s< k_0$ that $\left|x -\frac{s}{d}\right|<\frac{\eta}{d}\leqslant\eta$.
Since the intervals
$\left(  \frac{j}{k_0}-\eta ,\frac{j}{k_0}+\eta\right)$
do not intersect, then $\frac{s}{d}=\frac{s_0}{d_0}$,
therefore in view of the fact that
$s_0$ and $d_0$ are coprime we have $s_0|s$ and $d_0|d$. Therefore if $d_0\nmid d|k_0$, then
 $dx\not \in \cup_{m \in \sym Z}(m-\eta,m+\eta)$ and $S_t(dx)=O(1)$.

 Therefore for $\frac{1}{n}<|u|<\eta$ we have
\begin{equation*}
\begin{split}
U_n(e^{2\pi ix})&=\sum_{d:d_0|d|k_0}\frac{\mu(d)}{d^{1-it'}}S_{t,n}(dx)
+O(1)
\\
&=\sum_{d:d_0|d|k_0}\frac{\mu(d)}{d^{1-it'}}S_{t,n}\left(d\left( x-\frac{s_0}{d_0}\right)
+\frac{d}{d_0}s_0\right)
+O(1)
\\
&=\sum_{s|\frac{k_0}{d_0}}\frac{\mu(sd_0)}{(sd_0)^{1-it'}}S_{t,n}(sd_0u)
+O(1)
\\
&=\sum_{s|\frac{k_0}{d_0}}\frac{\mu(sd_0)}{sd_0}\left( \frac{|u|^{-it'}-|sd_0|^{it'}}{it'}\right)
+O(1)
\\
&=\frac{|u|^{-it'}-1}{it'}\frac{\mu(d_0)}{d_0}\sum_{s|\frac{k_0}{d_0}}\frac{\mu(s)}{s}+O(1)
\\
&=\frac{|u|^{-it'}-1}{it'}\frac{\mu(d_0)}{d_0}\prod_{p|\frac{k_0}{d_0}}\left(1-\frac{1}{p} \right) +O(1)
\\
&=\gamma_0\frac{|u|^{-it'}-1}{it'}{\mu(d_0)}\prod_{p|d_0}\frac{1}{p-1} +O(1).
\end{split}
\end{equation*}
If  $|u|\leqslant \frac{1}{n}$ similarly we have
$$
U_n(e^{2\pi ix})=S_{t'}(0) \gamma_0 {\mu(d_0)}\prod_{p|d_0}\frac{1}{p-1} +O(1).
$$
Therefore if $x\in\left[ \frac{s_0}{d_0}-\eta,\frac{s_0}{d_0}+\eta \right] $,
 $(s_0,d_0)=1$ and $d_0\not =2$, then
$$
|U_n(e^{2\pi ix})|\leqslant \frac{\gamma_0}{2}\log n +O(1).
$$
Therefore
$$
\left| u_n(t)\right| ^{\frac{2(\gamma_0+1)}{\gamma_0}}\ll\frac{1}{n^{2\gamma_0+1}}
\int_{[-\eta,\eta]\cup I_{1/2}}|\exp\{U_n(e^{2\pi ix})\}|^2|U_n'(e^{2\pi ix})|^2\,dx+O\left( \frac{1}{n^{\gamma_0}}\right),
$$
here
$$
I_{1/2}=\begin{cases}
\varnothing,&\text{if $2\nmid k$ }\\
\left[\frac{1}{2}-\eta,\frac{1}{2}+\eta \right],&\text{if $2|k$}
        \end{cases}.
$$
Since for $|t'|\leqslant n$ we have $\zeta(1-it')=\sum_{s=1}^{n}\frac{1}{s^{1-it'}}-\frac{n^{it'}}{it'}+O(n^{-1})$ (see. \cite{karatsuba})
and $\zeta(z)=\frac{1}{z-1}+O(1)$ for $|z-1|\leqslant 1$, then for
$|t|\leqslant 1$ we have
$$
S_t(0)=\sum_{s=1}^{n}\frac{1}{s^{1-it'}}=\frac{n^{it'}-1}{it'}+O\left(
1\right).
$$
Now we have when $2\nmid k$ we have $I_{1/2}=\varnothing$ and applying the same considerations as in Theorem \ref{th_turlarge} we have
\begin{multline*}
\left| u_n(t)\right|^{\frac{2(\gamma_0+1)}{\gamma_0}}\ll\frac{1}{n^{2\gamma_0 +1}}
\int_{-\eta}^{\eta}|\exp\{U_n(e^{2\pi ix})\}|^2|U'_n(e^{2\pi ix})|^2\,dx+\frac{1}{n^{\gamma_0}}
\\
\ll\frac{1}{n^{2\gamma_0 +1}}
\int_{\frac{1}{n}\leqslant|t|\leqslant\eta}\exp \left\{ 2\gamma_0\log {\frac{1}{|x|}}\left(
  {\frac{
\sin \left( t'\log {\frac{1}{|x|}} \right)
                       }{t'\log {1\over {|x|}}}}
\right)  \right\}
|U_n'(e^{2\pi ix})|^2\,dx
\\
+\frac{1}{n^{2\gamma_0 }}\exp\left\{ 2\gamma_0 \log n \frac{\sin(t'\log n)}{t'\log n} \right\}
+ \frac{1}{n^{\gamma_0}}
\ll \frac{1}{n^{c}}.
\end{multline*}
In a similar way we estimate the integral over the interval $I_{1/2}=\left[ \frac{1}{2}-\eta,\frac{1}{2}+\eta\right]$  when $2|k$.

The lemma is proved.
\end{proof}

%{\bf Theorem 6.} {\it
\begin{thm}
\label{clt_for_pn}
$$
 \sup_{x\in {\sym R}} \biggl|    \nu_{n}^{(k)} \biggl\{
 \frac{\log  P_n(\sigma)-{\bf M_n}\log P_n(\sigma)}{   \sqrt{\frac{\phi(k)}{3k}}\log^{3/2}n   }
 <x   \biggr\}-\Phi(x)
-
r_k(x){\frac{e^{-{x^2/2}}}{\sqrt{\log n}}} \biggr|
\ll\frac{1}{\log n},
$$
where $r_k(x)$ -- the same polynomial as in theorem \ref{clt}.
\end{thm}

%{\it Proof.}
\begin{proof}
Applying theorem  \ref{meanmn} with $p=\infty$ and
$\hat f(j)=\exp\left\{it\frac{\log j}{\sqrt{\frac{\phi(k)}{3k}}\log^{3/2}n}\right\}$ we obtain
\begin{equation*}
\begin{split}
u_n(t)&={\bf M_n}\exp\biggl\{it{\frac{\log  P_n(\alpha)
}{   \sqrt{\frac{\phi(k)}{3k}}\log^{3/2}n   } }\biggr\}
\\
&=\frac{H(f;1)}{H(1,1)}\exp\{L(1)\}\left(1 +\frac{it}{\sqrt{\frac{\phi(k)}{3k}}\log^{3/2}n}\sum_{\scriptstyle 1\leqslant j \leqslant n \atop (j,k)=1 \scriptstyle}\frac{\log j}{j}\left( \left( 1-\frac{j}{n} \right)^{\gamma_0-1} -1\right)\right.
\\
&\quad +\left.O\left( \frac{|t|^2+|t|}{\log n}+\frac{1}{n^\epsilon} \right)  \right),
\end{split}
\end{equation*}
for $|t|\leqslant \delta \sqrt{\log n}$. Putting as before $t'=\frac{t}{   \sqrt{\frac{\phi(k)}{3k}}\log^{3/2}n   }$ we have
\begin{equation*}
\begin{split}
L(1)&=\sum_{\scriptstyle 1\leqslant m\leqslant n \atop \scriptstyle (m,k)=1 }
\frac{m^{it'}-1}{m}=\sum_{d|k}\frac{\mu(d)}{d}
\sum_{ 1\leqslant l\leqslant n/d  }
\frac{(ld)^{it'}-1}{l}
\\
&=\sum_{d|k}\frac{\mu(d)}{d}\left(
d^{it'}\sum_{ 1\leqslant m\leqslant n/d  }
\frac{1}{m^{1-it'}}-\log \frac{n}{d} -c +O\left( \frac{1}{n}\right) \right)
\\
&=\sum_{d|k}\frac{\mu(d)}{d}\left(
d^{it'}\left( \zeta(1-it')+\frac{(n/d)^{it'}}{it'}  \right) -\log \frac{n}{d} -c +O\left( \frac{1}{n}\right) \right),
\end{split}
\end{equation*}
where $c$ -- Euler's constant. Here we have used the well known estimate  of  the Riemann Zeta function
$\zeta(s)=\sum_{m\leqslant
x}\frac{1}{m^{s}}+\frac{x^{1-s}}{s-1}+O(x^{-\Re s})$ which is true for
 $0<\sigma_0<\Re s<2$ and $x\geqslant \frac{|\Im s|}{\pi}$ (see
\cite{karatsuba}). As for  $|s-1|\leqslant 1$ we have
$\zeta(s)=\frac{1}{s-1}+c +O(|s-1|)$,
 therefore
\begin{equation*}
\begin{split}
L(1)&=\sum_{d|k}\frac{\mu(d)}{d}\left(
 \frac{n^{it'}-d^{it'}}{it'}  -\log \frac{n}{d} \right)  +O\left( |t'| +\frac{1}{n}\right)
 \\
 &=\gamma_0
 \frac{n^{it'}-1-it'\log n}{it'}   +O\left( |t'| +\frac{1}{n}\right)
\\
&=\gamma_0\left(\frac{it'}{2!}\log^2 n +\frac{(it')^2}{3!}\log^3 n+
\frac{(it')^3}{4!}\log^4 n
\right)     +O\left( |t'| +\frac{1}{n}+|t'|^4\log^5 n\right)
\\
&=\frac{\frac{it\gamma_0}{2}\log^2 n}{   \sqrt{\frac{\gamma_0}{3}}\log^{3/2}n   } -\frac{t^2}{2}+
\frac{3^{3/2}(it)^3}{4!\sqrt{\gamma_0}}\frac{1}{\log^{1/2} n}
    +O\left( \frac{|t|}{\log^{3/2} n} +\frac{1}{n}+\frac{|t|^4}{\log n}\right).
    \end{split}
\end{equation*}
Since $\frac{H(f;1)}{H(1;1)}=1+O\left( \frac{|t|}{\log^{3/2} n} \right)$ then finally we have
\begin{multline*}
u_n(t)e^{-it\sqrt{\frac{3\gamma_0}{4}\log
n}}=e^{-t^2/2}\left(1+\frac{1}{\sqrt{\log
n}}\sqrt{\frac{3}{\gamma_0}}\left( \frac{1}{8}(it)^3+C_0it
\right) \right.
\\
\left.+O\left(\frac{1}{n^\epsilon} +\frac{|t|+|t|^8}{\log n} \right)
\right)
\end{multline*}
for $|t|\leqslant \log^{1/6}n$. If $|t|\leqslant
\delta_0\sqrt{\log n}$, where $\delta_0$ --
fixed sufficiently small number then by means of similar calculations we have
$$
u_n(t)\ll e^{-t^2/2}.
$$
For small $t$ we will use crude estimate
$$
|u_n(t)-1|\leqslant |t|\frac{{\bf M_n}\log P_n(\alpha)
}{   \sqrt{\frac{\phi(k)}{3k}}\log^{3/2}n   }  \ll |t|\sqrt{\log
n}.
$$

We have
$$
q_n(t)=e^{-t^2/2}\left(1+\frac{1}{\sqrt{\log
n}}\sqrt{\frac{3}{\gamma_0}}\left( \frac{1}{8}(it)^3+C_0it \right)
\right)=\int_{-\infty}^{\infty}e^{itx}dQ_n(x),
$$
where
$$
Q_n(x)=\Phi(x)+
 \frac{1}{\sqrt{\log n}}\sqrt{\frac{3}{\gamma_0}}\left( \frac{(1-8C_0-x^2)}{8\sqrt{2\pi }}e^{-x^2/2}
\right).
$$
Applying the generalized inequality of Eseen \cite{petrov} we have
$$
\sup_{x\in \sym C}|F_n(x)-Q_n(x)|\ll
\int_{-T}^{T}\frac{|u_n(t)e^{-it\sqrt{\frac{3\gamma_0}{4}\log
n}}-q_n(t)|}{|t|} +O\left(\frac{1}{T}\right)
$$
with $T=\log n$
and using the earlier obtained estimates together with lemma \ref{ut}, we obtain the proof of the theorem.
\end{proof}

\section{Distribution of $\log O_n(\alpha)$ on $S_n^{(k)}$  }

While estimating the closeness of  $\log P_n(\alpha)$ and $\log O_n(\alpha)$
we, as before, will use the formula
$$
\log P_n(a)-\log O_n(a)=\sum_p \sum_{s\geqslant 1}(D_{n,p^s}-1)^+\log p,% \eqno(2)
$$
where the sum is taken over the all prime numbers, $a\in ({\sym Z}^+)^n$, $(d-1)^+=d-1+I[d=0]$ and
$$
D_{n,d}=D_{n,d}(a)=\sum_{j \leqslant n : d|j}a_j.
$$
Since $(d-1)^+=d-1+I[d=0]\geqslant 0$, therefore
$$
\Delta_{nd}:={\bf M_n}(D_{n,d}(\alpha)-1)^+=\nu_n(D_{n,d}=0)+{\bf M_n}D_{n,d}-1\geqslant 0.
$$

Putting in formula (\ref{f_gen_func}) $\hat f(j)=1$ for $d\not| j$ and
$\hat f(j)=0$ for $d | j$, we obtain
$$
\sum_{n=0}^\infty\frac{|S_n^{(k)}|}{n!}\nu_n(D_{n,d}=0)z^n=\prod_{d\nmid j}\left(
1+\sum_{{\scriptstyle s\geqslant 1}\atop{\scriptstyle q_k(j)|s}}\left(\frac{z^j}{j}\right)^s\frac{1}{s!} \right).
$$
In a similar way, putting $\hat f(j)=e^{it}$, for $d|j$ and $\hat f(j)=1$
in other cases,
we have
\begin{multline*}
\lefteqn{\sum_{n=0}^\infty\frac{|S_n^{(k)}|}{n!}{\bf M_n}e^{itD_{n,d}}z^n}
\\
=\prod_{j\geqslant 1:d\nmid j}\left(
1+\sum_{s\geqslant 1:q_k(j)|s}\left(\frac{z^j}{j}\right)^s\frac{1}{s!} \right)
\prod_{j\geqslant 1:d|j}\left(
1+\sum_{s\geqslant 1:q_k(j)|s}\left(\frac{e^{it}z^j}{j}\right)^s\frac{1}{s!} \right).
\end{multline*}
Differentiating this inequality by $t$ and putting $t=0$, we obtain
\begin{multline*}
\lefteqn{\sum_{n=0}^\infty\frac{|S_n^{(k)}|}{n!}{\bf M_n}D_{n,d}z^n}
\\
=\prod_{j=1}^\infty\left(
1+\sum_{s\geqslant 1:q_k(j)|s}\left(\frac{z^j}{j}\right)^s\frac{1}{s!} \right)
\sum_{j\geqslant 1:d|j}\frac{\displaystyle{\sum_{s\geqslant 1:q_k(j)|s} } \left(\frac{z^j}{j}\right)^s\frac{1}{(s-1)!}}{
1+\displaystyle{\sum_{s\geqslant 1:q_k(j)|s}}\left(\frac{z^j}{j}\right)^s\frac{1}{s!} }.
\end{multline*}
We have
$$
D_{n,d}=D_{n,d}'+D_{n,d}'',
$$
where
$$
D_{n,d}'=\sum_{\scriptstyle j\leqslant n \atop \scriptstyle (j,k)=1}a_j\quad\hbox{\it and} \quad
D_{n,d}''=\sum_{\scriptstyle j\leqslant n \atop \scriptstyle (j,k)>1}a_j.
$$
Later we will often use the estimate of the following lemma
\begin{lem} Let as before $c_n=\frac{|S_n^{(k)}|}{n!}$ then  we have
\label{l_cnestimate}
\begin{equation*}
\sum_{j\leqslant n:d|j}c_{n-j}\ll c_{n-[n/d]d}+\frac{n}{d}c_n,
\end{equation*}
for $d\leqslant n$.
\end{lem}
\begin{proof} Let $r=n-[n/d]d$ then  applying Theorem \ref{pavlovpn} we have
\begin{equation*}
\begin{split}
\sum_{j\leqslant n:d|j}c_{n-j}&=\sum_{s=0}^{[n/d]}c_{r+sd}\ll c_r+
\sum_{s=1}^{[n/d]}(r+sd)^{\gamma_0-1}\ll c_r+
\sum_{s=1}^{[n/d]}(sd)^{\gamma_0-1}
\\
&\ll
c_r+d^{\gamma_0-1}\left( \frac{n}{d} \right)^{\gamma_0}\ll
c_r+\frac{n^{\gamma_0}}{d}\ll c_{n-[n/d]d}+\frac{n}{d}c_n.
\end{split}
\end{equation*}
The lemma is proved.
\end{proof}

\begin{lem}
\label{md}
For $1\leqslant j\leqslant n$ and $(d,k)=1$ we have
$$
{\bf M_n}D_{n,d}'=\frac{\gamma_0}{d}\log\frac{n}{d} +O\left( \frac{1}{d}+\frac{c_{n-d[n/d]}}{nc_n}
%\frac{1}{n^{\gamma_0}}
\right);
$$
for any $1\leqslant d\leqslant n$ we have
$$
{\bf M_n}D_{n,d}''\ll \frac{1}{d^2}  + \frac{c_{n-d[n/d]}}{n^2c_n}.
$$
\end{lem}
%}

% {\it Proof.}
\begin{proof}
Suppose $(d,k)=1$.  Since
\begin{multline*}
\frac{|S_n^{(k)}|}{n!}{\bf M_n}D_{n,d}'
\\= \left[ \prod_{j=1}^\infty\left(
1+\sum_{s\geqslant 1:q_k(j)|s}\left(\frac{z^j}{j}\right)^s\frac{1}{s!} \right)
\sum_{\scriptstyle j\geqslant 1:d|j \atop \scriptstyle (k,j)=1}\frac{z^j}{j}\right]_{(n)}
=
\left[ F(z)
\sum_{\scriptstyle j\geqslant 1:d|j \atop \scriptstyle (k,j)=1}\frac{z^j}{j}\right]_{(n)},
\end{multline*}
where $F(z)=\sum_{j=0}^\infty c_jz^j$ and $c_m=\frac{|S_m^{(k)}|}{m!}=cm^{\gamma_0-1}\left( 1+
O(m^{-\epsilon})  \right)$, therefore
\begin{equation*}
\begin{split}
c_n{\bf M_n}D_{n,d}'&=\sum_{\scriptstyle  j\leqslant n:d|j  \atop (j,k)=1 \scriptstyle }
\frac{c_{n-j}}{j}=\frac{1}{d}\sum_{\scriptstyle  s\leqslant \frac{n}{2d}  \atop (j,k)=1 \scriptstyle}
\frac{c_{n-sd}}{s}+O\left( \frac{1}{n}
\sum_{ \frac{n}{2d}\leqslant s\leqslant \frac{n}{d}  }{c_{n-sd}} \right)
\\
&=\frac{c_n}{d}\sum_{\scriptstyle  s\leqslant \frac{n}{2d}  \atop (s,k)=1 \scriptstyle}
\frac{1}{s}+O\left( \frac{1}{d}\sum_{s\leqslant \frac{n}{2d}}
\frac{|c_{n-sd}-c_n|}{s}\right) +O\left( \frac{c_{n-d[n/d]}}{n}
%\frac{1}{n}
+\frac{c_n}{d} \right).
\end{split}
\end{equation*}
Since
$$
\sum_{\scriptstyle s\leqslant x \atop \scriptstyle (s,k)=1}\frac{1}{s}=\gamma_0
\log x +O\left( 1\right),
$$
then
\begin{multline*}
c_n{\bf M_n}D_{n,d}'=\gamma_0\frac{c_n}{d}\log \frac{n}{d}
\\
+
O\left( \frac{c_n}{d}\right) +
O\left(\frac{c_n}{d}+
\frac{1}{d} n^{\gamma_0-1-\epsilon}\log n\right)
+
O\left(\frac{c_{n-d[n/d]}}{n}
+\frac{c_n}{d}\right).
\end{multline*}
Dividing each side of this equation by $c_n$,
we obtain the first assertion of the lemma.

Let us prove the second estimate. We have
$$
\frac{|S_n^{(k)}|}{n!}{\bf M_n}D_{n,d}''\leqslant
\left[ F(z)
\sum_{\scriptstyle j\geqslant 1:d|j \atop \scriptstyle (k,j)>1} \sum_{s\geqslant 1:q_k(j)|s} \frac{z^{sj}}{j^s}\frac{1}{(s-1)!}     \right]_{(n)}.
$$
As here $s\geqslant 2$ and $(s-1)!\gg s^4$, then
 $$
\frac{|S_n^{(k)}|}{n!}{\bf M_n}D_{n,d}''
\ll
\left[ F(z)
\sum_{j\geqslant 1: d|j} \sum_{s\geqslant 1} \frac{z^{sj}}{j^2}\frac{1}{s^4}     \right]_{(n)}
=
\left[ F(z)
\sum_{m\geqslant 1:d|m}z^m \sum_{\scriptstyle j\geqslant 1:js=m \atop \scriptstyle d|j} \frac{1}{j^2s^4}  \right]_{(n)}
$$

$$
=\left[ F(z)
\sum_{m\geqslant 1:d|m}\frac{z^m}{m^2} \sum_{\scriptstyle j\geqslant 1:js=m \atop \scriptstyle d|j} \frac{1}{s^2}  \right]_{(n)}
\leqslant \left[ F(z)
\sum_{m\geqslant 1:d|m}\frac{z^m}{m^2}   \right]_{(n)}
\ll \frac{c_n}{d^2}  + \frac{c_{n-d[n/d]}}{n^2}.
$$
The lemma is proved.
\end{proof}

%-----------------------------------------------------------------------------------------------
%---------------------------------Lemma 4--------------------------------------------------
%-----------------------------------------------------------------------------------------------
%{\bf Lemma 4.} {\it
\begin{lem}
\label{meanalph}
For $d\geqslant \log n$  and $(d,k)=1$ we have
\begin{multline*}
0\leqslant {\bf M_n}(D_{n,d}(\alpha)-1)^+=\nu_n(D_{n,d}=0)+{\bf M_n}D_{n,d}-1
\\
\ll \left( \frac{\log n}{d} \right)^2 + \frac{\log n}{nd}\frac{c_{n-d[n/d]}}{c_n}.
\end{multline*}
If $(d,k)>1$, then
$$
0\leqslant {\bf M_n}(D_{n,d}(\alpha)-1)^+\leqslant {\bf M_n}D_{n,d}\ll \frac{1}{d^2}  + \frac{c_{n-d[n/d]}}{n^2c_n}.
$$
\end{lem}
%}
%{\it Proof.}
\begin{proof}
Suppose $(d,k)=1$. Let us denote $H_k(z)=H_k(1;z)$. Applying the
formulas which were obtained above together with lemma \ref{coef},
we have
\begin{equation*}
\begin{split}
\frac{|S_n^{(k)}|}{n!}\nu_n(D_{n,d}=0)&=\left[ \prod_{{\scriptstyle j\geqslant 1}\atop{\scriptstyle d\nmid j}}\left(
1+\sum_{{\scriptstyle s\geqslant 1}\atop{\scriptstyle q_k(j)|s}}\left(\frac{z^j}{j}\right)^s\frac{1}{s!} \right) \right]_{(n)}
\\
&=\left[\frac{p(z)}{p(z^d)^{1/d}} \prod_{\scriptstyle j\geqslant 1:d\nmid j \atop \scriptstyle (k,j)>1}\left(
1+\sum_{{\scriptstyle s\geqslant 1}\atop{\scriptstyle q_k(j)|s}}\left(\frac{z^j}{j}\right)^s\frac{1}{s!} \right) \right]_{(n)}
\\
&\leqslant \left[\frac{p(z)}{p(z^d)^{1/d}} \prod_{ {\scriptstyle j\geqslant 1}\atop{\scriptstyle(k,j)>1}}\left(
1+\sum_{{\scriptstyle s\geqslant 1}\atop{\scriptstyle q_k(j)|s}}\left(\frac{z^j}{j}\right)^s\frac{1}{s!} \right) \right]_{(n)}
\\
&=
\left[\frac{p(z)}{p(z^d)^{1/d}}H_k(z)\right]_{(n)}.
\end{split}
\end{equation*}
In a similar way we have
\begin{equation*}
\begin{split}
\frac{|S_n^{(k)}|}{n!}{\bf M_n}D_{n,d}&=\frac{|S_n^{(k)}|}{n!}{\bf M_n}D'_{n,d}
+\frac{|S_n^{(k)}|}{n!}{\bf M_n}D''_{n,d}
\\
&=\left[ p(z)H_k(z)
\sum_{\scriptstyle j\geqslant 1:d|j \atop \scriptstyle (k,j)=1}\frac{z^j}{j}\right]_{(n)}+
\frac{|S_n^{(k)}|}{n!}{\bf M_n}D''_{n,d}
\end{split}
\end{equation*}
Hence we have
\begin{multline*}
0\leqslant \frac{|S_n^{(k)}|}{n!}\Delta_{n,d}
\leqslant\left[ \frac{p(z)}{p(z^d)^{1/d}}H_k(z)\left( 1+p(z^d)^{1/d}\sum_{\scriptstyle j\geqslant 1:d|j \atop \scriptstyle (k,j)>1}\frac{z^j}{j}-p(z^d)^{1/d}\right) \right]_{(n)}
\\
+\frac{|S_n^{(k)}|}{n!}{\bf M_n}D''_{n,d}
=:S_1+S_2.
\end{multline*}
Since the coefficients in the Taylor expansion of the functions
 $\frac{p(z)}{p(z^d)^{1/d}}$ and $H_k(z)$ are positive, therefore $\left[ \frac{p(z)}{p(z^d)^{1/d}}H_k(z)\right]_{(m)}\geqslant 0$ for $m\geqslant 0$.
Putting  $\psi(z)=\sum_{\scriptstyle j\geqslant 1:d|j \atop \scriptstyle (k,j)>1}\frac{z^j}{j}$, we have $e^{\psi(z)}=p(z^d)^{1/d}$.
Applying here the estimate 4) of  lemma \ref{coef}, we obtain
$$
S_1\leqslant \left[ \frac{p(z)}{p(z^d)^{1/d}}H_k(z)\psi(z)\left(e^{\psi(z)}-1 \right) \right]_{(n)}.
$$
Since $[\psi(z)]_{(s)}\leqslant \left[ \frac{1}{d}\log \frac{1}{1-z^d}\right]_{(s)}$ and
$\left[ \frac{p(z)}{p(z^d)^{1/d}}\right] _{(s)}\leqslant [p(z)]_{(s)}$ for $s\geqslant 0$,
then applying here the estimates 1) and 3) of the lemma \ref{coef}, we have
\begin{multline*}
S_1\leqslant \left[ p(z)H_k(z)\frac{1}{d}\log \frac{1}{1-z^d}\left(\frac{1}{(1-z^d)^{1/d}}-1 \right) \right]_{(n)}
\\
\ll\left[ p(z)H_k(z)\frac{1}{d^2}\log^2 \frac{1}{1-z^d} \right]_{(n)}
\end{multline*}
for $d\geqslant \log n$.
Here we have used the fact that for
 $0<v <1$
\begin{equation*}
\begin{split}
\left[ \frac{1}{(1-z)^v}\right]_{(m)}&=\frac{v (v+1)\cdots (v +m-1)}{m!}
\\
&=\frac{v }{m}{\prod_{j=1}^{m-1}\left(1+\frac{v}{j}\right)}
\leqslant ve^2\left[ \log \frac{1}{1-z}\right]_{(m)},
\end{split}
\end{equation*}
if $1\leqslant m\leqslant n$ and $n$ is such that $v \log n \leqslant 1$.

Since $p(z)H_k(z)=F(z)=\sum_{j=0}^{\infty}c_jz^j$ with $c_m=\frac{|S_m^{k}|}{m!}$, then
finally we obtain
\begin{multline*}
S_1\ll \frac{1}{d^2}\sum_{1\leqslant j \leqslant \frac{n}{d}}\frac{\log j}{j}c_{n-jd}
\ll \frac{c_n}{d^2}\sum_{1\leqslant j \leqslant \frac{n}{2d}}\frac{\log j}{j}
+ \frac{1}{d^2}\sum_{\frac{n}{2d}\leqslant j \leqslant \frac{n}{d}-1}\frac{\log j}{j}c_{n-jd} \\+\frac{\log n}{nd}c_{n-d[n/d]}
\ll c_n\left( \frac{\log n}{d} \right)^2 + \frac{\log n}{nd}c_{n-d[n/d]}.
\end{multline*}
Lemma \ref{md} yields the estimate of  $S_2$: $S_2\ll \frac{c_n}{d^2}+
\frac{c_{n-d[n/d]}}{n^2}.$

Suppose now $(d,k)>1$. Then $D_{n,d}=D_{n,d}''$ and the desired estimate follows from the lemma \ref{md}.

The lemma is proved.
\end{proof}

%-----------------------------------------------------------------------------------------------------
%----------------------------END OF PROOF----------------------------------------------------
%-----------------------------------------------------------------------------------------------------
%\smallskip

%{\bf Lemma 5.} {\it
\begin{lem}
\label{coefp}
There exists such a positive constant $\epsilon$, that
$$
\left[\frac{p(z)}{p(z^d)^{\frac{1}{d}}}\right]_{(n)}
=\frac{d^{\gamma_0/d}A_k^{1-1/d}}{\Gamma\left( \gamma_0\left(1-\frac{1}{d} \right) \right) }n^{\gamma_0\left(1-\frac{1}{d} \right)-1}
\bigl(1+O(n^{-\epsilon}) \bigr),
$$
if $C<d\leqslant \log ^{D}n$. Here $D$ --
is an arbitrary fixed positive constant and  $S$ --
sufficiently big positive number.
%}
\end{lem}
%{\it  Proof.}
\begin{proof}

Let us denote by $K_d=C_2\cup\cup_{j=0}^{dk_0}L_j$ the integration contour, here $L_j$ denotes the sides of the intervals on the complex plain, connecting the points
 $e^{\frac{2\pi ij}{dk_0}}$ and $2e^{\frac{2\pi ij}{dk_0}}$. $C_2$ is the circle $|z|=2$. As  $\frac{1}{|p(z^d)|^{\frac{1}{d}}}\ll \frac{1}{(|z|^d-1)^{\frac{\gamma_0}{d}}}\ll \frac{1}{(|z|-1)^{\frac{\gamma_0}{d}}}$
for $1\leqslant |z|\leqslant 2$ and $p(z)\ll d^{\gamma_0}$ for $z\in L_j$, $k_0\nmid j$ then
applying the formula of Cauchy, we have
\begin{multline*}
\left[ \frac{p(z)}{p(z^d)^{\frac{1}{d}}}\right]_{(n)}=\frac{1}{2\pi i}\int_{K_d}\frac{p(z)}{p(z^d)^{\frac{1}{d}}}\frac{dz}{z^{n+1}}
=
\sum_{ j=0}^{dk_0-1}
\frac{1}{2\pi i}\int_{L_j}\frac{p(z)}{p(z^d)^{\frac{1}{d}}}\frac{dz}{z^{n+1}}
\\
\shoveleft{=\frac{1}{2\pi i}\int_{L_0}\frac{p(z)}{p(z^d)^{\frac{1}{d}}}\frac{dz}{z^{n+1}}
+O\left(\sum_{\scriptstyle 0< j< {dk_0} \atop \scriptstyle d\nmid j}d^{\gamma_0}\int_1^2
\frac{dr}{(r-1)^{\frac{\gamma_0}{d}}r^{n+1}} \right. }
\\
\shoveright{\left.+\sum_{s=1}^{k_0-1}\int_1^2
\frac{dr}{(r-1)^{\frac{\gamma_0}{d}+\gamma_s}r^{n+1}}\right) +O(2^{-n})}
\\
=I+O\left( d^{\gamma_0+1}n^{\frac{\gamma_{0}}{d}-1}
+\sum_{s=1}^{k_0-1}
n^{\gamma_s+\frac{\gamma_0}{d}-1}\right)
=
I+O\left( d^{\gamma_0+1}n^{\frac{\gamma_{0}}{d}-1}
+n^{\gamma'+\frac{\gamma_0}{d}-1}\right) ,
\end{multline*}
where $\gamma'=\max_{j\not=0}\gamma_j <\gamma_0$ and
\begin{equation}
\label{intg}
\begin{split}
I&=
\frac{1}{2\pi i}\int_{L(1,1+\frac{1}{2d})}\frac{p(z)}{p(z^d)^{\frac{1}{d}}}\frac{dz}{z^{n+1}}
+O\left( \int_{1+\frac{1}{2d}}^2\frac{1}{(r-1)^{\gamma_0}}\frac{dr}{r^{n+1}}\right)
\\
&=\frac{1}{2\pi i}\int_{L(1,1+\frac{1}{2d})}\frac{p(z)}{p(z^d)^{\frac{1}{d}}}\frac{dz}{z^{n+1}}
+O\left( d^{\gamma_0}e^{-\frac{n}{2d+1}}\right) ,
\end{split}
\end{equation}
and $L(1,1+\frac{1}{2d})$ is the part of  $L_0$, which lies inside of disc $|z-1|<\frac{1}{2d}$.
Here we have used the fact that
$|z^d|>c>1$ for
$|z|>1+\frac{1}{2d}$ and $|p(w)|\to1$,  if  $|w|\to \infty$.

Suppose $S=\{z\in {\sym C}:|\arg (z)|<1/(2k_0)\}$. Then for $z\in S$ we have
$$
p(z)=\prod_{j=0}^{k_0-1}
\frac{1}{(1-ze^{-2\pi i j/k})^{\gamma_j}}=\frac{u(z)}{(1-z)^{\gamma_0}},
$$
here $u(z)$ -- analytic function in the sector  $S$.
If  $z$ is such, that $|\arg(z)|<1/(4dk_0)$, then $z\in S$ and $z^d\in S$.
The previous formula yields
$$
\frac{p(z)}{p(z^d)^{\frac{1}{d}}}=
\frac{(1-z^d)^{\frac{\gamma_0}{d}}}{(1-z)^{\gamma_0}}
\frac{u(z)}{u(z^d)^{1/d}}=\frac{1}{(1-z)^{\gamma_0(1-1/d)}}\left( \frac{1-z^d}{1-z} \right)^{\frac{\gamma_0}{d}} \frac{u(z)}{u(z^d)^{1/d}}.
$$
With the same restrictions on $z$ we have
\begin{equation*}
\begin{split}
\left( \frac{1-z^d}{1-z} \right)^{\frac{\gamma_0}{d}}-d^{\frac{\gamma_0}{d}}
&=(1+z+\cdots+z^{d-1})^{\frac{\gamma_0}{d}}-d^{\frac{\gamma_0}{d}}
\\
&=d^{\frac{\gamma_0}{d}}
\left( \left(\frac{1+z+\cdots+z^{d-1}}{d} \right)^{\frac{\gamma_0}{d}}-1  \right).
\end{split}
\end{equation*}
Since $|z|\leqslant 1+ |z-1|\leqslant e^{|z-1|}$, then
$$
|z^j-1|=|z-1||1+z+\cdots+z^{j-1}|\leqslant j|z-1|e^{j|z-1|},
$$
therefore for  $|z-1|\leqslant 1/(2d)$
we have
\begin{multline*}
\left| \frac{(z-1)+(z^2-1)+\cdots+(z^{d-1}-1)}{d} \right|
\\
\leqslant
e\frac{|z-1|+2|z-1|+\cdots+(d-1)|z-1|}{d}
=\frac{ed}{2}|z-1|\leqslant \frac{e}{4}.
\end{multline*}
Applying this inequality together with the estimate $(1+z)^\theta=1+O(\theta |z|)$  for $z$ such that $|\arg(z)|\leqslant 1/(4k_0d)$ and $|1-z|\leqslant \frac{1}{4d}$, we obtain
\begin{equation}
\label{zdiference}
\left|\left( \frac{1-z^d}{1-z} \right)^{\frac{\gamma_0}{d}}-d^{\frac{\gamma_0}{d}}\right|
\ll |1-z|.
\end{equation}
In a similar way, having the same restrictions on $z$ we obtain
\begin{equation*}
\begin{split}
\left|\frac{u(z)}{u(z^d)^{\frac{1}{d}}}-\frac{u(1)}{u(1)^{\frac{1}{d}}}\right|&=
\left|\frac{u(z)-u(1)}{u(z^d)^{\frac{1}{d}}}+u(1)
\left(\frac{1}{u(z^d)^{\frac{1}{d}}}-\frac{1}{u(1)^{\frac{1}{d}}}\right) \right|
\\
&\ll |z-1|+
\left| \left(\frac{u(z^d)}{u(1)} \right)^{\frac{1}{d}} -1 \right| \ll |z-1|.
\end{split}
\end{equation*}
Finally we obtain that for $|\arg(z)|<1/(4dk_0)$ and $|1-z|\leqslant \frac{1}{4d}$ holds the estimate
$$
\frac{p(z)}{p(z^d)^{\frac{1}{d}}}=\frac{d^{\gamma_0/d}{u(1)}^{1-1/d}+O(|1-z|)}{(1-z)^{\gamma_0(1-1/d)}}.
$$
Putting this  estimate in (\ref{intg}) and recalling that $u(1)=A_k$, we obtain
% the prove of the lemma.
\begin{equation*}
\begin{split}
I&=\frac{1}{2\pi i}\int_{L(1,1+\frac{1}{2d})}\frac{d^{\gamma_0/d}{A_k}^{1-1/d}+O(|1-z|)}{(1-z)^{\gamma_0(1-1/d)}z^{n+1}}dz
+O\left( d^{\gamma_0}e^{-\frac{n}{2d+1}}\right)
\\
&=\frac{1}{2\pi i}\int_{L(1,1+\frac{1}{2d})}\frac{d^{\gamma_0/d}{A_k}^{1-1/d}}{(1-z)^{\gamma_0(1-1/d)}z^{n+1}}dz
+O\left( n^{\gamma_0(1-1/d)-2}+d^{\gamma_0}e^{-\frac{n}{2d+1}}\right)
\\
&=\frac{d^{\gamma_0/d}A_k^{1-1/d}}{\Gamma\left( \gamma_0\left(1-\frac{1}{d} \right) \right) }n^{\gamma_0\left(1-\frac{1}{d} \right)-1}
\left(1+O\left(\frac{1}{n}\right) \right)+O\left( n^{\gamma_0(1-1/d)-2}+d^{\gamma_0}e^{-\frac{n}{2d+1}}\right),
\end{split}
\end{equation*}
because for $v>0$
\begin{equation*}
\begin{split}
\label{flajform}
\frac{1}{2\pi i}\int_{L(1,1+\frac{1}{2d})}\frac{dz}{(1-z)^{v}z^{n+1}}&=
\frac{1}{2\pi i}\int_{L_0\cup C_2}\frac{dz}{(1-z)^{v}z^{n+1}}
+O\left(2^{-n} + d^{v}e^{-\frac{n}{2d+1}}\right)
\\
&=\binom{n+v-1}{n}+O\left(d^{v}e^{-\frac{n}{2d+1}}\right)
\\
&=\frac{n^{v-1}}{\Gamma(v)}\left( 1+O\left(\frac{1}{n}\right)\right)
+O\left(d^{v}e^{-\frac{n}{2d+1}}\right),
\end{split}
\end{equation*}
the last estimate was proved in \cite{flajolet}.
\end{proof}

The proof of the next lemma is analogous.
\begin{lem}
\label{coefprl}
There exists such a positive constant $\epsilon$, that for  $C<l,r\leqslant \log ^{D}n$
and $(l,k)=1$, $(r,k)=1$, we have
$$
\left[\frac{p(z)p(z^{rl})^{\frac{1}{rl}}}{p(z^r)^{\frac{1}{r}}p(z^l)^{\frac{1}{l}}}\right]_{(n)}
=\frac{l^{\gamma_0/l}r^{\gamma_0/r}A_k^{1-\frac{1}{l}-\frac{1}{r}+\frac{1}{rl}}}{(lr)^{\frac{\gamma_0}{lr}}\Gamma\left( \gamma_0\left(1-\frac{1}{l}-\frac{1}{r}+
\frac{1}{lr} \right) \right) }n^{\gamma_0\left(1-\frac{1}{l}-\frac{1}{r}+\frac{1}{lr} \right)-1}
\bigl(1+O(n^{-\epsilon}) \bigr).
$$
\end{lem}%}
\begin{proof} As in proof of lemma \ref{coefp} we denote $K_{rl}=C_2\cup\cup_{j=0}^{rlk_0}L_j$ the integration contour, here $L_j$ denotes the sides of the intervals on the complex plain, connecting the points
 $e^{\frac{2\pi ij}{lrk_0}}$ and $2e^{\frac{2\pi ij}{lrk_0}}$. $C_2$, as before, is the circle $|z|=2$.
\begin{multline*}
\left[\frac{p(z)p(z^{rl})^{\frac{1}{rl}}}{p(z^r)^{\frac{1}{r}}p(z^l)^{\frac{1}{l}}}\right]_{(n)}=
\frac{1}{2\pi i}\int_{K_{rl}}\frac{p(z)p(z^{rl})^{\frac{1}{rl}}}{p(z^r)^{\frac{1}{r}}p(z^l)^{\frac{1}{l}}}
\frac{dz}{z^{n+1}}
\\
%\frac{1}{2\pi i}\int_{L_0}\frac{p(z)}{p(z^d)^{\frac{1}{d}}}\frac{dz}{z^{n+1}}
\shoveleft{=
\frac{1}{2\pi i}\int_{L_0}\frac{p(z)p(z^{rl})^{\frac{1}{rl}}}{p(z^r)^{\frac{1}{r}}p(z^l)^{\frac{1}{l}}}
\frac{dz}{z^{n+1}}+O\left(\sum_{\scriptstyle 0< j< {lrk_0} \atop \scriptstyle d\nmid j}(lr)^{\gamma_0}\int_1^2
\frac{dr}{(r-1)^{\gamma_0\left(\frac{1}{l}+\frac{1}{r}+\frac{1}{lr}\right)}r^{n+1}}\right.}
\\
\shoveright{\left.+\sum_{s=1}^{k_0-1}\int_1^2
\frac{dr}{(r-1)^{\gamma_0\left(\frac{1}{l}+\frac{1}{r}+\frac{1}{lr}\right)+\gamma_s}r^{n+1}}\right) +O(2^{-n})}
\\
\shoveleft{=\frac{1}{2\pi i}\int_{L_0}\frac{p(z)p(z^{rl})^{\frac{1}{rl}}}{p(z^r)^{\frac{1}{r}}p(z^l)^{\frac{1}{l}}}
\frac{dz}{z^{n+1}}}
\\+
O\left((rl)^{\gamma_0+1}n^{\gamma_0\left(\frac{1}{l}+\frac{1}{r}+\frac{1}{lr}\right)-1}
+n^{\gamma_0\left(\frac{1}{l}+\frac{1}{r}+\frac{1}{lr}\right)+\gamma'-1}
\right),
\end{multline*}
where $\gamma'=\max_{j\not=0}\gamma_j <\gamma_0$.
Applying estimate  (\ref{zdiference}) with $d=rl$ we obtain
\begin{multline*}
\frac{p(z)p(z^{rl})^{\frac{1}{rl}}}{p(z^r)^{\frac{1}{r}}p(z^l)^{\frac{1}{l}}}
% \frac{(1-z^d)^{\frac{\gamma_0}{d}}}{(1-z)^{\gamma_0}}
% \frac{u(z)}{u(z^d)}
\\=\frac{1}{(1-z)^{\gamma_0\left(1-\frac{1}{l}-\frac{1}{r}+\frac{1}{lr}\right)}}\left( \frac{1-z^l}{1-z} \right)^{\frac{\gamma_0}{l}}
\left( \frac{1-z^r}{1-z} \right)^{\frac{\gamma_0}{r}}
\left( \frac{1-z}{1-z^{rl}} \right)^{\frac{\gamma_0}{rl}}
 \frac{u(z)u(z^{lr})^{\frac{1}{lr}}}{u(z^l)^{\frac{1}{l}}u(z^r)^{\frac{1}{r}}}
\\
=\frac{r^{\gamma_0/r}l^{\gamma_0/l}u(1)^{\gamma_0\left(1-\frac{1}{l}-\frac{1}{r}+\frac{1}{lr}\right)}+O(|z-1|)}{(rl)^{\gamma_0/rl}(1-z)^{\gamma_0\left(1-\frac{1}{l}-\frac{1}{r}+\frac{1}{lr}\right)}},
\end{multline*}
when $|z-1|\leqslant\frac{1}{4k_0rl}$. Denoting as before by $L(1,1+\frac{1}{4lr})$ the part of  $L_0$, which lies inside of disc $|z-1|<\frac{1}{4rl}$, we have
\begin{multline*}
\frac{1}{2\pi i}\int_{L_0}\frac{p(z)p(z^{rl})^{\frac{1}{rl}}}{p(z^r)^{\frac{1}{r}}p(z^l)^{\frac{1}{l}}}
\frac{dz}{z^{n+1}}=
\frac{1}{2\pi i}\int_{L(1,1+\frac{1}{4lr})}\frac{p(z)p(z^{rl})^{\frac{1}{rl}}}{p(z^r)^{\frac{1}{r}}p(z^l)^{\frac{1}{l}}}
\frac{dz}{z^{n+1}}
\\
\shoveright{+O\left( \int_{1+\frac{1}{4lr}}^2\frac{dy}{(y-1)^{\gamma_0\left(1+\frac{1}{l}+\frac{1}{r}+\frac{1}{lr}\right)}y^{n+1}}\right)}
\\
=\frac{1}{2\pi i}\int_{L(1,1+\frac{1}{4lr})}
\frac{r^{\gamma_0/r}l^{\gamma_0/l}u(1)^{\gamma_0\left(1-\frac{1}{l}-\frac{1}{r}+\frac{1}{lr}\right)}+O(|z-1|)}{(rl)^{\gamma_0/rl}(1-z)^{\gamma_0\left(1-\frac{1}{l}-\frac{1}{r}+\frac{1}{lr}\right)}z^{n+1}}
+O\left((lr)^{\gamma_0}e^{-\frac{n}{4lr+1}}\right)
\\
\shoveleft{=\frac{r^{\gamma_0/r}l^{\gamma_0/l}u(1)^{\gamma_0\left(1-\frac{1}{l}-\frac{1}{r}+\frac{1}{lr}\right)}}{(rl)^{\gamma_0/rl}2\pi i}\int_{L(1,1+\frac{1}{4lr})}\frac{dz}{(1-z)^{\gamma_0\left(1-\frac{1}{l}-\frac{1}{r}+\frac{1}{lr}\right)}z^{n+1}}    }
\\
+O\left(n^{\gamma_0\left(1-\frac{1}{l}-\frac{1}{r}+\frac{1}{lr}\right)-2}\right)+O\left((lr)^{\gamma_0}e^{-\frac{n}{4lr+1}}\right).
\end{multline*}
Applying here (\ref{flajform}) with $d=2rl$ we complete the proof of the lemma.

\end{proof}

Let us denote
%\smallskip
$$
H_k^{(d)}(z)=\prod_{j\geqslant 1:d|j}\left(
1+\sum_{s\geqslant 1:q_k(j)|s}\left(\frac{z^j}{j}\right)^s\frac{1}{s!} \right).
$$

%{\bf Lemma 7.} {\it
\begin{lem}
\label{nud}
For $(k,d)=1$ and $C<d\leqslant \log ^{D}n$ we have
$$
\nu_n(D_{nd}=0)=\frac{\Gamma(\gamma_0)d^{\gamma_0/d}A_k^{-1/d}}{\Gamma\left( \gamma_0\left(1-\frac{1}{d} \right) \right) H_k^{(d)}(1)}n^{-\frac{\gamma_0}{d} }
\bigl(1+O(n^{-\epsilon}) \bigr).
$$
\end{lem}
%{\it Proof.}
\begin{proof}
Since  $\left[\frac{H_k(z)}{H_k^{(d)}(z)}\right]_{(m)}\leqslant [H_k(z)]_{(m)}=O(m^{-2})$ for
$m\geqslant 1$, then for \hbox{$(k,d)=1$}
\begin{multline*}
\frac{|S_n^{(k)}|}{n!}\nu_n(D_{n,d}=0)=\left[\frac{p(z)}{p(z^d)^{\frac{1}{d}}}
\frac{H_k(z)}{H_k^{(d)}(z)}\right]_{(n)}
=\sum_{j=0}^n\left[\frac{p(z)}{p(z^d)^{\frac{1}{d}}}\right]_{(j)}
\left[
\frac{H_k(z)}{H_k^{(d)}(z)}\right]_{(n-j)}
\\
\shoveleft{=\sum_{n/2<j\leqslant n}\left[\frac{p(z)}{p(z^d)^{\frac{1}{d}}}\right]_{(j)}
\left[\frac{H_k(z)}{H_k^{(d)}(z)}\right]_{(n-j)}+
O\left(\frac{1}{n^2}\sum_{j\leqslant n/2}\left[\frac{p(z)}{p(z^d)^{\frac{1}{d}}}\right]_{(j)}\right)}
\\
\shoveleft{=\sum_{s\leqslant n/2}\left[\frac{p(z)}{p(z^d)^{\frac{1}{d}}}\right]_{(n-s)}
\left[\frac{H_k(z)}{H_k^{(d)}(z)}\right]_{(s)}+
O\left(\frac{1}{n^2}\frac{p(e^{-1/n})}{p(e^{d/n})^{\frac{1}{d}}}\right)}
\\
\shoveleft{=\sum_{s\leqslant n/2}\left(\frac{d^{\gamma_0/d}A_k^{1-1/d}}{\Gamma\left( \gamma_0\left(1-\frac{1}{d} \right) \right) }(n-s)^{\gamma_0\left(1-\frac{1}{d} \right)-1}
\bigl(1+O(n^{-\epsilon}) \bigr) \right)
\left[\frac{H_k(z)}{H_k^{(d)}(z)}\right]_{(s)} }
\\
\shoveright{+
O\left(n^{\gamma_0(1-\frac{1}{d})-2}\right)}
\\
=\frac{d^{\gamma_0/d}A_k^{1-1/d}}{\Gamma\left( \gamma_0\left(1-\frac{1}{d} \right) \right) }n^{\gamma_0\left(1-\frac{1}{d} \right)-1}
\frac{H_k(1)}{H_k^{(d)}(1)}+
O\left(n^{\gamma_0(1-\frac{1}{d})-1-\epsilon}\right).\hfill
\end{multline*}

As  $\frac{|S_n^{(k)}|}{n!}=\frac{A_kn^{\gamma_0-1}}{\Gamma(\gamma_0)}(1+O(n^{-\epsilon})$, we obtain the proof of the lemma.
\end{proof}

%---------------------------------------------------------------------------------------------------------
%-------------------------------------Lemma-----------------------------------------------------------
%---------------------------------------------------------------------------------------------------------
%{\bf Lemma 8.} {\it
\begin{lem}
\label{nulr}
If $(k,l)=1$, $(k,r)=1$ and $C<l,r\leqslant \log ^{D}n$,
then
\begin{multline*}
\nu_n(D_{n,l}=0,D_{n,r}=0)
\\=\frac{H_k^{(lr)}(1)}{H_k^{(l)}(1)H_k^{(r)}(1)}\frac{\Gamma(\gamma_0)l^{\frac{\gamma_0}{l}}
r^{\frac{\gamma_0}{r}}A_k^{-\frac{1}{l}-\frac{1}{r}+\frac{1}{rl}}}{(lr)^{\frac{\gamma_0}{lr}}\Gamma\left( \gamma_0\left(1-\frac{1}{l}-\frac{1}{r}+
\frac{1}{lr} \right) \right) }n^{\gamma_0\left(-\frac{1}{l}-\frac{1}{r}+\frac{1}{lr} \right)}
\bigl(1+O(n^{-\epsilon}) \bigr).
\end{multline*}
\end{lem}

%{\it Proof.}
\begin{proof}
For $(k,l)=1$, $(k,r)=1$ we have
$$
\sum_{n=0}^\infty\frac{|S_n^{(k)}|}{n!}\nu_n(D_{n,r}=0,D_{n,l}=0)z^n=
\frac{p(z)p(z^{lr})^{\frac{1}{lr}}}{p(z^l)^{\frac{1}{l}}p(z^r)^{\frac{1}{r}}}
\frac{H_k(z)H_k^{(lr)}(z)}{H_k^{(l)}(z)H_k^{(r)}(z)}.
$$
Applying the same considerations as in lemma \ref{nud}, we complete the proof of the lemma.
\end{proof}

%
%---------------------------------------------------------------------------------------------
%--------------------------Lemma 9-----------------------------------------------------------
%------------------------------------------------------------------------------------------------

%{\bf Lemma 9.} {\it
\begin{lem}
\label{fxy}
Suppose $f\in C^2[-\epsilon,\epsilon]$ and $f(0)=1$, then
$$
f(x+y)-f(x)f(y)\ll |xy|,
$$
for $|x|\leqslant \frac{\epsilon}{2}$ and $|y|\leqslant \frac{\epsilon}{2}$.
% }
% \smallskip
\end{lem}
%{\it Proof. }
\begin{proof}
We have
\begin{equation*}
\begin{split}
f(x+y)-f(x)f(y)&=\int_0^1\frac{d}{dt}\bigl(f(xt+yt)-f(xt)f(yt) \bigr)dt
\\
&=x\int_0^1\bigl(f'(xt+yt)-f'(xt)f(yt) \bigr)dt
\\
&\quad+y\int_0^1\bigl(f'(xt+yt)-f(xt)f'(yt) \bigr)dt.
\end{split}
\end{equation*}
Putting here $f(yt)=1+O(t|y|)$ in the first integral, and $f(xt)=1+O(t|x|)$ in the second,
we have
\begin{multline*}
f(x+y)-f(x)f(y)=x\int_0^1\bigl(f'(xt+yt)-f'(xt) \bigr)dt
\\+y\int_0^1\bigl(f'(xt+yt)-f'(yt) \bigr)dt+O(|xy|)
\ll |xy|.
\end{multline*}
The lemma is proved.
\end{proof}
%\smallskip
%------------------------------------------------------------------------------------------------
%------------------------------------End of Lemma------------------------------------------
%------------------------------------------------------------------------------------------------

%----------------------------------------------------------------------------------------------------
%-----------------------------Proof of theorem 3------------------------------------------------
%--------------------------------------------------------------------------------------------------------
%{\it The proof of the theorem 3.}
\begin{proof}[The proof of  theorem \ref{riem}]
Since for  $(j,k)>1$
$$
\frac{|S_n^{(k)}|}{n!}{\bf M_n}\alpha_j
\ll
\left[ F(z)
\sum_{s\geqslant 1} \frac{z^{sj}}{j^2}\frac{1}{s^4}     \right]_{(n)},
$$
then
$$
{\bf M_n}\alpha_j\ll\frac{1}{j^2} +\frac{1}{n^{\gamma_0+1}}
$$
for $(j,k)>1$.
For $(k,j)=1$ we have
$
{\bf M_n} \alpha_j=\frac{c_{n-j}}{jc_{n}}.
$
Hence we obtain
$$
{\bf M_n}P_n(\alpha (\sigma))=\sum_{j\leqslant n}\log j{\bf M_n}\alpha_j=
\sum_{j\leqslant n, (j,k)=1}\frac{c_{n-j}\log j}{jc_{n}}+O(1).
$$
Applying here the estimate
 $c_m=cm^{\gamma_0-1}\left( 1+
O(m^{-\epsilon})  \right)$ we obtain
\begin{equation*}
\begin{split}
{\bf M_n}\log P_n(\alpha (\sigma))&=\sum_{\scriptstyle j\leqslant n-1 \atop \scriptstyle (j,k)=1}\frac{\log j}{j}
\left(1-\frac{j}{n} \right)^{\gamma_0-1} +O(1)
\\
&=\sum_{\scriptstyle j\leqslant n-1 \atop \scriptstyle (j,k)=1}\frac{\log j}{j}
+
\log n\sum_{\scriptstyle j\leqslant n-1 \atop \scriptstyle (j,k)=1}\frac{1}{j}
\left( \left(1-\frac{j}{n} \right)^{\gamma_0-1}-1 \right)
\\
&\quad-
 \frac{1}{n}\sum_{\scriptstyle j\leqslant n-1 \atop \scriptstyle (j,k)=1}\frac{\log \frac{n}{j}}{j/n}
\left( \left(1-\frac{j}{n} \right)^{\gamma_0-1}-1 \right)+O(1)
\\
&=\frac{\gamma_0}{2}\log^2 n +C_0\log n +O(1),
\end{split}
\end{equation*}
where
$C_0=\gamma_0\int_0^1\frac{(1-y)^{\gamma_0-1}-1}{y}dy$.

Application of lemma \ref{meanalph} gives
$$
\mu_n={\bf M_n}(\log P_n(\alpha)-\log O(\alpha))=
\sum_{\scriptstyle m\leqslant \log^2 n \atop \scriptstyle (m,k)=1}\Lambda(m)
{\bf M_n}(D_{n,m}-1)^+ +O(1).
$$
From lemmas \ref{nud} and \ref{meanalph}  we have
$$
{\bf M_n}(D_{n,m}-1)^+=\nu_n(D_{n,m}=0) +{\bf M_n}D_{n,m}-1=\lambda\left(\frac{m}{x} \right)+O\left(
\frac{\log m}{m} \right),
$$
for $m\leqslant \log^2n$ and $(m,k)=1$
where $\lambda(y)=e^{\frac{1}{y}}-1+\frac{1}{y}$ and $x=\gamma_0\log n$.  Therefore
$$
\mu_n=\sum_{\scriptstyle m=1 \atop \scriptstyle (m,k)=1}^{\infty}\Lambda(m)
\lambda\left(\frac{m}{x} \right) +O((\log \log n)^2)
=S(x)-K(x) +O((\log \log n)^2),
$$
where
$$
S(x)=\sum_{ m=1 }^{\infty}\Lambda(m)
\lambda\left(\frac{m}{x} \right)\quad \hbox{\it and}\quad
K(x)=\sum_{\scriptstyle m=1 \atop \scriptstyle (m,k)>1 }^{\infty}\Lambda(m)
\lambda\left(\frac{m}{x} \right).
$$
Since
$$
\int_0^\infty\lambda(x)x^{s-1}dx=\Gamma(-s),
$$
for $1<\Re s <2$, then applying Mellin's  inversion  formula we
have
\begin{multline*}
K(x)=\sum_{\scriptstyle m=1 \atop \scriptstyle (m,k)>1 }^{\infty}\Lambda(m)
\frac{1}{2\pi i}\int_{\sigma -i\infty}^{\sigma +i\infty}
\Gamma(-s)\left( \frac{m}{x} \right)^{-s}ds
\\
=
\frac{1}{2\pi i}\int_{\sigma -i\infty}^{\sigma +i\infty}
\biggl( \sum_{p|k}\frac{\log p}{p^s-1} \biggl)\Gamma(-s)x^{s}ds,
\end{multline*}
for $1<\sigma<2$. Changing the integration contour in the above
integral from the line
 $\Re s=\sigma$ to $\Re s =-\frac{1}{2}$ and calculating the residuals, we obtain
$$
K(x)=x\left( \sum_{p|k}\frac{\log p}{p-1} \right)+O(\log x).
$$
By means of similar considerations it has been proved in [10] that
$$
S(x)=x(\log x -1)+{\frac{\zeta'(0)}{\zeta(0)}}
-\sum_{\rho}\Gamma(-\rho)x^{\rho}+{\rm O}\left( {\frac{1}{\sqrt x}}\right) ,
$$
where $\sum_{\rho}$ denotes the sum over the non-trivial zeroes of the Rienmann Zeta function. Therefore
\begin{equation*}
\begin{split}
{\bf M_n }\log O(\alpha)&=
\frac{\gamma_0}{2}\log^2 n +C_0\log n-S(\gamma_0\log n)+K(\gamma_0\log n)
+O((\log \log n)^2)
\\
&=\frac{\gamma_0}{2}\log^2 n +C_0\log n- \gamma_0\log n(\log (\gamma_0\log n) -1)
\\
&\quad+\sum_{\rho}\Gamma(-\rho)(\gamma_0\log n)^{\rho}
+
C_1\gamma_0\log n+O((\log \log n)^2),
\end{split}
\end{equation*}
here
$C_1=\sum_{p|k}\frac{\log p}{p-1}$.

Theorem \ref{riem} is proved.
\end{proof}

%--------------------------------------------------------------------------------------------------
%---------------------------Lemma 10--------------------------------------------------------------
%---------------------------------------------------------------------------------------------------
%{\bf Lemma 10.} {\it
\begin{lem}
\label{covrl}
 For $1\leqslant r,l,d\leqslant n^\epsilon$ and $(r,l)=1$ we have
$$
{\bf cov}(D_{n,r}',D_{n,r}')\ll\frac{\log n}{rl},\leqno1)
$$

$$
{\bf M_n}D_{n,r}'D_{n,l}''\ll \frac{\log n}{rl^2},\leqno2)
$$
$$
{\bf M_n}D_{n,r}''D_{n,l}''\ll \frac{1}{r^2l^2},\leqno3)
$$
$$
{\bf D_n}D_{n,d}\ll \frac{\log n}{d}.\leqno4)
$$
Here $\epsilon$ is a sufficiently small fixed number.
\end{lem}%}

%{\it Proof. }
\begin{proof}
 During the proof of the assertion 1) we will assume that $(rl,k)=1$,
since otherwise whether $D_{n,r}'=0$, or $D_{n,l}'=0$. Then for $j_1+j_2\leqslant n$,  $j_1\not=j_2$ and $(j_1j_2,k)=1$ we have
$$
{\bf M_n}\alpha_{j_1}\alpha_{j_2}=\frac{c_{n-j_1-j_2}}{c_nj_1j_2}.
$$
Therefore
\begin{equation*}
\begin{split}
{\bf cov}(D_{n,r}',D_{n,r}')&=
{\bf M_n}D_{n,r}'D_{n,l}'-{\bf M_n}D_{n,r}'{\bf M_n}D_{n,l}'
\\
&=
\sum_{\scriptstyle j_1\not=j_2, j_1,j_2\leqslant n/4\atop \scriptstyle r|j_1,l|j_2,(j_1j_2,k)=1}
\frac{c_nc_{n-j_1-j_2}-c_{n-j_1}c_{n-j_2}}{c^2_nj_1j_2}
\\
&\quad+\sum_{{j_1\not=j_2,\scriptstyle j_2+j_2\leqslant n}\atop
{\scriptstyle  j_1 \geqslant n/4 \vee  j_2 \geqslant n/4 \atop {\scriptstyle r|j_1,l|j_2,(j_1j_2,k)=1}}}
\frac{c_{n-j_1-j_2}}{c_nj_1j_2}
-\sum_{{\scriptstyle j_1\not=j_2, j_2,j_2\leqslant n}\atop
{\scriptstyle  j_1 \geqslant n/4 \vee  j_2 \geqslant n/4 \atop {\scriptstyle r|j_1,l|j_2,(j_1j_2,k)=1}}}
\frac{c_{n-j_1}c_{n-j_2}}{c^2_nj_1j_2}
\\
&\quad +\sum_{{\scriptstyle j\leqslant n :rl|j}\atop{\scriptstyle (j,k)=1}}{\bf D_n}\alpha_j=S_1+S_2+S_3+S_4.
\end{split}
\end{equation*}
Since
$c_m=\frac{|S_m^{(k)}|}{m!}=cm^{\gamma_0-1}\left( 1+
O(m^{-\epsilon})  \right)$ then one can easily obtain that
$$
S_1\ll
\sum_{\scriptstyle  j_1,j_2\leqslant n/4\atop \scriptstyle r|j_1,l|j_2}
\frac{\left|\left( 1-\frac{j_1+j_2}{n}\right)^{\gamma_0-1} -\left( 1-\frac{j_1}{n}\right)^{\gamma_0-1}\left( 1-\frac{j_2}{n}\right)^{\gamma_0-1}\right|}{j_1j_2}
+\frac{n^{-\epsilon}\log^2 n}{rl}\ll
\frac{1}{rl}.
$$
As for any $d\geqslant 1$ and $m\geqslant 1$  we have
$$
\sum_{j\leqslant m: d|j} c_{m-j}\ll 1+\frac{m^{\gamma_0}}{d},
$$ therefore
\begin{multline*}
S_2\ll \frac{1}{n} \sum_{\scriptstyle j_1+j_2\leqslant n \atop
\scriptstyle r|j_1,l|j_2}
\frac{c_{n-j_1-j_2}}{c_n}\left(\frac{1}{j_1}+\frac{1}{j_2} \right)
= \frac{1}{nc_n} \sum_{\scriptstyle j_1\leqslant n \atop
\scriptstyle r|j_1}\frac{1}{j_1}\sum_{\scriptstyle j_2\leqslant n-j_1\atop \scriptstyle l|j_2}
{c_{n-j_1-j_2}}
\\
+
\frac{1}{nc_n} \sum_{\scriptstyle j_2\leqslant n \atop
\scriptstyle l|j_2}\frac{1}{j_2}\sum_{\scriptstyle j_1\leqslant n-j_2\atop \scriptstyle r|j_1}
{c_{n-j_2-j_1}}
\ll\frac{\log n}{n^{\gamma_0}r}+\frac{\log n}{n^{\gamma_0}l}+\frac{\log n}{rl}\ll\frac{\log n}{rl}
\end{multline*}
for $r,l \leqslant n^{\gamma_0}$.

In a similar way we obtain that
for $r,l \leqslant n^{\gamma_0}$
$$
S_3\ll \frac{1}{n} \sum_{\scriptstyle j_1,j_2\leqslant n \atop
\scriptstyle r|j_1,l|j_2}
\frac{c_{n-j_1}c_{n-j_2}}{c^2_n}\left(\frac{1}{j_1}+\frac{1}{j_2} \right) \ll\frac{\log n}{rl}.
$$
Since for  $(j,k)=1$ we have
${\bf D_n}\alpha_j\leqslant {\bf M_n}\alpha_j^2=\frac{1}{j}\frac{c_{n-j}}{c_n}+\frac{1}{j^2}\frac{c_{n-2j}}{c_n}$ (here we assume that $c_m=0$ for $m<0$),
then
$$
S_4\ll \frac{\log n}{rl}+\frac{1}{n^{\gamma_0}}\ll \frac{\log n}{rl}
$$
for $r,l\leqslant n^{\gamma_0/2}$.

Assertion 1) is proved.

Applying the same considerations we obtain
\begin{equation*}
\begin{split}
c_n{\bf M_n}D_{n,r}'D_{n,l}''&\ll
\left[F(z) \sum_{r|j}\frac{z^j}{j}\sum_{l|m}\frac{z^m}{m^2} \right]_{(n)}=
\sum_{\scriptstyle j+m\leqslant n \atop \scriptstyle r|j, l|m}
\frac{c_{n-j-m}}{jm^2}
\\
&\leqslant
\sum_{\scriptstyle j,m\leqslant n/4 \atop \scriptstyle r|j, l|m}
\frac{c_{n-j-m}}{jm^2}
+\sum_{\scriptstyle j+m\leqslant n,j\geqslant n/4 \atop \scriptstyle r|j, l|m}
\frac{c_{n-j-m}}{jm^2}
\\
&\quad+
\sum_{\scriptstyle j+m\leqslant n,m\geqslant n/4 \atop \scriptstyle r|j, l|m}
\frac{c_{n-j-m}}{jm^2}\ll c_n\frac{\log n}{r l^2}
\end{split}
\end{equation*}
for $r,l \leqslant n^{\gamma_0}$.

In a similar way we prove 3).

One can easily see that
$$
{\bf D_n}D_{n,d}\leqslant 2 {\bf D_n}D_{n,d}'+2{\bf D_n}D_{n,d}''.
$$
Slightly changing the proof of assertion 1), one can check that ${\bf D_n}D_{n,d}'\ll\frac{\log n}{d}$ for $d\leqslant n^{\gamma_0}$. In a similar way  ${\bf D_n}D_{n,d}''\leqslant {\bf M_n}D_{n,d}^2\ll\frac{1}{d^2}$.
Hence follows the proof of the estimate 4).

The lemma is proved.
\end{proof}

%-----------------------------------------------------------------------------------------------------------------
%----------------------------------------------Lemma 11----------------------------------------------------
%------------------------------------------------------------------------------------------------------------------
%{\bf Lemma 11.} {\it
\begin{lem}
\label{cov2}
If $(r,l)=1$,
$C<l\hbox{\ and \ }r\leqslant \log ^{D}n$,
then
$$
{\bf cov}(I[D_{n,l}=0],I[D_{n,r}=0])\ll\frac{\log n}{lr}.
$$%}
\end{lem}
%{\it Proof.}
\begin{proof}
 Let us suppose at first that $(k,l)=1$, $(k,r)=1$, then
\begin{multline*}
{\bf cov}(I[D_{n,l}=0],I[D_{n,r}=0])
\\=\nu_n(D_{n,l}=0,D_{n,r}=0)-\nu_n(D_{n,l}=0)\nu_n(D_{n,r}=0).
\end{multline*}
Inserting here the estimates of the lemmas \ref{nud} and \ref{nulr}
and applying the lemma \ref{fxy} to function
$f(x)=\frac{\Gamma(\gamma_0)}{\Gamma(\gamma_0(1+x))}$ with  $x=1/l$ and $y=1/r$, we obtain the desired estimate.

Suppose now $(k,r)>1$. Then $D_{n,r}=D_{n,r}''$.

Since $1-I[D_{n,d}=0]=I[D_{n,d}>0]\leqslant D_{n,d}$, therefore, applying the lemmas 10 and 3, we obtain
\begin{equation*}
\begin{split}
|{\bf cov}(I[D_{n,l}=0],I[D_{n,r}=0])|&=|{\bf cov}(I[D_{n,l}>0],I[D_{n,r}>0])|
\\
&\leqslant {\bf M_n}D_{n,l}D_{n,r}''
+ {\bf M_n}D_{n,l}{\bf M_n}D_{n,r}''
\\
&={\bf M_n}D_{n,l}'D_{n,r}'' + {\bf M_n}D_{n,l}'{\bf M_n}D_{n,r}''+{\bf M_n}D_{n,l}''D_{n,r}''
\\
&\quad + {\bf M_n}D_{n,l}''{\bf M_n}D_{n,r}''
\ll \frac{\log n}{lr^2},
\end{split}
\end{equation*}
for $r,l\leqslant \log^Dn$.

The lemma is proved.
\end{proof}

%--------------------------------------------------------------------------------------------------------------------
%--------------------------------------Lemma 12---------------------------------------------------------------
%---------------------------------------------------------------------------------------------------------------------
\begin{lem}
\label{vard}
%{\bf Lemma 12.} {\it
If $1\leqslant d\leqslant n^\epsilon$, then
$$
{\bf D_n}(D_{n,d}-1)^+\leqslant {\bf D_n}I[D_{n,d}=0]+{\bf D_n}D_{n,d}  \ll \frac{\log n}{d}.
$$
\end{lem}%}
%{\it  Proof.}
\begin{proof}
 Since $(D_{n,d}-1)^+=I[D_{n,d}=0]+D_{n,d}-1$, then
\begin{equation*}
\begin{split}
{\bf D_n}(D_{n,d}-1)^+&={\bf D_n}I[D_{n,d}=0]+{\bf D_n}D_{n,d}+2{\bf cov}(I[D_{n,d}=0],D_{n,d})
\\
&\leqslant {\bf D_n}I[D_{n,d}=0]+{\bf D_n}D_{n,d},
\end{split}
\end{equation*}
as ${\bf cov}(I[D_{n,d}=0],D_{n,d})=-{\bf M_n}I[D_{n,d}=0]{\bf M_n}D_{n,d}\leqslant 0$.

Applying the inequality $\nu_n(D_{n,d}>0)\leqslant {\bf
M_n}D_{n,d}$ we obtain
\begin{equation*}
\begin{split}
{\bf D_n}(I[D_{n,d}=0])&=\nu_n(D_{n,d}=0)-\nu_n(D_{n,d}=0)^2
\\
&\leqslant \nu_n(D_{n,d}>0)
\leqslant {\bf M_n}D_{n,d}\ll \frac{\log n}{d};
\end{split}
\end{equation*}
the last inequality follows from the lemma \ref{md}.
Hence and from the estimate 4) of lemma \ref{covrl} we obtain the proof of the lemma.
\end{proof}
%\smallskip
\begin{prop}
\label{propbarbour}
%{\bf Proposition 1}. {\it
For any fixed $K>0$ we have

$$
{\nu_n}\left(
{  \frac{ |\log P_n(\alpha)-\log O_n(\alpha)-\mu_n|}{\log^{3/2} n} } >
K\left( {{\log \log n}\over \log n}\right) ^{2/3}
\right)\ll\left( {\frac{\log \log n}{\log n}}\right) ^{2/3},
$$
where $\mu_n={\bf M_n}\bigl(\log P_n(\alpha)-\log O_n(\alpha)\bigr)$.
\end{prop}
%}.
%\smallskip

%{\it Proof.}
\begin{proof}
The proof of this proposition uses the considerations of A. D. Barbour and S. Tavar\`e of the work \cite{barbour_tavare}
\begin{equation*}
\begin{split}
\log P_n(\alpha)-\log O_n(\alpha)&=
\sum_{m\leqslant C}\Lambda(m)(D_{n,m}(\alpha)-1)^+
\\
&\quad+\sum_{C<p^s\leqslant \log^2 n }(D_{n,p^s}(\alpha)-1)^+\log p
\\
&\quad+\sum_{m>\log^2n} \Lambda(m)(D_{n,m}(\alpha)-1)^+
=V_1+V_2+V_3.
\end{split}
\end{equation*}
Applying lemma \ref{meanalph} we obtain
$$
{\bf M_n}V_3 \ll 1.
$$
Applying the recently proved lemma we have
\begin{multline*}
{\bf D_n}V_1\leqslant \log^2C{\bf D_n}\biggl(\sum_{m\leqslant C}(D_{n,m}(\alpha)-1)^+\biggr)
\\
\leqslant C\log^2C
\sum_{m\leqslant C}
{\bf D_n}(D_{n,m}(\alpha)-1)^+
\leqslant C\log^3C\log n.
\end{multline*}
Since $(D_{n,m}(\alpha)-1)^+=I[D_{n,m}(\alpha)=0]+D_{n,m}(\alpha)-1$ then denoting
$V_2^{(1)}=\sum_{C<p^s\leqslant \log^2 n }I[D_{n,p^s}=0]\log p$ and
$V_2^{(2)}=\sum_{C<p^s\leqslant \log^2 n }D_{n,p^s}\log p$ we have
\begin{equation*}
\begin{split}
{\bf D_n}V_2^{(1)}&=\sum_{C<p^sq^l\leqslant \log^2 n }
{\bf cov}(I[D_{n,p^s}=0],I[D_{n,q^l}=0])\log p \log q
\\
&\ll\sum_{\scriptstyle C<p^sq^l\leqslant \log^2 n \atop \scriptstyle p\not=q }
\frac{\log n}{p^sq^l}\log p \log q+
\sum_{\scriptstyle C<p^s\leqslant \log^2 n }\log^2 p{\bf D_n}I[D_{n,q^l}=0]
\\
&\quad+\sum_{\scriptstyle C<p\leqslant \log^2 n }
\sum_{s>l\geqslant 1}({\bf D_n}I[D_{n,p^s}=0])^{1/2}({\bf D_n}I[D_{n,p^l}=0])^{1/2}\log^2 p
\\
&\ll \log n (\log \log n)^2+
\log n\sum_{\scriptstyle C<p\leqslant \log^2 n }
\sum_{s>l\geqslant 1}\frac{\log^2 p}{p^{\frac{s+l}{2}}}
\ll\log n (\log \log n)^2.
\end{split}
\end{equation*}
Applying the same calculations we have
$$
{\bf D_n}V_2^{(2)}\ll\log n (\log \log n)^2.
$$
Therefore
$$
{\bf D_n}V_2\leqslant 2{\bf D_n}V_2^{(1)}+2{\bf D}V_2^{(2)}\ll\log n (\log \log n)^2.
$$
Applying the Chebyshev inequality we have
$$
{\nu_n}\left( {\frac{|V_1-{\bf M_n}V_1|}{\log^{3/2} n}}>{\frac{1}{3}}K
\left( {\frac{\log \log n}{\log n}}\right) ^{2/3} \right)
\ll {\frac{1}{ {(\log \log n)}^{4/3} \log ^{2/3}n } },
$$

$$
{\nu_n}\left( {\frac{|V_2-{\bf M_n}V_2|}{\log^{3/2} n}}>{\frac{1}{3}}K
\left( {\frac{\log \log n}{\log n}}\right) ^{2/3} \right)
\ll \left( {\frac{\log \log n}{\log n}}\right) ^{2/3},
$$

$$
{\nu_n}\left( {\frac{|V_3-{\bf M_n}V_3|}{\log^{3/2} n}}>{\frac{1}{3}}K
\left( {\frac{\log \log n}{\log n}}\right) ^{2/3} \right)
\ll {\frac{1}{ {(\log \log n)}^{2/3} \log ^{5/6}n} }.
$$
Hence follows the proof of the proposition.
\end{proof}
 %\smallskip

%\smallskip

%\smallskip

% The following lemma is a generalization of lemma 2.5 of the work \cite{barbour_tavare}, it has already been used in \cite{zakpoly} in the proof of the appropriate result.

%{\bf Lemma 17.} {\it
% \begin{lem}
% \label{cmp_x_and_u}
% Let $U$ and $X$ be arbitrary random variables such that
% $\sup_{x\in R}|\nu_n(U<x)-G(x)|\leqslant \eta$, where $G(x)$ differentiable function such that $|G'(x)|\leqslant C$. Then for any $\epsilon >0$ we have
%
% $$
% \sup_{x\in R}|\nu_n(U+X<x)-G(x)|\ll\eta +\epsilon + \nu_n(|X|>\epsilon),
% $$
% the constant in symbol $\ll$ depends on $C$ only. %}
%\end{lem}
%\smallskip
%{\it Proof of theorem 3.}
\begin{proof}[Proof of theorem \ref{clt}]
Putting in lemma \ref{cmp_x_and_u}
$
U=\frac{\log  P_n(\sigma)-{\bf M_n}\log P_n(\sigma)}{\sqrt{\frac{\phi(k)}{3k}}\log^{3/2}n   }
$
and
$
X=\frac{\log  P_n(\sigma)-\log O_n(\sigma)-\mu_n}{\sqrt{\frac{\phi(k)}{3k}}\log^{3/2}n   }
$
and applying the proposition 1, we obtain the proof of the theorem.
\end{proof}
\chapter*{Conclusions}
%\backmatter
The asymptotic expansions proved in Chapters 2 and 3
show that the best possible estimate for the convergence rate of the
distribution function  of the logarithm of the order of a random
permutation on $S_n$ and
 $S_n^{(k)}$ is $\frac{1}{\sqrt{\log n}}$.

The similar conclusion is true for the distribution of the logarithm of the degree of the splitting field of a random polynomial.
\newpage
%\backmatter
\thispagestyle{plain}

\tableofcontents

\end{document}